\documentclass{amsart}
\usepackage[hyphens]{url}
\usepackage{mathtools, microtype, mathrsfs, xfrac, hyperref, amssymb, enumerate, xspace, stmaryrd}
\usepackage{bbm}
\usepackage[utf8]{inputenc}
\usepackage{tikz-cd}
\usepackage{pdflscape}
\usepackage{bm}

\usepackage{todonotes}

\numberwithin{equation}{section}

\numberwithin{subsection}{section}

\allowdisplaybreaks[1]

\AtBeginDocument{%
   \def\MR#1{}
}

\newtheorem*{namedtheorem}{\theoremname}
\newcommand{\theoremname}{testing}

\newtheorem{theorem}{Theorem}[section]
\newtheorem{proposition}[theorem]{Proposition}
\newtheorem{corollary}[theorem]{Corollary}
\newtheorem{lemma}[theorem]{Lemma}
\newtheorem*{theorem*}{Theorem}

\theoremstyle{definition}
\newtheorem{definition}[theorem]{Definition}
\newtheorem{hypothesis}[theorem]{Hypothesis}

\newtheorem{example}[theorem]{Example}

\newtheorem{remark}[theorem]{Remark}

\newtheorem*{question*}{Question}

\theoremstyle{remark}

\renewcommand{\mathcal}{\mathscr}

\newcommand\cA{\mathcal{A}} \newcommand\cB{\mathcal{B}}
\newcommand\cC{\mathcal{C}} 
 
\newcommand\cG{\mathcal{G}} \newcommand\cH{\mathcal{H}}
 
 \newcommand\cL{\mathcal{L}}
 \newcommand\cN{\mathcal{N}}
\newcommand\cO{\mathcal{O}} \newcommand\cP{\mathcal{P}}
 
\newcommand\cS{\mathcal{S}} 
 \newcommand\cV{\mathcal{V}}
 \newcommand\cX{\mathcal{X}}
\newcommand\cY{\mathcal{Y}} 

\renewcommand\AA{\mathbb{A}} 
\newcommand\CC{\mathbb{C}} 
 \newcommand\FF{\mathbb{F}}
\newcommand\GG{\mathbb{G}}

 \newcommand\PP{\mathbb{P}}
\newcommand\QQ{\mathbb{Q}} \newcommand\RR{\mathbb{R}}

 \newcommand\ZZ{\mathbb{Z}}

 \newcommand\bB{\mathbf{B}}

 \newcommand\bL{\mathbf{L}}
 
 \newcommand\bP{\mathbf{P}}

 \newcommand\bX{\mathbf{X}}

\newcommand\rC{\mathrm{C}}

\newcommand\rI{\mathrm{I}} 
 
 \newcommand\rN{\mathrm{N}}
 \newcommand\rP{\mathrm{P}}
 \newcommand\rR{\mathrm{R}}

\newcommand\fG{\mathfrak{G}} 
\newcommand\fH{\mathfrak{H}}

\newcommand\fM{\mathfrak{M}} 
\newcommand\fN{\mathfrak{N}}

\newcommand\fQ{\mathfrak{Q}}

\newcommand\arr{\ifinner\to\else\longrightarrow\fi}

\newcommand\arrto{\ifinner\mapsto\else\longmapsto\fi}

\renewcommand\H{\operatorname{H}}

\def\displaytimes_#1{\mathrel{\mathop{\times}\limits_{#1}}}

\def\displayotimes_#1{\mathrel{\mathop{\bigotimes}\limits_{#1}}}

\renewcommand\hom{\operatorname{Hom}}

\newcommand\Id{\operatorname{Id}}

\newcommand\aut{\operatorname{Aut}}

\newcommand\spec{\operatorname{Spec}}

\newcommand{\underisom}{\mathop{\underline{\mathrm{Isom}}}\nolimits}

\newcommand{\underaut}{\mathop{\underline{\mathrm{Aut}}}\nolimits}

\newlength{\ignora}

\newcommand{\GL}{\mathrm{GL}}
\newcommand{\SL}{\mathrm{SL}}
\newcommand{\PGL}{\mathrm{PGL}}

\newcommand{\br}{\operatorname{Br}}
\newcommand{\gal}{\operatorname{Gal}}

\DeclareFontFamily{U}{mathx}{\hyphenchar\font45}
\DeclareFontShape{U}{mathx}{m}{n}{
      <5> <6> <7> <8> <9> <10>
      <10.95> <12> <14.4> <17.28> <20.74> <24.88>
      mathx10
      }{}
\DeclareSymbolFont{mathx}{U}{mathx}{m}{n}
\DeclareFontSubstitution{U}{mathx}{m}{n}
\DeclareMathAccent{\widecheck}{0}{mathx}{"71}
\DeclareMathAccent{\wideparen}{0}{mathx}{"75}

\renewcommand{\epsilon}{\varepsilon}

\newcommand{\cha}{\operatorname{char}}

\newcommand{\diag}{\operatorname{diag}}
\newcommand{\fKl}{\mathfrak{Kl}}
\newcommand{\tr}{\operatorname{Tr}}

\author{Giulio Bresciani}
\address[G. Bresciani]{Dipartimento di Matematica, Università di Pisa, Italy}
\email{giulio.bresciani1@unipi.it}

\author{Tianzhi Yang}
\address[T. Yang]{Scuola Normale Superiore di Pisa, Italy}
\email{tianzhi.yang@sns.it}

\begin{document}

\title[Neutral representations in dim. $\le 3$ and fields of moduli]{Neutral representations in dimension $\le 3$ and fields of moduli}

\begin{abstract}
	A representation $V$ of an algebraic group $G$ induces a vector bundle $[V/G] \to BG$. The representation $V$ of $G$ is \emph{neutral} if, for every twisted form $\mathcal{V} \to \mathcal{G}$ of $[V/G] \to BG$ over a field $k$, we have $\mathcal{G}(k) \neq \emptyset$.
	
	Twisted forms of representations arise in many ways, for instance as cohomology of families of varieties on residual gerbes of moduli spaces, and from quotient singularities. Moreover, every Tannakian category is the category of vector bundles on some gerbe. Because of this, studying neutral representations yields numerous applications, especially to problems about fields of moduli.

	The present article has three main results. First, we completely classify neutral, faithful representations of finite groups in dimension $\le 3$. Second, we give a very general, computation-friendly result for proving that representations of finite abelian groups are neutral, in arbitrary dimensions. Third, we develop the abstract concept of the normalizer $\mathcal{G} \to \mathcal{N} \to \mathcal{H}$ of a morphism of gerbes $\mathcal{G} \to \mathcal{H}$ on an arbitrary site (twisted representations correspond to morphisms of gerbes $\mathcal{G} \to B\mathrm{GL}_{n}$), and show that the normalizer $\mathcal{N}$ only depends on the geometric type of $\mathcal{G} \to \mathcal{H}$.
\end{abstract}


\maketitle

\setcounter{tocdepth}{1}
\tableofcontents

\section{Introduction}

Fix a perfect field $k_{0}$ with algebraic closure $K$. Every field is implicitly assumed to be an extension of $k_{0}$.

Given a variety $X$ over $K$, possibly with some additional structure $\xi$ such as a marked point, a subvariety $Y \subset X$, a group structure on $X$, an endomorphism $X \to X$, and so forth. It is a natural question to ask on which subfields of $K$ the pair $(X, \xi)$ is defined. Clearly, the field generated by the coefficients of a fixed set of equations is sufficient, but perhaps we can find a better set of equations.

There is a natural candidate for a minimal field of definition, which is called the \emph{field of moduli} $k_{(X,\xi)}$, which is contained in every field of definition. It can be constructed either through Galois theory  or moduli theory. The question then becomes: can we define $(X,\xi)$ over its field of moduli $k_{(X,\xi)}$?

Among the first ones studying this question there are A. Weil \cite{weil}, T. Matsusaka \cite{matsusaka} and G. Shimura \cite{shimura-automorphic}, and over the years it has received considerable attention. Recently, the first author, in joint work with A. Vistoli, developed a new method for studying this problem based on the theory of \emph{gerbes} \cite{bresciani-vistoli}, which led to numerous new applications, such as the resolution of a conjecture by J. Doyle and J. Silverman \cite{bresciani-bound} and the proof of the fact that a smooth, complex plane curve $C$ of odd degree can be defined by a polynomial with real coefficients if and only if $C \simeq \bar{C}$ \cite{bresciani-plane} (no elementary proof is known, and there are counterexamples in even degree).

Given $(X,\xi)$ as above, if $\aut(X,\xi)$ is finite of order prime with $\cha k_{0}$ there is a gerbe $\cG_{(X,\xi)}$ over $k_{(X,\xi)}$ which classifies twisted forms of $(X,\xi)$ \cite{bresciani-vistoli}. In particular, $(X,\xi)$ is defined over $k_{(X,\xi)}$ if and only if $\cG_{(X,\xi)}$ is \emph{neutral}, which means that $\cG_{(X,\xi)}(k_{(X,\xi)}) \neq \emptyset$.

Because of this, it becomes important to know how to prove that a gerbe is neutral. In joint work with A. Vistoli \cite{bresciani-vistoli-yang}, we observed that the existence of certain vector bundles on a gerbe force it to be neutral. The main result of the present paper is the classification of vector bundles of rank $\le 3$ with this property: this automatically produces many new examples of varieties defined over the field of moduli.

\subsection{Twisted forms of representations}

Given an algebraic group $G$ over a field $k$ with a representation $V$, the quotient stack $[V/G]$ defines a vector bundle on the classifying stack $BG$. On the other hand, if $\cV \to BG$ is a vector bundle, the restriction $\cV_{\tau}$ to the tautological section $\tau : \spec k \to BG$ defines a representation of $G$, and the two constructions are inverses of each other.

A twist of the representation $V$ is a twisted form of the morphism $[V/G] \to BG$, i.e. a vector bundle $\cV \to \cG$ whose base change to $\bar{k}$ is isomorphic to $[V/G] \to BG$. Twisted forms of $BG$ are gerbes, hence twisted representations are essentially vector bundles on gerbes.

A representation $V$ of $G$ is \emph{neutral} if, for every twisted form $\cV \to \cG$, we have $\cG(k) \neq \emptyset$. Equivalently, $V$ is neutral if every twist of $V$ is a representation of a twist of $G$.

We are interested in studying which representations are neutral. Contrary to what one might expect, it turns out that the large majority of faithful representations are neutral, and non-neutral ones are the exception. 

In \cite{bresciani-vistoli-yang}, we introduced some techniques for studying this problem for very simple groups (finite and diagonalizable), with no restriction on the dimension of the representation. In the present paper, we go in an orthogonal direction: no restrictions on the group, but strict restrictions on the dimension. We fix a dimension $d \le 3$, and study every faithful representation of every finite group in dimension $d$. This forced us to find techniques which are more refined than the ones in \cite{bresciani-vistoli-yang}; nonetheless, the techniques work in arbitrary dimensions.

\subsection{Applications}\label{sect:applications}

Understanding which representations are neutral automatically leads to several applications. Here are some examples.

\begin{example}\label{example:cohomology}
	Let $(X, \xi)$ be as above, with $X$ proper. Assume that the automorphism group $\aut(X, \xi)$ of the pair $(X, \xi)$ is finite of order prime with $\cha k_{0}$. Let us show that, if the representation of $\aut(X, \xi)$ on $\H^{p}(X, \cO_{X})$ is neutral for some $p$, the pair $(X,\xi)$ is defined over the field of moduli.
	
	Twisted forms of the pair $(X,\xi)$ are parametrized by a morphism $f:\cX_{(X,\xi)} \to \cG_{(X,\xi)}$ where $\cG_{(X,\xi)}$ is a gerbe defined over the field of moduli $k_{(X,\xi)}$ \cite[Propositions 3.9, 3.10, and section 5.1]{bresciani-vistoli}. The higher direct images $R^{p}f_{*}\cO_{\cX}$ define twisted forms on $\cG_{(X,\xi)}$ of $\H^{p}(X, \cO_{X})$ as a representation of $\aut(X, \xi)$. Since this is neutral, we get that $\cG(k_{(X,\xi)}) \neq \emptyset$.
	
	For instance, in \cite{bresciani-vistoli-yang} this idea is used to prove the following: if $C$ is a smooth, proper algebraic curve with $|\aut(C)| = p$ prime, and $p$ does not divide the difference between the genus of $C$ and that of $C/\aut(C)$, then $C$ is defined over its field of moduli. This is done by proving that $\H^{1}(C, \cO_{C})$ is a neutral representation of $\aut(C)$.
\end{example}
	
\begin{example}\label{example:tangent}
	Let $(X,\xi)$ be as above, and assume that a marked point $x \in X(\bar{k})$ is part of the additional structure $\xi$, e.g. if $\xi$ is a group structure we can choose $x$ as the origin. In particular, $\aut(X, \xi)$ fixes $x$, and hence acts on its tangent space $T$. If the action of $\aut(X, \xi)$ on $T$ is neutral, the pair $(X, \xi)$ is defined over the field of moduli.
	
	In fact, since the marked point $x \in X(\bar{k})$ is part of the additional structure, it induces a section $\cG_{(X,\xi)} \to \cX_{(X,\xi)}$ of the universal family. The normal bundle $\cN$ of this section is a twisted form of $T$ as a representation of $\aut(X, \xi)$, hence $\cG_{(X,\xi)}(k_{(X,\xi)}) \neq \emptyset$ if the representation is neutral.
\end{example}

\begin{example}
	Let $S$ be a variety with quotient singularities over $k$. Every rational point $s \in S(k)$ defines a twisted form $\cV \to \cG$ of the natural representation of the local fundamental group of the singularity \cite[\S 6.4]{bresciani-vistoli}. 
	
	If the representation is neutral, then $s$ lifts to a rational point of a resolution of singularities \cite[Proposition 6.5]{bresciani-vistoli}. This has again consequences for field of moduli problems \cite[Theorem 5.4]{bresciani-vistoli}, and the Lang--Nishimura theorem holds for rational maps $S \dashrightarrow X$ to a proper scheme $X$ even though $s$ is not smooth.
\end{example}

More generally, arithmetic aspects of gerbes and stacks have recently been under the spotlight \cite{ellenberg-satriano-zureick, bragg-lieblich, loughran-santens, loughran-sankaran}.

We also mention a Tannakian interpretation of twisted representations. Given an affine gerbe $\cG$ over $k$, the tensor category $\operatorname{Rep}_{k}(\cG)$ of vector bundles on $\cG$ is a Tannakian category over $k$. On the other hand, a Tannakian category $\cC$ defines an affine gerbe $\operatorname{Fib}(\cC)$ of fibre functors. These constructions are inverses to each other \cite[1.12]{deligne}, and neutral Tannakian categories correspond to neutral gerbes.

Because of this, understanding which representations are not neutral is equivalent to understanding the objects of non-neutral Tannakian categories. Given the widespread use of Tannakian categories, this is a very basic question which has not been addressed much.

\subsection{Arithmetic of quotient singularities}

A point $s$ of a variety $S$ is a tame quotient singularity if, étale locally around $s$, the variety $S$ is a quotient of a smooth variety by a finite group of order prime with $\cha k_{0}$.

Following \cite{bresciani-vistoli}, a tame quotient singularity $(S,s)$ is of type $\rR$ if and only if, for every twisted form $(S',s')$ over a field $k$, the $k$-point $s'$ lifts to a resolution of singularities (these always exist for tame quotient singularities). The arithmetic of quotient singularities is the study of which quotient singularities are of type $\rR$. Again, this concept has applications to the study of fields of moduli.

In dimension $2$, tame quotient singularities of type $\rR$ are completely classified \cite{bresciani-sing}. 

If $(S,s)$ is a singularity with local fundamental group $G$ and associated representation $V$, then $V$ is neutral if and only if $(S,s)$ is of type $\rR$ \cite[Theorem 3.2]{bresciani-vistoli-yang}. One result of this paper is the classification of faithful, neutral representations in dimension $3$ and characteristic $0$: as a corollary, we get the classification of tame quotient singularities of type $\rR$ in dimension $3$ and characteristic $0$.

\subsection{Neutral representations in dimension $\le 3$}

The first result of this paper is a complete classification of faithful, neutral representations of finite groups in dimension $\le 3$ in characteristic $0$.

Let us introduce some notation. Write $\diag(a_{1}, \dots, a_{n})$ for the diagonal $n \times n$ matrix with eigenvalues $a_{1}, \dots, a_{n}$. 

Given $\underline{n} = (n_{1}, \dots, n_{s}) \in \ZZ^{s}_{>0}$, $\underline{m} = (m_{1}, \dots, m_{s}) \in \ZZ^{s}$ vectors of integers with $n_{i} > 0$ and $n = n_{1} + \dots + n_{s}$, define $\SL_{\underline{n}}^{\underline{m}}$ as the subgroup of $\GL_{n}$ of block matrices, with blocks $M_{1}, \dots, M_{s}$ of dimensions $n_{1}, \dots, n_{s}$, such that $\det M_{1}^{m_{1}} \cdot \dots \cdot M_{s}^{m_{s}} = 1$.

In dimension $1$, it is relatively easy (though by no means obvious) to show that every faithful representation of a finite group is neutral, see Proposition~\ref{prop:gl1}.

In dimension $2$, tame quotient singularities of type $\rR$ were classified in \cite{bresciani-sing}. This means that faithful representations associated with singularities, i.e. those without pseudo-reflections, were already classified. We complete the analysis by classifying all faithful representations.

\begin{theorem}\label{thm:main1}
	Let $G \subseteq \GL_{2}(K)$ be a finite group, $\bar{G}$ its image in $\PGL_{2}(K)$. If $\cha k_{0} = 2$, then $G$ is neutral. If $\cha k_{0} \neq 2$, the following are equivalent.
	\begin{enumerate}
		\item The representation of $G$ on $\AA^{2}$ is not neutral.
		\item $G$ is conjugate to
		\[\left<  
			\left( \begin{array}{cc}
			\zeta_{2m} 	& 	0  			\\
			0 			&	\zeta_{2m} 	\\
			\end{array} \right),~
			\left( \begin{array}{cc}
			\zeta_{2n}^{-1} 	& 	0  			\\
			0 					&	\zeta_{2n} 	\\
			\end{array} \right)
			\right>\]
			for some $m, n \ge 1$ prime with $\cha k_{0}$, which is the inverse image in $\SL_{2}^{(m)}$ of
			\[\left< \diag(1, \zeta_{n}) \right>.\]
	\end{enumerate}
\end{theorem}

In dimension $3$, the result is completely new. We assume $\cha k_{0} = 0$ in order to have a better grasp on the finite subgroups of $\GL_{3}(K)$, but the classification can certainly be given in positive characteristic, too.

\begin{theorem}\label{thm:main2}
	Assume $\cha k_{0} = 0$, and let $G \subseteq \GL_{3}(K)$ be a finite subgroup with image $\bar{G} \subseteq \PGL_{3}(K)$. The following are equivalent.
	\begin{enumerate}
		\item The representation of $G$ on $\AA^{3}$ is not neutral.
		\item $G$ is conjugate to one of the following groups.
			\begin{itemize}
				\item Given a positive integer $c$, the inverse image in $\SL_{3}^{(c)}$ of
				\[\left< \diag(\zeta_{a},1,1), \diag(1,\zeta_{a},1), \diag(\zeta_{an},\zeta_{an}^{d},1) \right> \subseteq \PGL_{3}(K).\]
				for some positive integers $a,d,n$ with $d^{2} - d + 1 \cong 0 \pmod{n}$.

				\item 
				
				Given integers $c_{1},c_{2}$ with $2c_{1} + c_{2} \neq 0$, the inverse image in $\SL_{(2,1)}^{(c_{1},c_{2})}$ of
				 \[\left< \diag(\zeta_{2m}, \zeta_{2m}, 1), \diag(\zeta_{2n}, \zeta_{2n}^{-1}, 1)\right> \subseteq \PGL_{3}(K)\]
				for some positive integers $m, n$.
				
				\item Given a positive integer $c$, the inverse image in $\SL_{3}^{(c)}$ of $H_{2}$.
				
				\item If $\zeta_{3} \notin k_{0}$, given a positive integer $c$, the inverse image in $\SL_{3}^{(c)}$ of $H_{3}$.
			\end{itemize}
	\end{enumerate}
\end{theorem}

Using Examples~\ref{example:cohomology} and \ref{example:tangent}, Theorems~\ref{thm:main1} and \ref{thm:main2} automatically give new examples of varieties defined over the field of moduli.

As can be seen from these theorems, non-neutral finite subgroups of $\GL_{n}(K)$ are a very small fraction of all finite subgroups, at least in low dimension.

This means that, if $\cG$ is a non-neutral gerbe, there are strong constraints on which vector bundles can appear on $\cG$. Equivalently, we get strong constraints on the objects of a non-neutral Tannakian category.

The classification of $2$-dimensional singularities of type $\rR$ was the key ingredient for studying fields of moduli of plane curves in \cite{bresciani-plane}. An immediate consequence of Theorem~\ref{thm:main2} is the classification of $3$-dimensional singularities of type $\rR$, which can be applied to study fields of moduli of curves and surfaces embedded in $\PP^{3}$, or more generally $3$-dimensional varieties.

\subsection{Diagonal subgroups with $1$-dimensional eigenspaces}

When studying these problems, a counter-intuitive phenomenon emerges. The more complex a subgroup $G \subseteq \GL_{n}(K)$, the easier it is to prove neutrality. For instance, notice that almost all non-neutral subgroups in dimension $n \le 3$ have abelian image in $\PGL_{n}$, even though non-abelian finite subgroups of $\PGL_{n}$ are abundant for $n = 2, 3$. The reason is that the complexity of the group can actually be leveraged to apply the methods introduced in \cite{bresciani-vistoli}.

Because of this, it is very important to know how to address ``simpler'' groups.

We have found a way to tackle finite, diagonal subgroups of $\GL_{n}(K)$ with $1$-dimensional eigenspaces, different from the one given in \cite{bresciani-vistoli-yang} (which was fine for most diagonal subgroups of $\GL_{3}$, but not all of them). In particular, this new method turned out to be strong enough to address all such subgroups of $\GL_{n}(K)$ for $n \le 3$, but it gives results in arbitrary dimensions.

\begin{theorem}\label{thm:main3}
	Let $G \subseteq \GL_{n}(K)$ be a finite, diagonal subgroup with $1$-dimensional eigenspaces. Let the normalizer be $\GG_{m}^{n} \rtimes S$, where $S \subseteq S_{n}$ is a group of permutation matrices acting on the eigenspaces. 
	
	Consider the action of $S$ on the character group $X(\GG_{m}^{n}/G) \subseteq X(\GG_{m}^{n}) = \ZZ^{n}$. If $X(\GG_{m}^{n}/G)$ is a direct summand of a permutation $S$-module, then $G$ is neutral.
\end{theorem}

We also have a more flexible version of this result, where the hypothesis is only required for a subgroup of $S$ and a submodule of $X(\GG_{m}^{n}/G)$, see Theorem~\ref{thm:diag-neutral+}.

\subsection{The normalizer of a morphism of gerbes}

Finally, we have a new, general result about gerbes on sites, which is extremely useful for studying twisted representations and more generally twisted forms of ``things''. Given that this part is very different in tone from the rest of the paper, it is contained in the appendix.

We say that two faithful morphisms of gerbes $\cG \to \cH$, $\cG' \to \cH$ on a site $\cC$ with a terminal object $T$ are locally equivalent if there exists a covering $\{U_{i} \to T\}$ and isomorphisms $\cG_{U_{i}} \to \cG'_{U_{i}}$ making the obvious diagram $2$-commutative. There is a more general definition which does not require a terminal object, see Definition~\ref{def:loc-eq}. We may think of two locally equivalent morphisms as twisted forms of each other.

\begin{theorem}\label{thm:main4}
	Let $f : \cG \to \cH$ be a faithful morphism of gerbes on a small site $\cC$. There exists a gerbe $\cN$ and a factorization
	\[\cG \xrightarrow{\phi} \cN \xrightarrow{\psi} \cH\]
	of $f$ called the \emph{normalizer} of $f$, having the following properties.
	
	\begin{enumerate}
		\item	$\phi$, $\psi$ are faithful.
		\item	For every $S \in \cC$ and every object $s: S \to \cG$, the group sheaf $\underaut_{\cN}(\phi \circ s)$ is the normalizer of $\underaut_{\cG}(s)$ in $\underaut_{\cH}(\psi \circ \phi \circ s)$.
		\item	The normalizer is unique in the following sense. If $\cG \to \cN' \to \cH$ is another factorization having properties $(1)$ and $(2)$, there exists an isomorphism $\cN \to \cN'$ making the diagram
	\[\begin{tikzcd}[row sep = tiny]
								&	\cN \ar[dr] \ar[dd]		&				\\
		\cG \ar[ur]\ar[dr]		&							&	\cH		\\
								&	\cN' \ar[ur]
	\end{tikzcd}\]
	$2$-commutative.
		\item  If $\cG \to \cH$, $\cG' \to \cH$ are locally equivalent, the normalizers of $\cG \to \cH$ and $\cG' \to \cH$ coincide. More precisely, if $\cG \to \cN \to \cH$ is the normalizer of $\cG \to \cH$, there exists a faithful morphism $\cG' \to \cN$ such that the composition $\cG' \to \cN \to \cH$ is the normalizer of $\cG' \to \cH$.
	\end{enumerate}
\end{theorem}

In the case of twisted representations, we let $\cC$ be the fppf site on $k_{0}$, $\cH = B\GL_{n}$, $\cG' = BG$ where $G \subseteq \GL_{n}$ is a linear subgroup. Then $\cG \to B\GL_{n}$ defines a vector bundle $\cV \to \cG$ which is a twisted form of the standard representation of $G \subseteq \GL_{n}$ on $\AA^{n}$. The normalizer $BG \to \cN \to B\GL_{n}$ is $\cN = BN$, where $N$ is the normalizer of $G$ in $\GL_{n}$. 

The theorem then implies that every twisted form $\cG \to B\GL_{n}$ of $BG \to B\GL_{n}$ factorizes through $BN$. This is particularly helpful when $N$ is small. For instance, if $N = G$, the theorem implies that every twisted form of $BG \to B\GL_{n}$ is trivial. While $G = N$ cannot happen if $G$ is finite, we show that $N/G \simeq \GG_{m}$ is sufficient (see Corollary~\ref{cor:special}), and this happens for many finite subgroups of $\GL_{n}$.

The result goes well beyond twisted representations, though. If $H$ is an algebraic group, $G \subseteq H$ is a subgroup with normalizer $N$ and $N/G$ is special, we get that there are no non-trivial twisted forms of $BG \to BH$. For instance, there are no non-trivial twisted forms $\cG \to BS_{3}$ of $BC_{2} \to BS_{3}$.

In a similar vein, one can study twisted projective representations of a group, i.e. twisted forms of $BG \to B\PGL_{n}$ for a subgroup $G \subseteq \PGL_{n}$. This will give results about twisted forms of various kinds of objects we can define in a projective space, for instance hypersurfaces.

Finally, let us mention that our result is even more general than stated above, and it allows to construct gerbes corresponding to normalizers of arbitrary subsets of a group, see Theorem~\ref{thm:type-normalizer} and Example~\ref{example:normalizers-over-fields}. For instance, we can construct gerbes corresponding to centralizers, since a centralizer is the normalizer of a single element.

\subsection{Structure of the paper}

In section \ref{sect:previous}, we recall some previous results which are used throughout the paper. Most importantly, we recall the Lang--Nishimura theorem for stacks from \cite{bresciani-vistoli-valuative}.

Section \ref{sect:neutral} lays down the general theory of twisted representations, putting it in a more general context. Given an algebraic group $\Gamma$ over $k_{0}$ with a finite subgroup $G \subseteq \Gamma(K)$ of the geometric points, we study twisted forms of $BG \to B\Gamma_{K}$, i.e. morphisms $\cG \to B\Gamma$ with $\cG$ a gerbe. The main examples are $\Gamma = \GL_{n}$ and $\Gamma = \PGL_{n}$. Every twisted representation defines a twisted projective representation by passing to the projectivization of the vector bundle, we study the interactions between the two.

In section \ref{sect:gates}, we show that it is often possible to replace $\Gamma$ with a smaller, more manageable group. Loosely speaking, a subgroup $\Pi \subset \Gamma$ is called a gate for $G$ if every twisted form of $BG \to B\Gamma_{K}$ factorizes through $B\Pi$. It turns out that gates much smaller than $\Gamma$ can often be found, thus making things considerably simpler. We also use this concept to prove Theorem~\ref{thm:main3}.

In section \ref{sect:dim2}, we classify finite, neutral subgroups of $\PGL_{1}$, $\GL_{1}$, $\PGL_{2}$, $\GL_{2}$, $\PGL_{3}$. This is mostly a re-organization of previous works by the first author, though considerable additions are made for $\GL_{2}$ and $\PGL_{3}$.

In section \ref{sect:gl3}, we classify finite, neutral subgroups of $\GL_{3}$.

In the appendix, we prove Theorem~\ref{thm:main4}, and more generally give a framework for constructing sub-gerbes of a given gerbe on a site.

\subsection{Acknowledgements}

We would like to thank A. Vistoli for countless discussions about the topics of the paper.

\subsection{AI disclosure}

ChatGPT 5.6 was used to review the paper. Furthermore, Proposition~\ref{prop:q8} was obtained through back and forth collaboration between humans and ChatGPT. 

\subsection{Notations and conventions}

We work over a perfect field $k_{0}$ with algebraic closure $K$. Whenever we have a field $k$, it is equipped with a fixed embedding $k_{0} \subseteq k$. Typically, $k_{0}$ is either $\QQ$ or $\FF_{p}$, which means that the only role of $k_{0}$ is to fix the characteristic. However, results can change if $k_{0}$ is larger. In particular, some results in dimension $3$ change when $k_{0}$ contains a primitive $3$rd root of $1$.

For every integer $n \ge 1$ prime with $\cha k_{0}$, choose $\zeta_{n} \in K$ a primitive $n$-th root of unity with $\zeta_{mn}^{m} = \zeta_{n}$.

\section{Previous results}\label{sect:previous}

In this section, we summarize previous results from \cite{bresciani-vistoli-valuative, bresciani-plane}, which will be used extensively in the paper.

We start with the Lang--Nishimura theorem for stacks, which is probably the most important tool for studying these kinds of problems.

\begin{theorem}[{Lang--Nishimura for stacks \cite[Theorem 4.1]{bresciani-vistoli-valuative}}]\label{thm:LN}
	Let $X \dashrightarrow Y$ be a rational map of algebraic stacks over a scheme $S$, with $X$ locally noetherian and integral and $Y \to S$ tame and proper. If $s \in S$ lifts to a regular $k(s)$-point of $X$, the same holds for $Y$.
\end{theorem}

In characteristic $0$, the fact that $Y \to S$ is tame simply means that $Y$ is Deligne--Mumford with finite inertia.

We will use this theorem mainly in the case in which $Y$ is a (not necessarily proper) tame stack which is generically a scheme over a field $k$, $S$ is the coarse moduli space of $Y$, $X \to S$ is a resolution of singularities, and $X \to S \dashrightarrow Y$ is the birational inverse of the projection map $Y \to S$, which is proper even if $Y$ is not proper over $k$. In this case, the Lang--Nishimura theorem says that any rational point of $S$ which lifts to $X$ lifts to $Y$ as well.

\begin{definition}
	We say that a vector bundle $\cV$ over a finite gerbe $\cG$ is \emph{faithful} if the corresponding representations of the automorphism groups of geometric points of $\cG$ are faithful. Equivalently, $\cV$ is faithful if it is generically a scheme.
\end{definition}

The following lemma follows essentially from \cite{bresciani-vistoli-yang}, though it was not written down explicitly.

\begin{lemma}\label{lem:line-neut}
	Let $\cG$ be a gerbe with a faithful line bundle $\cL \to \cG$. Then $\cG$ is neutral.
\end{lemma}

\begin{proof}
	This is analogous to the proof of \cite[Lemma 7.1]{bresciani-vistoli-yang}. The faithful line bundle $\cL$ defines a faithful morphism $\cG \to B\GG_{m}$. In particular, we get that $\cG$ is abelian banded either by $\GG_{m}$ or by $\mu_{n}$ for some $n$. In the first case, $\cG \to B\GG_{m}$ is an isomorphism and $\cG$ is neutral. In the second case, $\cG$ corresponds to a class $[\cG] \in \H^{2}(k, \mu_{n})$ which maps to $0$ in $\H^{2}(k, \GG_{m})$. By Hilbert's theorem 90, $\H^{1}(k, \GG_{m})$ is trivial, hence $\H^{2}(k, \mu_{n}) \to \H^{2}(k, \GG_{m})$ is injective, $[\cG] = 0$ and $\cG(k) \neq \emptyset$.
\end{proof}

Recall that the index of a gerbe $\cG$ over $k$ is the greatest common divisor of the degrees of extensions $k'/k$ with $\cG(k') \neq \emptyset$.

\begin{lemma}[{\cite[Lemma A.3]{bresciani-plane}}]\label{lem:index-primes}
	Let $\cG$ be a finite étale gerbe over $k$, write $G$ for the automorphism group of a geometric point. The prime factors of the index of $\cG$ divide $|G|$.
\end{lemma}

\begin{lemma}[{\cite[Corollary A.6]{bresciani-plane}}]\label{lem:index-neutral}
	Let $\cG$ be a finite étale gerbe over $k$, write $G$ for the automorphism group of a geometric point. Assume either that
	\begin{itemize}
		\item the Sylow subgroups of $G$ are abelian, or
		\item $[G,G]$ is abelian and intersects trivially the center of $G$.
	\end{itemize}
	The index of $\cG$ is $1$ if and only if $\cG$ is neutral.
\end{lemma}

An \emph{affine gerbe} is a fpqc gerbe over a field $k$ with affine diagonal and possessing an affine chart.

Recall that a morphism of gerbes $f:\Gamma \to \Delta$ is locally full if, for every section $s:S \to \Gamma$, the induced homomorphism of group sheaves $\underaut_{\Gamma}(s) \to \underaut_{\Delta}(f \circ s)$ is sheaf-theoretically surjective \cite[Definition 3.4, Proposition 3.10]{borne-vistoli}.

Every homomorphism of groups factorizes in a surjective one and an injective one. Analogously, any morphism of affine gerbes factorizes in a locally full one and a faithful one. This factorization is called \emph{canonical factorization}.

\begin{proposition}[{\cite[Proposition 3.9]{borne-vistoli}}]\label{prop:can-fact}
	Let $f: \Gamma \to \Delta$ be a morphism of affine gerbes. There is a factorization $\Gamma \to \Delta' \to \Delta$ with $\Delta'$ an affine gerbe, called canonical factorization, such that $\Gamma \to \Delta'$ is locally full, and $\Delta' \to \Delta$ is faithful. If $\Gamma \to \Delta'' \to \Delta$ is another such factorization, there exists an equivalence $\Delta' \to \Delta''$ making the obvious diagram $2$-commutative.
\end{proposition}

\section{Neutral and $\rho$-neutral subgroups}\label{sect:neutral}

Fix an affine group scheme $\Gamma$ of finite type over $k_{0}$. The main examples are $\Gamma = \GL_{n}$ or $\Gamma = \PGL_{n}$.

Now fix $j:G \hookrightarrow \Gamma(K)$ a finite subgroup. Loosely speaking, we want to define a $j$-gerbe as a gerbe $\cG$ over an extension $k$ of $k_{0}$ such that the image in $\Gamma(\bar{k})$ of the automorphism groups of geometric points of $\cG$ is conjugate to $G$, but we have a problem. In order to do this, we need to choose an embedding $K \subseteq \bar{k}$, and different embeddings lead to different $j$-gerbes. 

The definition works if the $\gal(K/k_{0})$-conjugates of $G$ are $\Gamma(K)$-conjugates of $G$ as well. The following lemma shows that we can always reduce to this situation, without losing generality.

\begin{lemma}\label{lem:minimal-field}
	Let $j: G \hookrightarrow \Gamma(K)$ be a finite subgroup, and let $k_{1}$ be the field fixed by the elements $\sigma \in \gal(K/k_{0})$ such that $\sigma^{*}G$ is $\Gamma(K)$-conjugate to $G$.
	
	Let $\cG$ be a gerbe over a field $k$ with a morphism $\cG \to B\Gamma$. Let $p \in \cG(\bar{k})$ be a geometric point and $\alpha$ an isomorphism between the image of $p$ in $B\Gamma$ and the tautological section. Assume that there is an embedding $K \subseteq \bar{k}$ such that $\alpha \underaut_{\cG}(p) \alpha^{-1}$ is $\Gamma(\bar{k})$-conjugate to $G \subseteq \Gamma(K) \subseteq \Gamma(\bar{k})$.
	
	Then $k_{1} \subseteq k$.
\end{lemma}

\begin{proof}
	We have an action of $\aut(\bar{k}/k) = \gal(\bar{k}/k^{p})$ on $K$ through the embedding $K \subseteq \bar{k}$, where $k^{p} \subseteq \bar{k}$ is the perfect closure. Since $k_{0}$ is perfect and $k_{1}/k_{0}$ is separable, it is enough to show that $k_{1}$ is fixed by $\aut(\bar{k}/k)$. Fix $\sigma \in \aut(\bar{k}/k)$, the Galois conjugate $\sigma^{*}G$ is well defined. 
	
	Any two geometric points of $\cG$ are isomorphic, because $\cG$ is a gerbe. Fix an isomorphism $\beta: \sigma^{*}p \to p$ in $\cG$, by abuse of notation call $\beta$ the induced isomorphism in $B\Gamma$ as well. We have that $\alpha\beta\underaut_{\cG}(\sigma^{*}p)\beta^{-1}\alpha^{-1} = \alpha \underaut_{\cG}(p) \alpha^{-1}$ is $\Gamma(\bar{k})$-conjugate to $G$ by hypothesis. However, $\alpha \circ \beta$ is an isomorphism from $\sigma^{*}p$ to the tautological section, so the hypothesis also implies that $\alpha\beta\underaut_{\cG}(\sigma^{*}p)\beta^{-1}\alpha^{-1}$ is $\Gamma(\bar{k})$-conjugate to $\sigma^{*}G$.
	
	In particular, $G$ and $\sigma^{*}G$ are $\Gamma(\bar{k})$-conjugate. Let $Z \subseteq \Gamma_{K}$ be the subvariety of points mapping $G$ to $\sigma^{*}G$ by conjugation. Since $Z(\bar{k}) \neq \emptyset$, then $Z$ is non-empty and we have $Z(K) \neq \emptyset$ as well, i.e. $G$ is $\Gamma(K)$-conjugate to $\sigma^{*}G$. Equivalently, $\sigma \in \gal(\bar{k}/k)$ maps to an element of $\gal(K/k_{1})$. Since this is true for every $\sigma \in \gal(\bar{k}/k)$, we get that $k_{1} \subseteq k$.
\end{proof}

\begin{hypothesis}\label{hyp:quasi-gal-invariant}
	From now on, we always make the blanket assumption that the $\gal(K/k_{0})$-conjugates of $G \subseteq \Gamma(K)$ are $\Gamma(K)$-conjugates of $G$ as well. If this is not the case, we replace $k_{0}$ with the subfield of $K$ fixed by the elements $\sigma \in \gal(K/k_{0})$ such that $\sigma^{*}G$ is $\Gamma(K)$-conjugate to $G$. Thanks to Lemma~\ref{lem:minimal-field}, by doing so we do not lose any generality.
\end{hypothesis}

\begin{definition}\label{def:ggerbe}
	Let $j:G \hookrightarrow \Gamma(K)$ be a finite subgroup. 
	
	A $j$-gerbe is the datum of a gerbe $\cG$ over a field $k$ and a faithful morphism $\cG \to B\Gamma$ such that the image in $\Gamma$ of the automorphism group of one (and hence all) geometric point of $\cG$ is $\Gamma$-conjugate to $G$.
\end{definition}

Two $j$-gerbes $\cG \to B\Gamma$, $\cG' \to B\Gamma$ are equivalent if there is an isomorphism $\cG \to \cG'$ making the obvious diagram $2$-commutative.

\subsection{Neutral subgroups}

\begin{definition}\label{def:neutral}
	A finite subgroup $j:G \hookrightarrow \Gamma(K)$ is \emph{neutral} if every $j$-gerbe is neutral.
\end{definition}

Being neutral is a property of the embedding $j$, and not of the abstract group $G$. Still, the embedding $j: G \hookrightarrow \Gamma(K)$ is always fixed, so by abuse of terminology we will often say that $G$ is neutral.

We are interested in classifying which finite subgroups of $\Gamma(K)$ are neutral. For $\PGL_{2}$, $\GL_{2}$, $\PGL_{3}$, this is essentially the content of \cite{bresciani-p1, bresciani-sing, bresciani-plane} respectively, even if it might be not be readily clear from titles. The main focus of the present article is the next step, i.e. the case $\Gamma = \GL_{3}$. 

Given the increased complexity of the case $\Gamma = \GL_{3}$, it is due time to establish a general theoretical framework for studying the problem. Constructing such a framework is the purpose of this section and of the next one.

Recall that an element $g \in \GL_{n}(K)$ is a pseudo-reflection if its order is finite and prime with $\cha k_{0}$, and exactly one eigen-value different from $1$. If $G \subset \GL_{n}(K)$ is finite of order prime with $\cha k_{0}$, the Chevalley--Shephard--Todd theorem states that $\AA^{n}/G$ is smooth if and only if $G$ is generated by pseudo-reflections.

In the context of the arithmetic of quotient singularities, Definition~\ref{def:neutral} is useful because of the following.

\begin{theorem}[{\cite[Theorem 3.2]{bresciani-vistoli-yang}}]\label{thm:neut-R}
	Let $G \subseteq \GL_{n}(K)$ be a finite subgroup of order prime with $\cha k_{0}$. If the singularity $(\AA^{n}/G, [0])$ is of type $\rR$, then $G$ is neutral. If $G$ has no pseudo-reflections, the converse holds.
\end{theorem}

\subsection{$\rho$-neutral subgroups}\label{sect:rhoneutral}

Now assume that there is a faithful action
\[\rho: \Gamma \times X \to X\]
where $X$ is a smooth, geometrically connected algebraic space over $k_{0}$. If $\Gamma = \GL_{n}$, we consider the standard action on $\AA^{n}$, whereas for $\PGL_{n}$ we consider the action on $\PP^{n-1}$.

There is a universal family
\[\cX = [X/\Gamma] \to B\Gamma\]
of copies of $X$ on $B\Gamma$, and for every $j$-gerbe $\cG \to B\Gamma$ we have an induced family $\cX_{\cG} \to \cG$.

\begin{remark}
	If $\Gamma = \GL_{n}$, the datum of a $j$-gerbe is equivalent to a rank $n$ vector bundle $\cV \to \cG$ where $\cG$ is a gerbe such that, for a geometric point $p$ of $\cG$, the representation of $\aut_{\cG}(p)$ on $\cV_{p}$ is equivalent to the representation $G \subseteq \GL_{n}(K)$. 

	Similarly, if $\Gamma = \PGL_{n}$, the datum of a $j$-gerbe is equivalent to a dimension $n-1$ projective bundle $\cP \to \cG$ where $\cG$ is a gerbe such that, for a geometric point $p$ of $\cG$, the action of $\aut_{\cG}(p)$ on $\cP_{p}$ is equivalent to the one of $G$ on $\PP^{n-1}$. 
\end{remark}

The following strengthening of Definition~\ref{def:neutral} is extremely useful.

\begin{definition}
	A finite subgroup $G \subseteq \Gamma(K)$ is \emph{$\rho$-neutral} if, for every $j$-gerbe $\cG \to B\Gamma$ with universal family $\cX_{\cG} \to \cG$, we have $\cX_{\cG}(k) \neq \emptyset$.
\end{definition}

Clearly, $\rho$-neutral implies neutral, but the converse might fail. For instance, the trivial subgroup of $\{\Id\} \subseteq \PGL_{n}(K)$ is clearly neutral, but it is not $\rho$-neutral for $n \ge 2$ since a non-trivial Brauer--Severi variety $P$ over $k$ defines a $\{\Id\}$-gerbe $\spec k \to B\PGL_{n}$ with $P(k) = \emptyset$.

If $\Gamma = \GL_{n}$ and $X = \AA^{n}$, the universal family is a vector bundle, and a subgroup is neutral if and only if it is $\rho$-neutral, because by Hilbert's 90 theorem $\AA^{n}$ is the only vector bundle on $\spec k$.

The stack $\cX_{\cG}$ has finite inertia, hence it has a coarse moduli space $\bX_{\cG}$, which is called the \emph{compression} \cite[Definition 5.3]{bresciani-vistoli}. Since the action of $G$ is faithful, the natural morphism $\cX_{\cG} \to \bX_{\cG}$ is birational.

If $G$ has order prime with $\cha k_{0}$, $\cG$ is tame. Since $\cX_{\cG} \to \cG$ is representable, then $\cX_{\cG}$ is tame as well. Hence, by the Lang--Nishimura theorem for stacks, $\cX_{\cG}$ has a rational point if and only if $\bX_{\cG}$ has a rational point which lifts to a resolution of singularities, i.e. it is \emph{liftable} in the sense of \cite[Definition 6.6]{bresciani-vistoli}. 

\begin{definition}\label{def:liftable}
	Let $S$ be a locally noetherian, integral scheme, and $s \in S$ a point. We say that $s$ is \emph{liftable} if there exists a locally noetherian, integral scheme $S'$ with a dominant morphism $S' \to S$ such that $s$ lifts to a regular point $k(s)$-point of $S'$.
\end{definition}

If $S$ has a resolution of singularities $\tilde{S} \to S$, for instance if it has tame quotient singularities, by the Lang--Nishimura theorem $s \in S$ is liftable if and only if $s$ lifts to a $k(s)$-rational point of $\tilde{S}$. Because of this, Definition~\ref{def:liftable} generalizes the one given in \cite[Definition 6.6]{bresciani-vistoli}.

The following Lemma~\ref{lem:compression-neutral} generalizes methods about singularities from \cite[\S 6.4]{bresciani-vistoli} to the context of representations.

\begin{lemma}\label{lem:compression-neutral}
	Let $j:G \hookrightarrow \GL_{n}(K)$ be a finite group, $\cV \to \cG$ a $j$-gerbe. Write $\PP(\cV) \to \cG$ for the corresponding projective bundle, and $\bP$ for its coarse moduli space. 
	\begin{itemize}
		\item If $k$ is finite, $\cG(k) \neq \emptyset$, $\cV(k) \neq \emptyset$, $\PP(\cV)(k) \neq \emptyset$, $\bP(k) \neq \emptyset$.
	
		\item If $k$ is infinite, the following are equivalent.
	
		\begin{enumerate}
			\item $\cG(k) \neq \emptyset$.
			\item $\cV(k) \neq \emptyset$.
			\item $\PP(\cV)(k) \neq \emptyset$.
			\item $\bP$ has a dense set of smooth rational points.
		\end{enumerate}
		If furthermore $|G|$ is prime with $\cha k_{0}$, these are equivalent to		
		\begin{enumerate}
			\setcounter{enumi}{4}
			\item $\bP$ has a liftable rational point. 
		\end{enumerate}
	\end{itemize}
\end{lemma}

\begin{proof}
	If $k$ is finite, then $\cG(k) \neq \emptyset$ by \cite[Lemma 13]{bresciani-p1}, $\cV(k) \neq \emptyset$ and $\PP(\cV)(k) \neq \emptyset$ because the fibers of $\cV \to \cG$, $\PP(\cV) \to \cG$ are isomorphic to $\AA^{n}$, $\PP^{n-1}$ respectively (recall that Brauer--Severi varieties over a finite field are trivial, because the Brauer group is trivial), and $\bP(k) \neq \emptyset$ because it is the coarse moduli space of $\PP(\cV)$. Assume that $k$ is infinite.

	$(1) \Rightarrow (2)$ follows by Hilbert's 90 theorem, because a vector bundle on $\spec k$ is just a $k$-vector space. 
	
	$(2) \Rightarrow (3)$ is obvious.
	
	For $(3) \Rightarrow (4)$, notice that a rational point in $\PP(\cV)$ allows us to give an isomorphism $\PP(\cV) \simeq [\PP^{n-1}/G']$ where $G'$ is a twisted form of $G$ over $k$. In particular, $\bP \simeq \PP^{n-1}/G' = \PP^{n-1}/\bar{G'}$, where $\bar{G'} \subset \PGL_{n}$ is the image of $G'$. Let $U \subset \PP^{n-1}$ be an open subset where the action of $\bar{G}'$ is free. In particular, $U/\bar{G'}$ is regular, and $U/\bar{G'}(k)$ is dense because $k$ is infinite. In order to conclude, it is enough to show that $U/\bar{G'}$ is smooth, and not only regular. This follows from the fact that $(U/\bar{G'})_{\bar{k}} = U_{\bar{k}}/\bar{G'}$ is regular.
	
	In general, passing to quotients by finite groups does not commute with base change in positive characteristic, but in our case it works because the action is free. There is a natural map $U_{\bar{k}}/\bar{G'} \to (U/\bar{G'})_{\bar{k}}$. Since the action of $\bar{G'}$ on $U$ is free, both the maps $U_{\bar{k}} \to U_{\bar{k}}/\bar{G'}$ and $U_{\bar{k}} \to (U/\bar{G'})_{\bar{k}}$ are finite étale of degree $|\bar{G'}|$, hence $U_{\bar{k}}/\bar{G'} \to (U/\bar{G'})_{\bar{k}}$ is finite étale of degree $1$, i.e. an isomorphism.
	
	$(4) \Rightarrow (5)$ is obvious.
	
	$(4) \Rightarrow (1)$, and $(5) \Rightarrow (1)$ if $|G|$ is prime with $\cha k_{0}$. Consider the blow-up $\cB$ of $\cV$ along the $0$-section $\cG \to \cV$, and let $\bB$ be its coarse moduli space. Notice that we have a natural codimension $1$ closed embedding $\PP(\cV)  \subseteq \cB$, and hence an induced morphism $\bP \to \bB$ which is generically a codimension $1$ embedding.
	
	Since $\cB \to \bB$ is birational, we have a rational map $\bB \dashrightarrow \cB$. Since $\cB$ is smooth, $\bB$ is normal, in particular $\bB$ is smooth at the generic point of $\bP$. Furthermore, notice that the generic stabilizer of $\PP(\cV)_{K} = [\PP^{n-1}/G]$ is $G \cap K^{*}$ which has order prime with $\cha k_{0}$, hence $\cB$ is tame at the generic point of $\PP(\cV) \subset \bB$.
	
	By \cite[Corollary 3.2]{bresciani-vistoli-valuative} applied at the generic point of $\PP(\cV) \subset \bB$, the rational map $\bB \dashrightarrow \cB$ induces a rational map $\bP \dashrightarrow \PP(\cV) \to \cG$. We conclude $\cG(k) \neq \emptyset$: this is obvious if $\bP$ has a dense set of rational points, and follows from Theorem~\ref{thm:LN} if $\bP$ has a liftable rational point and $|G|$ is prime with $\cha k_{0}$.
	
\end{proof}

We also state a version for projective bundles.

\begin{lemma}\label{lem:proj-compression-neutral}
	Let $j:G \hookrightarrow \PGL_{n}(K)$ be a finite group, $\cP \to \cG$ a $j$-gerbe. Write $\bP$ for the coarse moduli space of $\cP$. 
	\begin{itemize}
		\item If $k$ is finite, $\cG(k) \neq \emptyset$, $\cP(k) \neq \emptyset$, $\bP(k) \neq \emptyset$.
	
		\item If $k$ is infinite, the following are equivalent.
	
		\begin{enumerate}
			\item $\cP(k) \neq \emptyset$.
			\item $\bP$ has a dense set of smooth rational points.
		\end{enumerate}
		
		If furthermore $|G|$ is prime with $\cha k_{0}$, these are equivalent to		
		\begin{enumerate}
			\setcounter{enumi}{2}
			\item $\bP$ has a liftable rational point. 
		\end{enumerate}
	\end{itemize}
\end{lemma}

\begin{proof}
	For $k$ finite, the proof is the same as in Lemma~\ref{lem:compression-neutral}. 
	
	$(1) \Rightarrow (2)$ is analogous to $(3) \Rightarrow (4)$ in Lemma~\ref{lem:compression-neutral}.
	
	$(2) \Rightarrow (1)$ follows from the fact that $\cP \to \bP$ is birational.
	
	$(2) \Rightarrow (3)$ is obvious.
	
	$(3) \Rightarrow (1)$ follows from Theorem~\ref{thm:LN} applied to $\bP \dashrightarrow \cP$.
\end{proof}

\begin{proposition}\label{prop:rhoneut-neut}
	Let $G \subseteq \GL_{n}(K)$ be a finite subgroup, with image $\bar{G} \subseteq \PGL_{n}(K)$. If $\bar{G}$ is $\rho$-neutral, then $G$ is neutral.
\end{proposition}

\begin{proof}
	Let $\cG \to B\GL_{n}$ be a $j$-gerbe, $\cV \to \cG$ the corresponding vector bundle. We want to prove that $\cG(k) \neq \emptyset$. By \cite[Lemma 13]{bresciani-p1}, we may assume that $k$ is infinite. Let $\cG \to \bar{\cG} \to B\PGL_{n}$ be the canonical factorization of the composition $\cG \to B\GL_{n} \to B\PGL_{n}$. 
	
	We have that $\bar{\cG} \to \PGL_{n}$ is a $\bar{j}$-gerbe. Hence, if $\cP \to \bar{\cG}$ is the associated projective bundle, we have $\cP(k) \neq \emptyset$. By Lemma~\ref{lem:proj-compression-neutral}, we get that the coarse moduli space $\bP$ of $\cP$ has a dense set of smooth $k$-rational points. Since $\bP$ coincides with the coarse moduli space of $\PP(\cV)$ as well, we conclude that $\cG(k) \neq \emptyset$ by Lemma~\ref{lem:compression-neutral}.
\end{proof}

\section{Gates}\label{sect:gates}

Let us give a definition.

\begin{definition}
	A subgroup $\Pi \subseteq \Gamma$ defined over $k_{0}$ is a \emph{gate} for $G$ if $G \subseteq \Pi(K)$ and, for every $j$-gerbe $\cG \to B\Gamma$ over an extension $k/k_{0}$, there exists a factorization $\cG \to B\Pi \to B\Gamma$ where $\cG \to B\Pi$ is a $j'$-gerbe with respect to the embedding $j': G \hookrightarrow \Pi(K)$.
\end{definition}

If $\Pi \subseteq \Gamma$ is a gate for $G$, in order to understand whether $G \subseteq \Gamma(K)$ is neutral or not we can effectively replace $\Gamma$ with $\Pi$. If $\Pi$ is much smaller than $\Gamma$, this will make our task much easier.

Of course, the usefulness of gates depends on whether small gates exist often enough. As we shall see, this actually happens.

Notice that two subgroups of $\Pi(K)$ which are conjugate in $\Gamma(K)$ are not guaranteed to be conjugate in $\Pi(K)$ as well, so the existence of the factorization $\cG \to B\Pi \to B\Gamma$ is not sufficient. We need to require that $\cG \to B\Pi$ is a $j'$-gerbe with respect to $j':G \hookrightarrow \Pi(K)$.

\begin{lemma}
	Let $\Pi \subseteq \Gamma$ be a subgroup scheme with $G \subseteq \Pi (K)$. Assume that the normalizer $N \subseteq \Gamma(K)$ of $G$ in $\Gamma(K)$ maps surjectively on the set of left cosets $\Gamma/\Pi (K)$.
	
	For every algebraically closed field $\Omega \supseteq K$, if $G' \subseteq \Pi (\Omega)$ is $\Gamma(\Omega)$-conjugate to $G$, then it is $\Pi (\Omega)$-conjugate to $G$ as well.
\end{lemma}

\begin{proof}
	Up to extending $k_{0}$, we may assume $k_{0} = K$ is algebraically closed. We may then regard $G, N$ as the geometric points of group schemes $\fG, \fN$ of finite type over $K$. The fact that $\fN(K) \to \Gamma/\Pi (K)$ is surjective then implies that $\fN(\Omega) \to \Gamma/\Pi (\Omega)$ is surjective as well. 
	
	Let $x \in \Gamma(\Omega)$ be any element, consider the conjugate $x G x^{-1}$. Choose $y$ any element of $\fN(\Omega)$ such that $y^{-1}\Pi(\Omega) = x^{-1}\Pi(\Omega)$. Then $x G x^{-1} = x y^{-1} G y x^{-1}$, and $x y^{-1} \in \Pi (\Omega)$ because $x y^{-1}\Pi(\Omega) = x x^{-1}\Pi(\Omega) = \Pi(\Omega)$.
\end{proof}

\begin{corollary}\label{cor:easy-gate}
	Let $\Pi \subseteq \Gamma$ be a subgroup scheme with $G \subseteq \Pi (K)$. Assume that the normalizer $N \subseteq \Gamma(K)$ of $G$ in $\Gamma$ maps surjectively on the left cosets $\Gamma/\Pi (K)$.
	
	If every $j$-gerbe factorizes through $B\Pi$, then $\Pi$ is a gate for $G$ as well. 
\end{corollary}

\begin{lemma}\label{lem:gate-special}
	Let $\Pi \subseteq \Gamma$ be a normal subgroup scheme with $\Gamma/\Pi$ special and $G \subseteq \Pi(K)$. Assume that the normalizer of $G$ maps surjectively on $\Gamma/\Pi(K)$. Then $\Pi$ is a gate for $G$.
\end{lemma}

\begin{proof}
	By Corollary~\ref{cor:easy-gate}, we only need to check that every $j$-gerbe factorizes through $B\Pi$.

	Let $\cG \to B\Gamma$ be a $j$-gerbe. We have a cartesian diagram
	\[\begin{tikzcd}
		B\Pi \rar \dar			&	\spec k \dar	\\
		B\Gamma \rar			&	B\Gamma/\Pi
	\end{tikzcd}\]
	where $\spec k \to B\Gamma/\Pi$ is the tautological section.
	
	By the following Lemma~\ref{lem:trivial-fact}, we have a factorization $\cG \to \spec k \to B\Gamma/\Pi$. Since $\Gamma/\Pi$ is special, the new section $\spec k \to B\Gamma/\Pi$ is isomorphic to the tautological one, hence we get the desired morphism $\cG \to B\Pi$ by the diagram above.
\end{proof}

\begin{lemma}\label{lem:trivial-fact}
	Let $\Gamma \to \Delta$ be a morphism of affine gerbes over a field $k$. Fix $\gamma \in \Gamma(k')$ for some extension $k'/k$. If $\underaut_{\Gamma}(\gamma) \to \underaut_{\Delta}(\gamma)$ is trivial, there exists a factorization $\Gamma \to \spec k \to \Delta$.
\end{lemma}

\begin{proof}
	Let $\Gamma \to \Delta' \to \Delta$ be the canonical factorization. Since $\underaut_{\Gamma}(\gamma) \to \underaut_{\Delta}(\gamma)$ is trivial, then $\Delta'_{k'} = \spec k'$, which in turn implies $\Delta' = \spec k$.
\end{proof}

Given a positive integer $c$, write $\SL_{n}^{(c)} \subseteq \GL_{n}$ for the subgroup of matrices $M$ such that $(\det M)^{c} = 1$. 

\begin{corollary}\label{cor:slnm}
	If $\Gamma = \GL_{n}$ and $G \subseteq \SL_{n}^{(c)}(K)$, then $\SL_{n}^{(c)}$ is a gate for $G$.
\end{corollary}

\begin{proof}
	Apply Lemma~\ref{lem:gate-special} with $\Gamma = \GL_{n}$, $\Pi = \SL_{n}^{(c)}$. Notice that $\GL_{n}/\SL_{n}^{(c)} \simeq \GG_{m}$ is special. The multiples of the identity are in the normalizer of $G$, and map surjectively on $\GL_{n}/\SL_{n}^{(c)}(K)$.
\end{proof}

\subsection{Saturated subgroups}

Fix a quotient $\pi : \Gamma \to \bar{\Gamma}$ of $\Gamma$. The main example is $\Gamma = \GL_{n}$ and $\bar{\Gamma} = \PGL_{n}$. If $G \subseteq \Gamma(K)$ is a finite subgroup, denote by $\bar{G} = \pi(G)$ its image in $\bar{\Gamma}$.

Given a $j$-gerbe $\cG \to B\Gamma$, taking the canonical factorization of $\cG \to B\bar{\Gamma}$ we get a $2$-commutative diagram
\[\begin{tikzcd}
	\cG \rar \dar	&	B\Gamma\dar		\\
	\bar{\cG} \rar	&	B\bar{\Gamma}
\end{tikzcd}\]
where $\cG \to \bar{\cG}$ is locally full and $\bar{\cG} \to B\bar{\Gamma}$ is faithful. In particular, $\bar{\cG} \to B\bar{\Gamma}$ is a $\bar{j}$-gerbe. Furthermore, the factorization $\cG \to \bar{\cG} \to B\bar{\Gamma}$ is unique up to equivalence.

We say that $\bar{\cG}$ is the \emph{induced} $\bar{j}$-gerbe.

\begin{definition}\label{def:saturated}
	The group $G$ is \emph{saturated} if there exists a subgroup scheme $\Pi \subseteq \Gamma$ such that $\bar{\Pi} = \pi(\Pi)$ is a gate for $\bar{G}$ and
	\[\pi^{-1}(\bar{G}) \cap \Pi_{K} = G\]
	as group schemes over $K$.
\end{definition}

\begin{lemma}\label{lem:saturated}
	If $G$ is saturated, every $\bar{j}$-gerbe has the form $\bar{\cG}$ for some $j$-gerbe $\cG$.
\end{lemma}

\begin{proof}
	Let $\Pi \subseteq \Gamma$ be as in Definition~\ref{def:saturated}. Up to replacing $\Gamma$ with $\Pi$, we may assume that $\pi^{-1}(\bar{G}) = G$.
	
	Given a $\bar{j}$-gerbe $\cH \to B\bar{\Gamma}$, define $\cG$ by the $2$-cartesian diagram
	\[\begin{tikzcd}
		\cG \rar \dar	&	B\Gamma \dar	\\
		\cH \rar		&	B\bar{\Gamma}
	\end{tikzcd}\]
	Since $\pi : \Gamma \to \bar{\Gamma}$ is surjective, then $\cG$ is a $\pi^{-1}(\bar{G})$-gerbe, and hence a $j$-gerbe since $\pi^{-1}(\bar{G}) = G$. The uniqueness part of Proposition~\ref{prop:can-fact} implies that $\cH \to B\bar{\Gamma}$ is equivalent to $\bar{\cG} \to B\bar{\Gamma}$.
\end{proof}

\begin{corollary}
	If $G \subseteq \Gamma(K)$ is neutral and saturated, $\bar{G} \subseteq \bar{\Gamma}(K)$ is neutral as well.
\end{corollary}

As a consequence, we obtain the following strengthening of Proposition~\ref{prop:rhoneut-neut} for saturated subgroups.

\begin{proposition}\label{prop:saturated-neut}
	Let $G \subseteq \GL_{n}(K)$ be a finite subgroup with image $\bar{G} \subseteq \PGL_{n}(K)$, assume that $G$ is saturated. Then $G$ is neutral if and only if $\bar{G}$ is $\rho$-neutral.
\end{proposition}

\begin{proof}
	The ``if'' part is Proposition~\ref{prop:rhoneut-neut}. Assume that $G$ is neutral, let us prove that $\bar{G}$ is $\rho$-neutral. By Lemma~\ref{lem:saturated}, every $\bar{j}$-gerbe has the form $\bar{\cG}$ for some $j$-gerbe $\cG$. Let $\cV \to \cG$ be the corresponding vector bundle, and $\cP \to \bar{\cG}$ the projective bundle, there is a natural map $\PP(\cV) \to \cP$. Since $G$ is neutral, $\cG(k) \neq \emptyset$ and hence $\PP(\cV)(k) \neq \emptyset$ by Lemma~\ref{lem:compression-neutral}. We conclude that $\cP(k) \neq \emptyset$, too.
\end{proof}

\subsection{Normal subgroups}\label{sect:norm-gates}

Recall that we are assuming Hypothesis~\ref{hyp:quasi-gal-invariant}, i.e. that the $\gal(K/k_{0})$-conjugates of $G$ are $\Gamma(K)$-conjugates of $G$ as well. Studying $j$-gerbes becomes easier if we assume more strongly that $\gal(K/k_{0})$ maps $G$ to itself.

\begin{hypothesis}\label{hyp:gal-invariant}
	The subgroup $G \subseteq \Gamma(K)$ is $\gal(K/k_{0})$-invariant. Equivalently, there is a finite étale group scheme $\fG$ over $k_{0}$ with an embedding $\fG \subseteq \Gamma$ such that $G = \fG(K) \subseteq \Gamma(K)$.
\end{hypothesis}

\begin{lemma}\label{lem:diag-hyp}
	If $\Gamma = \GL_{n}$ and $G$ is diagonal, then Hypothesis~\ref{hyp:gal-invariant} is satisfied.
	
	As a consequence, if $G$ is abelian of order prime with $\cha k_{0}$ a conjugate of $G$ satisfies Hypothesis~\ref{hyp:gal-invariant}. 
\end{lemma}

\begin{proof}
	If $G$ is diagonal, there exists an integer $m$ such that $G$ is generated by elements of the form $g = \diag(\zeta_{m}^{a_{1}}, \dots, \zeta_{m}^{a_{n}})$. Choose $\sigma \in \gal(K/k_{0})$, then $\sigma(\zeta_{m}) = \zeta_{m}^{d}$ for some $d$ prime with $m$, and $\sigma(g) = g^{d}$. Since $\gcd(d,m) = 1$, then $g$ and $g^{d}$ generate the same group, and hence $\sigma(G) = G$.
\end{proof}

For the rest of \S\ref{sect:norm-gates}, we assume Hypothesis~\ref{hyp:gal-invariant}, but not in the rest of the paper.

If $\fG$ is normal in $\Gamma$, then our problem can be studied with cohomological methods. Obviously, $\fG$ is not normal most of the time. However, it turns out that the normalizer of $\fG$ is a gate. 

\begin{proposition}\label{prop:gate-norm}
	Let $\fG \subseteq \Gamma$ be a finite subgroup scheme. The normalizer $\fN$ is a gate for $j: G = \fG(K) \hookrightarrow \Gamma$.
\end{proposition}

\begin{proof}
	Every $j$-gerbe is locally equivalent to $B\fG \to B\Gamma$, see Definition~\ref{def:loc-eq} and Lemma~\ref{lem:easy-loc-eq}. The statement then follows from Theorem~\ref{thm:gerbe-normalizer}.
\end{proof}

\begin{remark}
	Thanks to Proposition~\ref{prop:gate-norm}, we can always reduce to the case in which $\fG$ is normal in $\Gamma$. In this case, the hypothesis of Corollary~\ref{cor:easy-gate} is automatically satisfied, so that finding gates is easier.
\end{remark}

Let $\fG \subseteq \Pi \subseteq \Gamma$ be a subgroup normalizing $\fG$, write $\fQ = \Pi/\fG$. Recall that the isomorphism classes of $B \fQ(k)$ are in natural bijection with the elements of $\H^{1}(k, \fQ)$, i.e. $\fQ$-torsors. Given a section $\tau: \spec k \to B\fQ$, the $2$-cartesian diagram

\[\begin{tikzcd}
	\cG_{\tau}	\rar \dar	&	\spec k	\dar	\\
	B\Pi \rar				&	B\fQ
\end{tikzcd}\]
defines a $j$-gerbe with the composition $\cG_{\tau} \to B\Pi \to B\Gamma$, and $\cG_{\tau}(k) \neq \emptyset$ if and only if $\tau$ lifts to $B\Pi$. As an immediate consequence, we obtain the following.

\begin{lemma}\label{lem:not-neutral}
	Let $\fG \subseteq \Pi \subseteq \Gamma$ be a subgroup normalizing $\fG$, write $\fQ = \Pi/\fG$. If
	\[\H^{1}(k, \Pi) \to \H^{1}(k, \fQ)\]
	is not surjective for some extension $k$ of $k_{0}$, then $G$ is not neutral.
\end{lemma}

\begin{lemma}\label{lem:gerbe-H1}
	Assume that $\fG$ is normal in $\Gamma$. For every extension $k/k_{0}$, the mapping $\tau \mapsto \cG_{\tau}$ defines a bijection between $\H^{1}(k, \Gamma/\fG)$ and $j$-gerbes up to equivalence.
	
	Neutral $j$-gerbes correspond to elements of $\H^{1}(k, \Gamma/\fG)$ lifting to $\H^{1}(k, \Gamma)$.
\end{lemma}

\begin{proof}
	Write $\fQ = \Gamma/\fG$. Let us construct a map from equivalence classes of $G$ gerbes to $\H^{1}(k, \fQ)$. Let $\cG \to B\Gamma$ be a $j$-gerbe over $k$, with $k$ an extension of $k_{0}$.
	
	By Lemma~\ref{lem:trivial-fact}, we get a factorization $\cG \to \spec k \to B\fQ$ of $\cG \to B\Gamma \to B\fQ$. The section $\tau:\spec k \to B\fQ$ is the desired element of $\H^{1}(k, \fQ)$. By definition of $\cG_{\tau}$, we have an induced morphism $\cG \to \cG_{\tau}$ which becomes an isomorphism after base change to $\bar{k}$. It follows that $\cG \to \cG_{\tau}$ is an isomorphism as well.
	
	On the other hand, let $\tau \in \H^{1}(k, \fQ)$ be a cohomology class, $\cG_{\tau} \to B\Gamma$ the associated $j$-gerbe, and $\tau' \in \H^{1}(k, \fQ)$ the cohomology class of $\cG_{\tau}$. The two compositions $\cG_{\tau} \to \spec k \xrightarrow{\tau} \to B\fQ$, $\cG_{\tau} \to \spec k \xrightarrow{\tau'} \to B\fQ$ both coincide with $\cG_{\tau} \to B\Gamma \to B\fQ$, hence $\tau = \tau'$ by the uniqueness part of Proposition~\ref{prop:can-fact}.
	
	The second statement is obvious.
\end{proof}

As a direct consequence, we obtain the following.

\begin{proposition}\label{prop:ggerbe}
	Let $\fG \subseteq \Pi \subseteq \Gamma$ be subgroups, with $\fG$ normal in $\Gamma$. The following are equivalent.
	
	\begin{itemize}
		\item $\Pi$ is a gate for $\fG$.
		\item Every $j$-gerbe has the form $\cG_{\tau} \to \Pi \to \Gamma$ for some $\tau \in \H^{1}(k, \Pi/\fG)$.
		\item The map
				\[\H^{1}(k, \Pi/\fG) \to \H^{1}(k, \Gamma/\fG)\]
				is surjective for every extension $k/k_{0}$.
	\end{itemize}
	In this case, a class in $\H^{1}(k, \Pi/\fG)$ defines a neutral $j$-gerbe if and only if it lifts to $\H^{1}(k, \Pi)$.
\end{proposition}

The following theorem is a direct consequence of Proposition~\ref{prop:ggerbe}.

\begin{theorem}\label{thm:coh-crit}
	Let $\Pi$ be a gate for $\fG$, with $\fG \subseteq \Pi$ normal, and write $\fQ = \Pi/\fG$. The subgroup $G = \fG(K) \subseteq \Gamma(K)$ is neutral if and only if
	\[\H^{1}(k, \Pi) \to \H^{1}(k, \fQ)\]
	is surjective for every extension $k$ of $k_{0}$.
\end{theorem}

As a direct consequence of Theorem~\ref{thm:coh-crit} and Proposition~\ref{prop:gate-norm}, we obtain the following variant of \cite[Theorem 4]{bresciani-structure}.

\begin{theorem}\label{thm:norm}
	Let $\fN \subseteq \Gamma$ be the normalizer of $\fG$, and $\fQ = \fN/\fG$.
	
	The subgroup $G = \fG(K) \subseteq \Gamma(K)$ is neutral if and only if
	\[\H^{1}(k, \fN) \to \H^{1}(k, \fQ)\]
	is surjective for every extension $k$ of $k_{0}$.
\end{theorem}

\begin{corollary}\label{cor:special}
	If $\fN/\fG$ is special, then $G$ is neutral.
\end{corollary}

\subsection{Distinguished subvarieties}

To prove that $G$ is $\rho$-neutral, we have to find rational points on $\cX_{\cG}$ for every $j$-gerbe $\cG \to B\Gamma$. More generally, it is useful to know ways to construct substacks of $\cX_{\cG}$.

Methods for doing this were given in \cite[\S 7]{bresciani-structure}. In this section, we generalize these methods and translate them to our setting.

\begin{definition}
	The \emph{extended normalizer} $\tilde{N}_{G}$ of $G$ is its normalizer in the semi-direct product $\Gamma_{K} \rtimes \gal(K/k_{0})$. If $\cha k_{0} = 0$, we may conflate it with the normalizer in $\Gamma(K) \rtimes \gal(K/k_{0})$.
	
	Notice that, thanks to Hypothesis~\ref{hyp:quasi-gal-invariant}, the homomorphism $\tilde{N}_{G} \to \gal(K/k_{0})$ is surjective. In particular, $\tilde{N}_{G}$ is an extension of $\gal(K/k_{0})$ by the normalizer $N_{G}$ of $G$ in $\Gamma_{K}$.
\end{definition}

\begin{example}\label{example:semilinear}
	If $\Gamma = \GL_{n}$, then $\GL_{n}(K) \rtimes \gal(K/k_{0})$ is the group of bijective maps $f: K^{n} \to K^{n}$ such that there exists $\sigma \in \gal(K/k_{0})$ with
	\begin{itemize}
		\item $f(v + w) = f(v) + f(w)$,
		\item $f(\lambda v) = \sigma(\lambda) f(v)$.
	\end{itemize}
	We call these \emph{semi-linear} automorphisms of $K^{n}$.
\end{example}

\begin{remark}\label{rmk:norm-hyp}
	If the standard Galois action on $\Gamma(K)$ maps $G$ to itself, we have $\tilde{N}_{G} = N_{G} \rtimes \gal(K/k_{0})$.
\end{remark}

As in \S\ref{sect:rhoneutral}, assume we have an action $\rho: \Gamma \times X \to X$ defined over $k_{0}$. The action $\rho$ induces an extended action
\[\tilde{\rho}: \left( \Gamma_{K} \rtimes \gal(K/k_{0}) \right) \times X_{K} \to X_{K}.\]
For every $j$-gerbe $\cG \to B\Gamma$, we have an associated family $\cX_{\cG} \to \cG$ which is the pullback of $\cX = [X/\Gamma] \to B\Gamma$.

\begin{definition}
	A subvariety $Y \subseteq X_{K}$ is \emph{distinguished} if $\tilde{N}_{G}$ maps $Y$ to itself.
\end{definition}

\begin{lemma}\label{lem:distinguished}
	Let $Y \subseteq X_{K}$ be a distinguished subvariety.

	For every $j$-gerbe $\cG \to B\Gamma$, we have that $Y$ descends to a subfamily $\cY \subseteq \cX_{\cG}$ of the associated family.
	
	This means that, for every geometric point $p:\spec \Omega \to \cG$ and every isomorphism between $p$ and the tautological section of $B\Gamma$ such that $\aut_{\cG}(p)$ maps to $G \subseteq \Gamma(\Omega)$, the fiber of $\cY \subseteq \cX_{\cG} \to \cG$ is $Y_{\Omega} \subseteq X_{\Omega}$.
\end{lemma}

\begin{proof}
	Let $k_{1}/k_{0}$ be a finite Galois extension such that $G \subseteq \Gamma(k_{1})$ and such that $Y$ descends to $Y_{1} \subseteq X_{k_{1}}$, write $N_{G,1}$ for the normalizer of $G$ in $\Gamma_{k_{1}}$. We may think of $\spec k_{1} \to \spec k_{0}$ as a cover in the fppf topology, and $G \subseteq \Gamma_{k_{1}}$. By Lemmas~\ref{lem:gerbe-type} and \ref{lem:loc-type}, we get a type $\Theta$ in the sense of Definition~\ref{def:type} such that $j$-gerbes coincide with morphisms $\cG \to B\Gamma$ of type $\Theta$, where $\cG$ is a gerbe.
	
	By Theorem~\ref{thm:type-normalizer}, we have a gerbe $\cN$ with a faithful morphism $\cN \to B\Gamma$ such that every $j$-gerbe factorizes through $\cN$. To prove the statement, it is sufficient to show that $Y$ descends to a subfamily of $\cX_{\cN}$.
	
	Consider the $j$-gerbe $BG \to B\Gamma$ over $K$, we have a factorization $BG \to \cN$ which identifies $\cN_{K}$ with $BN_{G}$ by Theorem~\ref{thm:type-normalizer}. Consider the groupoid $R = \spec K \times_{\cN} \spec K$ with the two natural projections $R \rightrightarrows \spec K$. The fact that $\tilde{N}_{G}$ maps $Y$ to itself then implies the same for $R$, hence $Y$ descends to a subfamily $[Y/R] \subseteq [X/R] = \cX_{\cN}$.
\end{proof}

\subsection{Abelian subgroups of $\GL_{n}$}

At least in low dimensions, when $G \subseteq \GL_{n}(K)$ is not abelian it is very often the case that $\bar{G}$ is $\rho$-neutral, so that $G$ is neutral by Proposition~\ref{prop:rhoneut-neut}. Because of this, most of our work is focused on finite, abelian subgroups of $G \subseteq \GL_{n}(K)$. Let also assume that $|G|$ is prime with $\cha k_{0}$, so that we may assume $G$ to be diagonal up to conjugation. Notice that in this case we may assume that $G = \fG(K)$ for some group scheme $\fG$ over $k_{0}$ thanks to Lemma~\ref{lem:diag-hyp}.

Among diagonal subgroups, the first class to consider is groups whose eigenspaces are all $1$-dimensional.

\begin{lemma}\label{lem:permutation}
	Let $G = \fG(K) \subseteq \GL_{n}(K)$ be a finite, diagonal subgroup with $1$-dimensional eigenspaces. The normalizer $\fN$ of $\fG$ has the form $\GG_{m}^{n} \rtimes S$, where $S \subseteq S_{n}$ is a group of permutation matrices which permutes the coordinates.
\end{lemma}

\begin{proof}
	The normalizer acts on the $n$ eigenspaces, giving a homomorphism $\fN \to S_{n}$ with image $S \subseteq S_{n}$ and kernel $\GG_{m}^{n}$. To conclude, it is enough to show that the permutation matrices corresponding to elements of $S$ are in $\fN$. This follows from the fact that any matrix which permutes the coordinates has the form $D \cdot P$, where $D$ is diagonal and $P$ is a permutation matrix.
\end{proof}

\begin{theorem}\label{thm:diag-neutral+}
	Let $G = \fG(K) \subseteq \GL_{n}(K)$ be a finite, diagonal subgroup with $1$-dimensional eigenspaces (in particular, $|G|$ is prime with $\cha k_{0}$). Let the normalizer be $\GG_{m}^{n} \rtimes S$, where $S \subseteq S_{n}$ is a group of permutation matrices. 
	
	Let $S' \subseteq S$ be a subgroup of index $\iota_{1}$, and consider the action of $S'$ on the character group $X(\GG_{m}^{n}/G) \subseteq X(\GG_{m}^{n}) = \ZZ^{n}$. Assume that $X(\GG_{m}^{n}/G)$ contains a submodule $N$ of index $\iota_{2}$ which is a direct summand of a permutation $S'$-module.
	
	The index of every $j$-gerbe divides $\iota_{1}\iota_{2}$.
\end{theorem}

\begin{proof}
	Let $k/k_{0}$ be an extension and $\phi \in \H^{1}(k, (\GG_{m}^{n}/\fG) \rtimes S)$ a class. Denote by $\tau$ the image of $\phi$ in $\H^{1}(k, S)$, and by $\psi$ the image of $\tau$ in $\H^{1}(k, (\GG_{m}^{n}/\fG) \rtimes S)$. By Theorem~\ref{thm:coh-crit} it is enough to find a finite, reduced algebra $A/k$ with $[A:k] = \iota_{1}\iota_{2}$ and $\phi_{A} = \psi_{A}$, since $\psi$ clearly lifts to $\H^{1}(k, \GG_{m}^{n} \rtimes S)$ and thus corresponds to a neutral gerbe.
	
	Let $H \subseteq  \H^{1}(k, (\GG_{m}^{n}/\fG) \rtimes S)$ be the inverse image of $\tau$. By non-abelian cohomology \cite[\S 5.4]{serre}, there is a surjective map
	\[\H^{1}(k, T) \to H,\]
	where $T$ is the twist of $\GG_{m}^{n}/\fG$ by $\tau$. Since both $\phi$ and $\psi$ are elements of $H$, it is enough to show that every $T$-torsor over $k$ is trivialized by a finite, reduced algebra of degree exactly $\iota_{1}\iota_{2}$. In fact, if $\phi_{0},\psi_{0} \in \H^{1}(k, T)$ lift $\phi$, $\psi$, it is enough to trivialize the difference $\phi_{0} - \psi_{0} \in \H^{1}(k, T)$.
	
	The class $\tau$ corresponds to an $S$-torsor $\spec A_{0} \to \spec k$, where $A_{0}$ is an étale algebra of degree $|S|$. Let $A_{1} = A_{0}^{S'}$ be the sub-algebra of elements fixed by $S'$, by construction $A_{1}$ has degree $\iota_{1}$.
	
	Let $\spec k' \subseteq \spec A_{1}$ be a point of $\spec A_{1}$, then $\tau_{k'}$ lifts to a class $\tau' \in \H^{1}(k', S')$. Furthermore, $T_{k'}$ coincides with the twist of $\GG_{m}^{n}/\fG$ by $\tau'$. For any $T_{k'}$-torsor, it is enough to find a finite, reduced algebra over $k'$ of degree exactly $\iota_{2}$ which trivializes the torsor: by working at each point of $\spec A_{1}$ separatedly, we may then construct the desired algebra $A$ of degree $\iota_{2}$ over $A_{1}$, and degree $\iota_{1}\iota_{2}$ over $k$.
	
	Twisting by $\tau'$, the $S'$-submodule $N \subseteq X(\GG_{m}^{n}/G)$ induces a $\gal(k'^{s}/k')$-module $M \subseteq X(T_{k'})$, where $k'^{s}$ is a separable closure. Furthermore, the fact that $N$ is a direct summand of a permutation $S'$-module implies that $M$ is a direct summand of a permutation $\gal(k'^{s}/k')$-module. In fact, the action of $\gal(k'^{s}/k')$ on $X(T_{k'})$ is just the action of $S'$ on $X(\GG_{m}^{n}/G) = \ZZ^{n}$ composed with the homomorphism $\gal(k'^{s}/k') \to S'$ corresponding to $\tau'$.
	
	Let $P$ be the torus over $k'$ with character group $M$, we have a quotient map $T_{k'} \to P$ of degree $\iota_{2}$ and $\H^{1}(k', P)$ is trivial by \cite[Theorem 18]{huruguen} \cite{ct}. It follows that every $T_{k'}$-torsor is induced by a $\ker(T_{k'} \to P)$-torsor. We conclude because every $\ker(T_{k'} \to P)$-torsor $D$ is trivialized by a finite, reduced algebra of degree $\iota_{2}$ over $k'$. In fact, $\H^{0}(D, \cO_{D})^{\rm red}$ trivializes $D$, and we may take a direct sum of $\left(\iota_{2}/\deg \H^{0}(D, \cO_{D})^{\rm red}\right)$ copies of $\H^{0}(D, \cO_{D})^{\rm red}$. 
	
	In the above, we have used that $\deg \H^{0}(D, \cO_{D})^{\rm red}$ divides $\iota_{2} = \deg D$, which is not true for arbitrary finite algebras. However, it is true in this case. In fact, if $\ker(T_{k'} \to P) = E \times C$ with $E$ étale and $C$ connected, then $D \simeq D_{E} \otimes D_{C}$, where $D_{E}$ is étale and $D_{C}$ is local artinian, so
	\[\iota_{2} = \deg D = \deg D_{E} \cdot \deg D_{C} = \deg D^{\rm red} \cdot m,\]
	where $m$ is the length of $D_{C}$.
\end{proof}

In the particular case $S' = S$, by applying Lemma~\ref{lem:index-neutral} and Theorem~\ref{thm:diag-neutral+} we directly obtain the following.

\begin{theorem}\label{thm:diag-neutral}
	Let $G = \fG(K) \subseteq \GL_{n}(K)$ be a finite, diagonal subgroup with $1$-dimensional eigenspaces. Let the normalizer be $\GG_{m}^{n} \rtimes S$, where $S \subseteq S_{n}$ is a group of permutation matrices. 
	
	Consider the action of $S$ on the character group $X(\GG_{m}^{n}/G) \subseteq X(\GG_{m}^{n}) = \ZZ^{n}$. If $X(\GG_{m}^{n}/G)$ is a direct summand of a permutation $S$-module, then $G$ is neutral.
\end{theorem}

\section{Neutral subgroups in dimension $\le 2$}\label{sect:dim2}

As we mentioned previously, neutral and $\rho$-neutral subgroups of $\PGL_{2}$, $\GL_{2}$, $\PGL_{3}$ have essentially been classified in \cite{bresciani-p1, bresciani-sing, bresciani-plane}. The point of view and the terminology was different, though, and the classification needs to be completed in some cases. For instance, for $\GL_{2}$ the paper \cite{bresciani-sing} only studies subgroups without pseudo-reflections, because the focus is on singularities rather than representations.

In this section, we summarize and supplement this classification.

\subsection{$\PGL_{1}$} For completeness, we start with $\PGL_{1}$, that is the trivial group. The only finite subgroup $G \subseteq \PGL_{1}(K)$ is the trivial one, every $j$-gerbe is trivial, hence $G$ is neutral. Furthermore, $\PP^{0}$ is just a point, every twisted form of $\PP^{0}$ is just a point as well, and hence $G$ is $\rho$-neutral.

Notice that the trivial subgroup $G$ of $\PGL_{n}(K)$ is in general not $\rho$-neutral. A $(n - 1)$-dimensional Brauer--Severi variety $P$ over $k$ defines a $j$-gerbe $\spec k \to B\PGL_{n}$ with associated family $P \to \spec k$. If $P(k) = \emptyset$, then the trivial subgroup of $\PGL_{n}$ is not $\rho$-neutral.

\subsection{$\GL_{1} = \GG_{m}$}

\begin{proposition}\label{prop:gl1}
	For every integer $n \ge 1$ prime with $\cha k_{0}$, the group $\mu_{n} \subseteq \GG_{m}(K)$ is neutral. Hence, every finite subgroup of $\GG_{m}(K)$ is neutral.
\end{proposition}

\begin{proof}
	This is a direct consequence of Lemma~\ref{lem:line-neut}, or alternatively Proposition~\ref{prop:rhoneut-neut}.
\end{proof}

\subsection{$\PGL_{2}$}

\begin{theorem}\label{thm:pgl2}
	\begin{itemize}
		\item A cyclic subgroup $C_{n} \subseteq \PGL_{2}(K)$ for $n$ even and prime with $\cha k_{0}$ is not neutral.
		\item A cyclic subgroup $C_{n} \subseteq \PGL_{2}(K)$ for $n$ odd and prime with $\cha k_{0}$ is neutral but not $\rho$-neutral.
		\item Every other finite subgroup of $\PGL_{2}(K)$ is $\rho$-neutral.
	\end{itemize}
	
	In particular, if $\cha k_{0} = 2$ every finite subgroup of $\PGL_{2}(K)$ is neutral, and every finite subgroup which is not cyclic of odd order is $\rho$-neutral as well. 
\end{theorem}

The proof is essentially given in \cite{bresciani-p1}, here we give additional details and construct counterexamples for $C_{n}$ for every $n$.

\begin{proof}
	Let $G \subseteq \PGL_{2}(K)$ be a finite subgroup, $\cG$ a $j$-gerbe, $\cP \to \cG$ the corresponding projective bundle, $\bP$ the coarse moduli space of $\cP$. Since we are in dimension $1$, $\bP$ is regular of genus $0$, and we have a rational map $\bP \dashrightarrow \cG$.
	
	If $\cP(k) = \emptyset$ and $\bP$ is regular, then $\bP(k) = \emptyset$ by Lemma~\ref{lem:proj-compression-neutral}, hence $G = C_{n}$ is cyclic of order $n$ prime with $\cha k_{0}$ by \cite[Proposition 28]{bresciani-p1}. If $\cG(k) = \emptyset$, then $n$ is even by \cite[Lemma 30]{bresciani-p1}. This shows that if $G$ is not $\rho$-neutral (resp. not neutral) then $G$ is cyclic of order prime with $\cha k_{0}$ (resp. cyclic of order prime with $\cha k_{0}$ and divisible by $2$).
	
	It remains to show that a cyclic subgroup $C_{n} \subset \PGL_{2}(K)$, $n$ prime with $\cha k_{0}$, is not $\rho$-neutral for $n$ odd, and not neutral for $n$ even (all such subgroups for fixed $n$ are conjugated, hence it is enough to work with one of them). Let $P$ be any Brauer--Severi variety of dimension $1$ over some field $k$ with $P(k) = \emptyset$, and $p \in P$ a point such that $[k(p):k] = 2$ and $k(p)$ separable (the point exists by \cite[Proposition 4.5.4]{gille-szamuely}). For instance, such a $P$ exists over $k = K(s,t)$ \cite[Chapter 1]{gille-szamuely}.
	
	Consider $p$ as a divisor of degree $2$ and let $\cP = \sqrt[n]{P, p}$ be the $n$-th root stack \cite[Appendix B]{abramovich-graber-vistoli}. We have $\cP_{k^{s}} \simeq \sqrt[n]{\PP^{1}, p + \bar{p}} = [\PP^{1}/\mu_{n}]$, whose étale fundamental group is $\mu_{n}$ since $n$ is prime with $\cha k_{0}$. Let $\cP \to \cG$ be the étale fundamental gerbe of $\cP$ \cite[\S 8]{borne-vistoli-nori}. Since the étale fundamental gerbe behaves well under separable base change \cite[Proposition A.18]{bresciani-anab}, the base change to the separable closure $k^{s}$ of $\cP \to \cG$ is $[\PP^{1}/\mu_{n}] \to B\mu_{n}$, hence $\cP \to \cG$ is a $1$-dimensional projective bundle with $\cP(k) = \emptyset$. This shows that $C_{n}$ is not $\rho$-neutral if $n$ is prime with $\cha k_{0}$. If $n$ is even as well, then $\cG(k) = \emptyset$ by \cite[Lemma 30]{bresciani-p1}, hence $C_{n}$ is not neutral for $n$ even and prime with $\cha k_{0}$.
\end{proof}

\subsection{$\GL_{2}$}

In \cite{bresciani-sing}, a classification is given for subgroups $G \subseteq \GL_{2}(K)$ without pseudo-reflections and order prime with $\cha k_{0}$. We now give a complete classification. The proof we give here is considerably easier than the one given in \cite{bresciani-sing}, thanks to the theoretical background developed in the previous sections.

\begin{theorem}\label{thm:gl2}
	Let $G \subseteq \GL_{2}(K)$ be a finite group, $\bar{G}$ its image in $\PGL_{2}(K)$. If $\cha k_{0} = 2$, then $G$ is neutral. If $\cha k_{0} \neq 2$, the following are equivalent.
	\begin{enumerate}
		\item $G$ is not neutral.
		\item $G$ is saturated and $\bar{G}$ is not $\rho$-neutral.
		\item $G$ is conjugate to
		\[\left<  
			\left( \begin{array}{cc}
			\zeta_{2m} 	& 	0  			\\
			0 			&	\zeta_{2m} 	\\
			\end{array} \right),~
			\left( \begin{array}{cc}
			\zeta_{2n}^{-1} 	& 	0  			\\
			0 					&	\zeta_{2n} 	\\
			\end{array} \right)
			\right>\]
			for some $m, n \ge 1$ prime with $\cha k_{0}$, which is the inverse image in $\SL_{2}^{(m)}$ of
			\[\left< \diag(1, \zeta_{n}) \right>.\]
	\end{enumerate}
\end{theorem}

\begin{proof}
	Assume first that $\cha k_{0} \neq 2$.
	
	$(3) \Rightarrow (2)$. We have that $G \subseteq \SL_{2}^{(m)}$ is the inverse image of $\bar{G}$ in $\SL_{2}^{(m)}$, so $G$ is saturated. Since $\bar{G}$ is cyclic of order prime with $\cha k_{0}$, it is not $\rho$-neutral by Theorem~\ref{thm:pgl2}.
	
	$(2) \Rightarrow (1)$. Follows from Proposition~\ref{prop:saturated-neut}.
	
	$(1) \Rightarrow (3)$. Since $G$ is not neutral, we may fix a non-neutral $j$-gerbe $\cV \to \cG$ over a field $k$. In particular, $k$ is infinite by Lemma~\ref{lem:compression-neutral}

	By Proposition~\ref{prop:rhoneut-neut} $\bar{G}$ is not $\rho$-neutral, hence it is cyclic of order prime with $\cha k_{0}$ by Theorem~\ref{thm:pgl2}. Since the kernel of $G \to \bar{G}$ is central, it follows that $G$ is abelian. Up to conjugation, we may then assume that
	\[\bar{G} = \left<
		\left( \begin{array}{cc}
		1 	& 	0  			\\
		0	&	\zeta_{n}	\\
		\end{array} \right)
		\right>
	\]
	for some $n$ prime with $\cha k_{0}$, and
	\[G = \left<
		x = \left( \begin{array}{cc}
		\zeta_{m} 	& 	0  			\\
		0 			&	\zeta_{m} 	\\
		\end{array} \right),
		y = \lambda \cdot \left( \begin{array}{cc}
		1 	& 	0  					\\
		0 	&	\zeta_{n}	 \\
		\end{array} \right)
		\right>\]
	where $m$ is the cardinality of the kernel of $G \to \bar{G}$, and $\lambda$ is a root of $1$. Clearly, $m$ must be prime with $\cha k_{0}$, because a finite subgroup of $K^{*} \simeq \ker(\GL_{2}(K) \to \PGL_{2}(K))$ has order prime with $\cha k_{0}$.
	
	Let us prove that $|G|$ is even. The coarse moduli space $\bP$ of $\PP(\cV)$ is a Brauer--Severi variety of dimension $1$, hence there exists $k'/k$ with $[k':k] \le 2$ such that $\bP_{k'} \simeq \PP^{1}_{k'}$, and $\cG(k') \neq \emptyset$ by Lemma~\ref{lem:compression-neutral}. Since $G$ is abelian, if $|G|$ is odd then $\cG$ is neutral by Lemmas~\ref{lem:index-primes} and \ref{lem:index-neutral}, giving a contradiction.
	
	Assume first that $n = 1$. Then $m = |G|$ is even, and $G$ has the desired form.
	
	Assume $n \neq 1$. By Lemma~\ref{lem:diag-hyp} and Remark~\ref{rmk:norm-hyp}, the extended normalizer $\tilde{N}_{G} \subseteq \GL_{2,K} \rtimes \gal(K/k_{0})$ is $N_{G} \rtimes \gal(K/k_{0})$, where $N_{G} \subseteq \GL_{2,K}$ is the normalizer.
	
	Notice that $N_{G}$ permutes the eigenspaces of $G$. Since $n \neq 1$, there are two $1$-dimensional eigenspaces. Hence, by Lemma~\ref{lem:permutation} there are two cases: either $N_{G} = \GG_{m}^{2}$, or $N_{G} = \GG_{m}^{2} \rtimes C_{2}$, where $C_{2}$ is the cyclic group generated by the matrix
		\[z = \left( \begin{array}{cc}
		0 	& 	1	\\
		1	&	0	 \\
		\end{array} \right).
		\]
	If $N_{G} = \GG_{m}^{2}$, then $G$ is neutral by Theorem~\ref{thm:diag-neutral}, giving a contradiction. Hence, we have $N_{G} = \GG_{m}^{2} \rtimes C_{2}$.
	
	Since $z$ normalizes $G$, then
	\[y \cdot z \cdot y \cdot z^{-1} = \lambda^{2} \cdot \zeta_{n} \Id\]
	is an element of $G$ mapping to the identity, hence $\lambda^{2} \cdot \zeta_{n} \in \left< \zeta_{m} \right>$ and $\lambda \cdot \zeta_{2n} \in \left< \zeta_{2m} \right>$. In particular, we can write
	\[\lambda = \zeta_{2n}^{-1} \cdot \zeta_{2m}^{a}\]
	for some $a$, that is
	
	\[G = \left<
		x = \left( \begin{array}{cc}
		\zeta_{m} 	& 	0  			\\
		0 			&	\zeta_{m} 	\\
		\end{array} \right),
		y = \zeta_{2m}^{a} \cdot \left( \begin{array}{cc}
		\zeta_{2n}^{-1} 	& 	0  					\\
		0 					&	\zeta_{2n}	 \\
		\end{array} \right)
		\right>.\]
	If $n$ is odd, $y^{2} \cdot x^{-a}$ maps to a generator of $\bar{G}$, hence we have
	\[G = \left<
		\left( \begin{array}{cc}
		\zeta_{m} 	& 	0  			\\
		0 			&	\zeta_{m} 	\\
		\end{array} \right),
		\left( \begin{array}{cc}
		\zeta_{n}^{-1} 	& 	0  					\\
		0 				&	\zeta_{n}	 \\
		\end{array} \right).
		\right>\]
	Since $|G|$ is even, then $m$ is even, hence $-\Id \in G$ and
	\[\left( \begin{array}{cc}
		\zeta_{2n}^{-1} 	& 	0  				\\
		0 					&	\zeta_{2n}	 	\\
	\end{array} \right) = 
	-\left( \begin{array}{cc}
		\zeta_{n}^{-1} 	& 	0  			\\
		0 				&	\zeta_{n}	\\
	\end{array} \right)^{\frac{n+1}{2}} \in G.\]
	This concludes the proof if $n$ is odd.
	
	Assume $n$ even. If $a$ is even as well, then $G$ clearly has the desired form. We are going to use Theorem~\ref{thm:diag-neutral} to show that, if $a$ is odd and $n$ is even, then $G$ is neutral, giving a contradiction and concluding the proof.
	
	The character group of $\GG_{m}^{2}/G$ inside of $X(\GG_{m}^{2}) = \ZZ^{2}$ is
	\[X(\GG_{m}^{2}/G) = \left\{(\alpha, \beta) \in \ZZ^{2} \mid \zeta_{m}^{\alpha + \beta} = \zeta_{2n}^{\beta - \alpha} \cdot \zeta_{2m}^{(\alpha + \beta)a} = 1 \right\}.\]
	Notice that $-\Id = y^{n} \cdot x^{-an/2} \in G$, hence $m$ is even. Since both $m$ and $n$ are even and $a$ is odd, it is immediate to check that the vectors
	\[\left( \frac{m+n}{2}, \frac{m-n}{2} \right), \left( \frac{m-n}{2}, \frac{m+n}{2} \right) \in \ZZ^{2}\]
	are well defined elements of $X(\GG_{m}^{2}/G) \subseteq \ZZ^{2}$; they generate a permutation $C_{2}$-submodule of $\ZZ^{2}$ of index
	\[\det
	\left( \begin{array}{cc}
		\frac{m+n}{2} 	& 	\frac{m-n}{2}  	\\
		\frac{m-n}{2}	&	\frac{m+n}{2}	\\
	\end{array} \right) =
	\frac{(m+n)^{2} - (m-n)^{2}}{4} = mn.\]
	
	Since $X(\GG_{m}^{2}/G)$ has index $|G| = mn$ as well, the elements above are a basis of $X(\GG_{m}^{2}/G)$, which is then a permutation $C_{2}$-module. By Theorem~\ref{thm:diag-neutral}, $G$ is neutral. This is in contradiction with our assumption that $G$ is not neutral, hence $a$ cannot be odd.
	
	This concludes the proof for $\cha k_{0} \neq 2$. Assume $\cha k_{0} = 2$, and fix a $j$-gerbe $\cV \to \cG$ over a field $k$. Let us show that, if $\cG(k) = \emptyset$, we get a contradiction.
	
	By repeating the proof of the implication $(1) \Rightarrow (3)$ above, we get that $\cG(k) = \emptyset$ implies that $|G|$ is even. By Proposition~\ref{prop:rhoneut-neut}, $\bar{G}$ is not $\rho$-neutral, hence it is cyclic of odd order by Theorem~\ref{thm:pgl2}. It follows that the kernel $G \cap K^{*}$ of $G \to \bar{G}$ has even order: this gives a contradiction, because every finite subgroup of $K^{*}$ has order prime with $\cha k_{0}$.
\end{proof}

In \cite[Theorem 4]{bresciani-sing}, it is proved that the singularities not of type $\rR$ are of the form $\AA^{2}/G$, with
\[G = \left< \diag(\zeta_{2n}, \zeta_{2n}^{d})\right>\]
with $d \cong \pm 1$ modulo each prime power dividing $2n$. The two points of view are reconciled by the following lemma.

\begin{lemma}\label{lem:equiv-rep}
	Let $G \subseteq \GG_{m}^{2}(K) \subseteq \GL_{2}(K)$ be a finite, diagonal group of order prime with $\cha k_{0}$. The following are equivalent.
	\begin{enumerate}
		\item $G = \left< \diag(\zeta_{2m}, \zeta_{2m}), 	\diag(\zeta_{2n}, \zeta_{2n}^{-1}) \right>$ for some positive integers $m, n$.
		\item $G = \left< \diag(\zeta_{a}, 1), \diag(1, \zeta_{a}), \diag(\zeta_{2an}, \zeta_{2an}^{d}) \right>$
		for some positive integers $a, d, n$, with $d \cong \pm 1$ modulo each prime power dividing $2n$.
	\end{enumerate}
\end{lemma}

\begin{proof}
	$(1) \Rightarrow (2)$. Write $a = \gcd(m,n)$, $m = am'$, $n = an'$. Since $\gcd(m', n') = 1$, we may write $1 = xm' + yn'$ for some $x, y \in \ZZ$.
	
	Let $d = yn' - xm'$. Notice that $d = 1 - 2xm' \cong 1 \pmod{2m'}$ and $d = 2yn' - 1 \cong - 1 \pmod{2n'}$. In $G$, we have the three elements
	\[\diag(\zeta_{2m}, \zeta_{2m})^{y} \cdot \diag(\zeta_{2n}, \zeta_{2n}^{-1})^{x} = \diag(\zeta_{2am'n'}^{yn' + xm'}, \zeta_{2am'n'}^{yn' - xm'}) = \]
	\[ = \diag(\zeta_{2am'n'}, \zeta_{2am'n'}^{d}),\]
	\[\diag(\zeta_{2m}, \zeta_{2m})^{m'} \cdot \diag(\zeta_{2n}, \zeta_{2n}^{-1})^{n'} = \diag(\zeta_{a}, 1),\]
	\[\diag(\zeta_{2m}, \zeta_{2m})^{m'} \cdot \diag(\zeta_{2n}, \zeta_{2n}^{-1})^{-n'} = \diag(1, \zeta_{a}).\]
	The group generated by these three elements has order $2a^{2}m'n'= 2mn = |G|$, hence they generate $G$.
	
	$(2) \Rightarrow (1)$. Notice that $d \cong 1 \cong -1 \pmod{2}$. Let $2m'$ be the largest divisor of $2n$ such that $d \cong 1 \pmod{2m'}$, and $2n'$ the largest divisor such that $d \cong -1 \pmod{2n'}$. The hypothesis on $d$ implies that $\gcd(m', n') = 1$ and $n = m'n'$. In $G$, we have elements
	\[\diag(\zeta_{2am'}, \zeta_{2am'}^{d}) = \diag(\zeta_{2am'}, \zeta_{2am'}^{1 + 2\alpha m'}) = \diag(\zeta_{2am'}, \zeta_{2am'}) \cdot \diag(1, \zeta_{a}^{\alpha})\]
	\[\diag(\zeta_{2an'}, \zeta_{2an'}^{d}) = \diag(\zeta_{2an'}, \zeta_{2an'}^{-1 + 2\beta n'}) = \diag(\zeta_{2an'}, \zeta_{2an'}^{-1}) \cdot \diag(1, \zeta_{a}^{\beta})\]
	for some $\alpha, \beta$. The group generated by $\diag(\zeta_{2am'}, \zeta_{2am'})$, $\diag(\zeta_{2an'}, \zeta_{2an'}^{-1})$ has order $2a^{2}m'n' = 2a^{2}n = |G|$, hence they generate $G$.
\end{proof}

\subsection{$\PGL_{3}$}

Neutral and $\rho$-neutral subgroups of $\PGL_{3}$ in characteristic $0$ were classified in \cite{bresciani-plane}, up to minor details. Here, we rephrase the classification in the new language, and give details for some neutral groups for which $\rho$-neutrality was not checked.

For the remainder of Sections~\ref{sect:dim2} and \ref{sect:gl3}, we assume that $\cha k_{0} = 0$, in order to have a better grasp on finite subgroups of $\GL_{3}(K)$.

\begin{definition}\label{def:hessian}
	Write
	\[ M_{0}=\left( \begin{array}{ccc}
	1 & 0 &	0 \\
	0 & \zeta_{3} & 0 \\
	0 & 0 & \zeta_{3}^{2} \\
	\end{array} \right),~
	M_{1}=\left( \begin{array}{ccc}
	0 & 0 & 1 \\
	1 & 0 & 0 \\
	0 & 1 & 0
	\end{array} \right),~
	M_{2}=-\left( \begin{array}{ccc}
	1 & 0 & 0 \\
	0 & 0 & 1 \\
	0 & 1 & 0
	\end{array} \right),
	\]
	
	\[
	M_{3}=\frac{1}{2\zeta_{3} + 1}\left( \begin{array}{ccc}
	1 & 1 & 1 \\
	1 & \zeta_{3} & \zeta_{3}^{2} \\
	1 & \zeta_{3}^{2} & \zeta_{3}
	\end{array} \right),{}~
	M_{4}=\frac{1}{2\zeta_{3} + 1}\left( \begin{array}{ccc}
	1 & 1 & \zeta_{3} \\
	1 & \zeta_{3} & 1 \\
	\zeta_{3}^{2} & \zeta_{3} & \zeta_{3}
	\end{array} \right),\]
	
	\[
	M_{5}=\zeta_{9}^{-1}\left( \begin{array}{ccc}
	1 & 0 &	0 \\
	0 & 1 & 0 \\
	0 & 0 & \zeta_{3} \\
	\end{array} \right).\]
	Let $H_{i}$ be the subgroup of $\PGL_{3}(K)$ generated by $M_{0}, \dots, M_{i}$.
\end{definition}

The matrices $M_{0}, \dots, M_{5}$ are normalized so that $M_{i} \in \SL_{3}$ for every $i$. This is not important as long as we work in $\PGL_{3}$, but it will be useful later when we pass to $\GL_{3}$.

\begin{definition}\label{def:type-x}
	Consider the following list of finite subgroups of $\PGL_{3}(K)$.
	\begin{enumerate}[(a)]
		\item The abelian group $C_{a} \times C_{an}$ generated by $\diag(\zeta_{a},1,1)$, $\diag(1,\zeta_{a},1)$, $\diag(\zeta_{an},\zeta_{an}^{d},1)$ for some positive integers $a,n,d$ satisfying $d^{2}-d+1\cong 0\pmod{n}$ and $3\mid an$.
		
		\item The abelian group of order $2mn$ generated by $\diag(\zeta_{2m}, \zeta_{2m}, 1)$, $\diag(\zeta_{2n}, \zeta_{2n}^{-1}, 1)$ for some positive integers $m,n$.
		
		\item The group $H_{2} \simeq C_{3}^{2}\rtimes C_{2}$ of order $18$.
		
		\item The group $H_{3} \simeq C_{3}^{2}\rtimes C_{4}$ of order $36$.
		
		\item The abelian group $C_{a} \times C_{an}$ generated by $\diag(\zeta_{a},1,1)$, $\diag(1,\zeta_{a},1)$, $\diag(\zeta_{an},\zeta_{an}^{d},1)$ for positive integers $a,n,d$ satisfying $d^{2}-d+1\cong 0\pmod{n}$ and $3 \nmid an$.
	\end{enumerate}
	
	We say that $G \subseteq \PGL_{3}(K)$ is of type $(x)$ if it is conjugate to one of the groups listed in $(x)$.
\end{definition}

In $\PGL_{3}$, a new phenomenon occurs: whether a group is neutral or not depends on $k_{0}$.

\begin{theorem}\label{thm:pgl3}
	Assume $\cha k_{0} = 0$, and let $G \subseteq \PGL_{3}(K)$ be a finite subgroup. 
	\begin{itemize}
		\item If $G$ is of type $(a)$, $(b)$, or $(c)$, it is not neutral, regardless of $k_{0}$.
		\item If $G$ is of type $(d)$ and $\zeta_{3} \in k_{0}$, then $G$ is $\rho$-neutral. If $\zeta_{3} \notin k_{0}$, then $G$ is not neutral.
		\item If $G$ is of type $(e)$, then $G$ is neutral and not $\rho$-neutral, regardless of $k_{0}$. 
		\item If $G$ is not of type $(a) - (e)$, then $G$ is $\rho$-neutral.
	\end{itemize}
	
\end{theorem}

\subsubsection{Cohomological computations for subquotients of $H_{4}$}\label{sect:h4-cohom}

Before starting the case-by-case analysis, let us establish some facts about the cohomology of some subquotients of $H_{4}$, these will be used later. 

Notice that $M_{0}, \dots, M_{4}$ are defined over $\QQ(\zeta_{3})$. Let us write the Galois conjugates of these matrices with respect to the only non-trivial automorphism of $\QQ(\zeta_{3})$ over $\QQ$.

\[\bar{M}_{0} = M_{0}^{-1},~~~\bar{M}_{1} = M_{1},~~~\bar{M}_{2} = M_{2},\]
\[\bar{M}_{3} = M_{3}^{-1},~~~\bar{M}_{4} = \zeta_{3}^{2}M_{4}M_{3}\]

In particular, we get that $H_{0}, \dots,  H_{4}$ are Galois invariant and define group schemes $\fH_{0}, \dots, \fH_{4} \subset \PGL_{3}$ over $\QQ$.

Furthermore, by direct computation we have $M_{2}^{2} = \Id$ and
\[M_{3}^{2} = (\zeta_{3}M_{1}M_{4})^{2} = (\zeta_{3}M_{1}M_{4}M_{3})^{2} = M_{2},\]
hence $Q_{8} = \langle M_{3}, \zeta_{3}M_{1}M_{4} \rangle$ is a Galois-invariant quaternion group defining a group scheme $\fQ_{8} \subset \SL_{3}$. Since the kernel of $\SL_{3} \to \PGL_{3}$ has order $3$, the composition $\fQ_{8} \subset \SL_{3} \to \PGL_{3}$ is injective and identifies $\fQ_{8}$ with a subgroup of $\fH_{4}$. The quotient of $\fQ_{8}$ by its center $\langle M_{2} \rangle$ is a twisted form $\fKl$ of the Klein group
\[1 \to C_{2} \to \fQ_{8} \to \fKl \to 1.\]

This twisted Klein group identifies with $\fH_{4}/\fH_{2}$ as well, hence we get a factorization
\[\fQ_{8} \subset \fH_{4} \to \fH_{4}/\fH_{2} = \fKl.\] 
The Galois action of $\gal(\QQ(\zeta_{3})/\QQ)$ on the twisted Klein group fixes the identity and an order $2$ element (the class of $M_{3}$), and swaps the other two elements. Hence, there is a short exact sequence
\[0 \to C_{2} \to \fKl \to C_{2} \to 0\]
where $C_{2} \subset \fKl$ is generated by $M_{3}$. 

\begin{lemma}\label{lem:newklein}
	Let $k$ be any field of characteristic $0$, and $E$ be the étale algebra $E = k[s]/(s^{2} + s + 1)$ ($E$ is a field if and only if $\zeta_{3} \notin k$).
	
	We have a natural isomorphism $\H^{1}(k, \fKl) = E^{*}/E^{*2}$, and the long exact sequence in cohomology
	\[\H^{1}(k, C_{2}) \to \H^{1}(k, \fKl) \to \H^{1}(k, C_{2})\]
	\[k^{*}/k^{*2} \to E^{*}/E^{*2} \to k^{*}/k^{*2}\]
	is induced by the inclusion $k^{*} \subset E^{*}$ on the left and by the norm $\rN_{E/k} : E^{*} \to k^{*}$ on the right.
\end{lemma}

\begin{proof}
	The twisted Klein group is the module induced by $C_{2}$ along the extension $E/k$. The statement then follows from \cite[\S 2.5]{serre}.
\end{proof}

\begin{proposition}\label{prop:q8}
	Let $k$ be any field of characteristic $0$, $E = k[s]/(s^{2} + s + 1)$, $\H^{1}(k, \fKl) = E^{*}/E^{*2}$ as in Lemma~\ref{lem:newklein}. Write $\tr_{E/k}$ for the trace.
	
	Consider the $3$-dimensional $k$-vector space $E \oplus k$. If
	\[\delta: \H^{1}(k, \fKl) = E^{*}/E^{*2} \to \H^{2}(k, C_{2}) = \br(k)[2]\]
	is the connecting map in non-abelian cohomology corresponding to 
	\[1 \to C_{2} \to \fQ_{8} \to \fKl \to 1,\]
	then $\delta([e])$ is the Brauer class corresponding to the conic
	\[\tr_{E/k}(ew^2) + z^{2} = 0,\]
	in $\PP(E \oplus k)$ with coordinates $(w,z) \in E \oplus k$.
\end{proposition}

\begin{proof}
	Remember that $\fQ_{8} \subset \PGL_{3}$ is a twisted form of the quaternion group generated by $J = M_{3}$, $K = \zeta_{3}M_{1}M_{4}$ satisfying $J^{2} = K^{2} = (JK)^{2} = M_{2}$.
	
	Write $W = \{y = z\} \subset \AA^{3}$, it is the $-1$ eigenspace of $J^{2}=K^{2}=M_{2}$. Since $J$, $K$ commute with $M_{2}$, they map the eigenspace $W$ to itself. Furthermore, the vector $(0,1,-1)$ is an eigenvector of both $J$ and $K$ with eigenvalue $1$. If we write $\AA^{3} = W \oplus \langle (0,1,-1) \rangle$, this implies $\fQ_{8} \subset \SL(W) \subset \SL_{3}$.
	
	Using the basis $e_{0} = (1,0,0)$, $e_{1} = (0,1,1)$ of $W$, identify $\SL(W) \simeq \SL_{2}$. Write $\rho$ for the induced embedding $\fQ_{8} \hookrightarrow \SL_{2}$, we have
	\[\rho(J) = \frac{1}{\sqrt{-3}}
	\begin{pmatrix}
	1&2\\
	1&-1
	\end{pmatrix},
	\]
	\[ \rho(K)
	=
	\frac{1}{\sqrt{-3}}
	\begin{pmatrix}
	1&2\zeta_{3}^2\\
	\zeta_{3}&-1
	\end{pmatrix},
	\]
	\[ \rho(J^{2}) = \rho(K^{2}) = \rho((JK)^{2})
	=
	\begin{pmatrix}
	-1&0\\
	0&-1
	\end{pmatrix}.
	\]
	
	Projectivization gives an embedding $\fQ_{8}/C_{2} = \fKl \subset \PGL_{2}$, hence we get a commutative diagram of short exact sequences
	\[\begin{tikzcd}
		1\rar	&	C_{2}\rar\dar[equal]	&	\fQ_{8}\rar\dar		&	\fKl\rar\dar	&	1	\\
		1\rar	&	C_{2}\rar				&	\SL_{2}\rar			&	\PGL_{2}\rar	&	1.
	\end{tikzcd}\]
	Since the connecting map $\H^{1}(k,\PGL_{2}) \to \H^{2}(k,C_{2}) = \br(k)[2]$ maps a $\PGL_{2}$-torsor $\tau$ to the Brauer class of the twist of $\PP^{1}$ by $\tau$ \cite[Theorems 4.4.5 and 5.2.1]{gille-szamuely}, we get that the connecting map $\H^{1}(k, \fKl) \to \br(k)[2]$ does the same with $\fKl$-torsors. Let us compute the twist of a class $[e] \in \H^{1}(k, \fKl) = E^{*}/E^{*2}$ explicitly.
	
	Consider the base change $\PP(E \oplus k)_{\bar{k}}$. The natural identification $E \otimes \bar{k} \simeq \bar{k} \oplus \bar{k}$ maps $w \otimes a \mapsto (w(\zeta_{3})a, w(\zeta_{3}^{2})a)$ (recall that $E = k[s]/(s^{2} + s + 1)$). The action of $\gal(\bar{k}/k)$ on $E \otimes \bar{k}$ turns into the action on $\bar{k} \oplus \bar{k}$ which swaps the coordinates whenever the Galois automorphism swaps $\zeta_{3}$ and $\zeta_{3}^{2}$. If $X,Y,Z$ are the coordinates on $\bar{k} \oplus \bar{k} \oplus \bar{k}$, the coordinates $(w, z)$ on $E \oplus k$ map to $X = w(\zeta_{3})$, $Y = w(\zeta_{3}^{2})$, $Z = z$.
	
	Over $\bar{k}$, consider the non-standard Veronese embedding
	\[\nu:\PP(W)_{\bar{k}}\longrightarrow \PP(E \oplus k)_{\bar{k}} \simeq \PP(\bar{k} \oplus \bar{k} \oplus \bar{k}),\]
	given by
	\[\nu([p:q])=[X:Y:Z],\]
	where
	\[
	\begin{aligned}
	X = -\zeta_{3} p^2+2pq+2\zeta_{3}^2q^2,\\
	Y = -\zeta_{3}^2p^2+2pq+2\zeta_{3} q^2,\\
	Z = -p^2+2pq+2q^2.
	\end{aligned}
	\]
	Since $\gal(\bar{k}/k)$ swaps $X$ and $Y$ whenever $\zeta_{3} \mapsto \zeta_{3}^{2}$, the map $\nu$ is Galois equivariant and descends to a map $\PP(W) \to \PP(E \oplus k)$.
	
	One checks directly that $\nu$ identifies $\PP(W)$ with the conic $\rC_{0}$ with equation
	\[X^2+Y^2+Z^2=0.\]
	The action of $\fKl_{\bar{k}}$ on $\PP(E \oplus k)_{\bar{k}}$ defined by
	\[[K]:[X:Y:Z]\longmapsto[-X:Y:Z],\]
	\[[JK]:[X:Y:Z]\longmapsto[X:-Y:Z]\]
	is Galois equivariant, because $\gal(\bar{k})$ swaps both $K$, $JK$ and $X$, $Y$ whenever $\zeta_{3} \mapsto \zeta_{3}^{2}$. It descends to an action of $\fKl$ on $\PP(E \oplus k)$ such that $\nu$ is $\fKl$-equivariant.
	
	Now fix a class $[e] \in \H^{1}(k,\fKl) = E^{*}/E^{*2}$, we want to compute the twist of $\PP(W) \subset \PP(E \oplus k)$. There exist $\alpha, \beta \in \bar{k}^{*}$ such that $\alpha^{2} = e(\zeta_{3})$, $\beta^{2} = e(\zeta_{3}^{2})$, then $[e]$ corresponds to the cocycle $\chi(g) = \left(\frac{g(\alpha)}{\alpha}, \frac{g(\beta)}{\beta}\right) \in \{\pm1\}^{2}$ when $g(\zeta_{3}) = \zeta_{3}$, and $\chi(g) = \left(\frac{g(\alpha)}{\beta}, \frac{g(\beta)}{\alpha}\right) \in \{\pm1\}^{2}$ when $g(\zeta_{3}) = \zeta_{3}^{2}$.
	
	The twisted Galois action on $\PP(E \oplus k)_{\bar{k}}$ maps $[X:Y:Z]$ to $\left[\frac{g(\alpha)}{\alpha} g(X):\frac{g(\beta)}{\beta}g(Y):g(Z)\right]$ when $g(\zeta_{3}) = \zeta_{3}$, and to $\left[\frac{g(\beta)}{\alpha}g(Y):\frac{g(\alpha)}{\beta} g(X):g(Z)\right]$ otherwise.
	
	The new coordinates $X' = \frac{X}{\alpha}$, $Y' = \frac{Y}{\beta}$, $Z' = Z$ satisfy
	\[g \ast X' = g(X'), g \ast Y' = g(Y'), g \ast Z' = g(Z')\]
	when $g(\zeta_{3}) = \zeta_{3}$ and 
	\[g \ast X' = g(Y'), g \ast Y' = g(X'), g \ast Z' = g(Z')\]
	otherwise. Hence, $\PP(E \oplus k)_{\bar{k}}$ under the new Galois action descends to a \emph{new copy} of $\PP(E \oplus k)$. The twist of $\PP(W)$ by $e$ is the descent $\rC_{e}$ of $\rC_{0}$ in the new copy of $\PP(E \oplus k)$. In coordinates, $\rC_{0}$ is given by
	\[X^{2} + Y^{2} + Z^{2} = 0,\]
	that is
	\[(\alpha X')^{2} + (\beta Y')^{2} + Z'^{2} = 0,\]
	\[e(\zeta_{3})X'^{2} + e(\zeta_{3}^{2}) Y'^{2} + Z'^{2} = 0.\]
	In the new copy of $\PP(E \oplus k)$, we have $X' = w(\zeta_{3})$, $Y' = w(\zeta_{3}^{2})$, $Z' = z$, hence $\rC_{e}$ over $k$ is defined by
	\[e(\zeta_{3})w(\zeta_{3})^{2} + e(\zeta_{3}^{2})w(\zeta_{3}^{2})^{2} + z^{2} = 0,\]
	\[\tr_{E/k}(ew^2)+z^2=0\]
	in $\PP(E\oplus k)$.
\end{proof}

\begin{corollary}\label{cor:klq8}
	The map $\H^{1}(k, \fQ_{8} \times C_{2}) \to \H^{1}(k, \fKl)$ is surjective, where $\fQ_{8} \to \fKl$ is the projection and $C_{2} \to \fKl$ is the embedding given above.
	
	As a consequence, the images of $\H^{1}(k, \fQ_{8}) \to \H^{1}(k, C_{2})$ and $\H^{1}(k, \fKl) \to \H^{1}(k, C_{2})$ coincide.
\end{corollary}

\begin{proof}
	Write $E = k[s]/(s^{2} + s + 1)$ as in Lemma~\ref{lem:newklein}, and choose $[e] \in \H^{1}(k, \fKl) = E^{*}/E^{*2}$. Since $\{1\} \times C_{2}$ is central in both $\fQ_{8} \times C_{2}$ and $\fKl$, then $\H^{1}(k,\{1\} \times C_{2}) = k^{*}/k^{*2}$ acts on both $\H^{1}(k, \fQ_{8} \times C_{2})$, $\H^{1}(k,\fKl)$, and $\H^{1}(k, \fQ_{8} \times C_{2}) \to \H^{1}(k,\fKl)$ is equivariant. This means that, if $a \in k^{*}$, then $[e]$ lifts to $\fQ_{8} \times C_{2}$ if and only if $[ae]$ lifts to $\fQ_{8} \times C_{2}$. 
	
	By Proposition~\ref{prop:q8}, $[e]$ lifts to $\H^{1}(k,\fQ_{8})$ if and only if the conic $\tr_{E/k}(ew^{2}) + z^{2} = 0$ in $\PP(E \oplus k)$ is split. If $\tr_{E/k}(e) = 0$, then $(w, z) = (1, 0)$ is a $k$-point of the conic. Otherwise, by the above we may replace $e$ with $-e/\tr_{E/k}(e)$, and assume that $\tr_{E/k}(e) = -1$. In this case, $(w, z) = (1, 1)$ is a $k$-point of the conic. This concludes the first part of the proof.
	
	By the above, the image of $\H^{1}(k, \fKl) \to \H^{1}(k, C_{2})$ is equal to the image of the composition $\H^{1}(k, \fQ_{8} \times C_{2}) \to \H^{1}(k, \fKl) \to \H^{1}(k, C_{2})$. Given that the composition $\{1\} \times C_{2} \to \fKl \to C_{2}$ is trivial, this coincides with the image of $\H^{1}(k, \fQ_{8}) \to \H^{1}(k, C_{2})$.
\end{proof}

\begin{corollary}\label{cor:q8surj}
	If $\zeta_{3} \in k_{0}$, $\H^{1}(k,\fQ_{8}) \to \H^{1}(k, C_{2})$ is surjective for every extension $k/k_{0}$.
	
	If $\zeta_{3} \notin k_{0}$, $\H^{1}(k,\fKl) \to \H^{1}(k, C_{2})$ is not surjective for $k = k_{0}(t)$
\end{corollary}

\begin{proof}
	Assume $\zeta_{3} \in k$, then $E = k[s]/(s^{2} + s + 1) \simeq k \oplus k$. By Corollary~\ref{cor:klq8}, it is enough to show that $\H^{1}(k,\fKl) \to \H^{1}(k, C_{2})$ is surjective. By Lemma~\ref{lem:newklein}, this map is just the multiplication $(k^{*}/k^{*2})^{2} \to k^{*}/k^{*2}$, which is obviously surjective.
	
	Assume now that $\zeta_{3} \notin k_{0}$ and write $k = k_{0}(t)$.	Let $v:k(\zeta_{3})^{*} \to \ZZ$ be the $t$-adic valuation, and $\sigma$ the only non-trivial Galois automorphism of $k(\zeta_{3})/k$. Clearly, $v(\sigma(f)) = v(f)$ for every $f \in k(\zeta_{3})^{*}$, hence $v(\rN_{k(\zeta_{3})/k}(f)) = v(f \sigma(f)) = v(f^{2}) = 2v(f)$ is even. On the other hand, if $g \in k^{*}$ is such that $[g] = [t]$, then $v(g)$ is odd, hence $[\rN_{k(\zeta_{3})/k}(f)] \neq [t]$ for every $f \in k(\zeta_{3})^{*}$.
\end{proof}

Let us now prove Theorem~\ref{thm:pgl3}. 

\subsubsection{$G$ of type $(a)$, $(b)$, $(c)$} Non-neutral $j$-gerbes are constructed over $\CC((s))((t))$ in \cite[\S 4.11, \S 4.12, \S 4.13]{bresciani-plane} respectively. The equivalence between type $(b)$ and the case treated in \cite[\S 4.12]{bresciani-plane} is proved in Lemma~\ref{lem:equiv-rep}. The field $\CC((s))((t))$ can be replaced by $K((s))((t))$ because the argument works for any algebraically closed field of characteristic $0$, hence non-neutral $j$-gerbes exist regardless of $k_{0}$.

\subsubsection{$G$ of type $(d)$}\label{sect:type-d} Up to conjugation, $G = H_{3}$. A non-neutral $H_{3}$-gerbe is constructed over $\RR$ in \cite[\S 4.14]{bresciani-plane}. If $\zeta_{12} \in k_{0}$, the fact that $H_{3}$ is $\rho$-neutral is proved in \cite[\S 4.6]{bresciani-plane}.

Let us strengthen these facts, and show that if $\zeta_{3} \in k_{0}$ then $H_{3}$ is $\rho$-neutral, whereas if $\zeta_{3} \notin k_{0}$ then $H_{3}$ is not neutral.

Assume first that $\zeta_{3} \in k_{0}$. Let $\tilde{H}_{3} = \left< \zeta_{3}\Id, M_{0}, M_{1}, M_{2}, M_{3}\right> \subset \GL_{3}(K)$ be the inverse image of $H_{3}$ in $\SL_{3}(K)$. Since $\zeta_{3} \in k_{0}$, $\tilde{H}_{3}$ has a natural structure of constant group scheme over $k_{0}$. The normalizer of $\tilde{H}_{3}$ is the inverse of $H_{4}$ in $\GL_{3}$, because $H_{4}$ is the normalizer of $H_{3}$. The inverse image $\tilde{H}_{4} = \left< \zeta_{3}\Id, M_{0}, M_{1}, M_{2}, M_{3}, M_{4}\right>$ of $H_{4}$ in $\SL_{3}$ is then a gate for $\tilde{H}_{3}$ thanks to Lemma~\ref{lem:gate-special}. Since $\tilde{H}_{3}$ is the inverse image of $H_{3}$ in $\tilde{H}_{4}$, we get that $\tilde{H}_{3}$ is saturated. Thanks to Proposition~\ref{prop:saturated-neut}, $H_{3}$ is $\rho$-neutral if and only if $\tilde{H}_{3}$ is neutral. 

By Theorem~\ref{thm:coh-crit}, $\tilde{H}_{3}$ is neutral if and only if $\H^{1}(k, \tilde{H}_{4}) \to \H^{1}(k, \tilde{H}_{4}/\tilde{H}_{3})$ is surjective for every extension $k$ of $k_{0}$. Since $\tilde{H}_{4}/\tilde{H}_{3} \simeq C_{2}$ and $\fQ_{8} \subset \tilde{H}_{4}$, this follows from Corollary~\ref{cor:q8surj}.

Assume now that $\zeta_{3} \notin k_{0}$. Since $\fH_{4}$ is the normalizer of $\fH_{3}$ and there is a factorization
\[\fH_{4} \to \fKl \to \fH_{4}/\fH_{3} = C_{2},\]
then $H_{3}$ is not neutral by Theorem~\ref{thm:norm} and Corollary~\ref{cor:q8surj}.

\subsubsection{$G$ of type $(e)$}

If $G$ is of type $(e)$, it is neutral by \cite[Theorem 4.1]{bresciani-plane}, but $\rho$-neutrality has not been studied yet. Let us do it.

\begin{lemma}\label{lem:perm}
	Let $P$ be a pro-finite group with a surjective, continuous homomorphism $P \to C_{3}$, and $M \subseteq \ZZ[C_{3}] = \ZZ^{3}$ be the submodule of vectors $(x, y, z)$ with $x + y + z \cong 0 \pmod{3}$. Then $M$ is not a direct summand of a permutation $P$-module.
\end{lemma}

\begin{proof}
	We have that $M = I \oplus \ZZ (1, 1, 1)$, where $I \subseteq \ZZ^{3}$ is the augmentation ideal defined by $x + y + z = 0$. Hence, it is sufficient to prove that $I$ is not a direct summand of a permutation module. By Shapiro's lemma, it is enough to prove that $\H^{1}(P, I)$ is not trivial.
	
	We have a long exact sequence in group cohomology
	\[\H^{0}(P, \ZZ[C_{3}]) \to \H^{0}(P, \ZZ) \to \H^{1}(P, I) \to \H^{1}(P, \ZZ[C_{3}])\]
	where $\H^{1}(P, \ZZ[C_{3}]) = 0$ by Shapiro's lemma again. It follows that $\H^{1}(P, I) = \ZZ/3$.
\end{proof}

\begin{proposition}\label{prop:type-e}
	Let $a,n,d$ positive integers with $d^{2} - d + 1 \cong 0 \pmod{n}$, $3 \nmid an$. The group
	\[ \bar{G} = \left< \diag(\zeta_{a},1,1), \diag(1,\zeta_{a},1), \diag(\zeta_{an},\zeta_{an}^{d},1) \right>\]
	is not $\rho$-neutral, regardless of $k_{0}$.
\end{proposition}

\begin{proof}
	If $an = 1$, then $\bar{G}$ is trivial. It is not $\rho$-neutral because $j$-gerbes correspond to Brauer--Severi surfaces, and non-trivial Brauer--Severi surfaces always exist over some extension of $k_{0}$. Let us show this. Over $k = K((s))((t))$, consider the inverse image $\tilde{H}_{1}$ of $H_{1}$ in $\SL_{3}$, it is a non-abelian group with an extension
	\[1 \to C_{3} \to \tilde{H}_{1} \to C_{3}^{2} \to 1.\]
	The natural surjective homomorphism $\gal(\bar{k}/k) = \hat{\ZZ}^{2} \to C_{3}^{2}$ does not lift to $\tilde{H}_{1}$ because it would give a surjective homomorphism $\hat{\ZZ}^{2} \to \tilde{H}_{1}$, which is impossible since $\tilde{H}_{1}$ is not abelian. Hence, we get a non-trivial class $\tau \in \H^{2}(k, C_{3}) = \br(k)[3]$. Let $k'/k$ be the field fixed by $3\hat{\ZZ}^{2} \subset \gal(\bar{k}/k)$, it splits $\tau$ because the composition $3\hat{\ZZ}^{2} \subset \hat{\ZZ}^{2} \to C_{3}^{2}$ is trivial. In particular, $\tau$ has index either $3$ or $9$. If the index is $3$, then $\tau$ is represented by a non-trivial Brauer--Severi surface. If the index is $9$, then $\tau_{k''}$ is represented by a non-trivial Brauer--Severi surface, where $k''/k$ is any subextension of $k'/k$ of degree $3$.
	
	Assume $an \neq 1$. By Proposition~\ref{prop:saturated-neut}, the statement is equivalent to the fact that the inverse image
	\[ G = \left< \diag(\zeta_{3a}^{2}, \zeta_{3a}^{-1}, \zeta_{3a}^{-1}),~\diag(\zeta_{3a}^{-1}, \zeta_{3a}^{2}, \zeta_{3a}^{-1}),~\diag(\zeta_{3an}^{2-d}, \zeta_{3an}^{2d-1}, \zeta_{3an}^{-d-1})\right>\]
	in $\SL_{3}$ is not neutral.
	
	Notice that $C_{3}$ normalizes $\bar{G}$, where $C_{3}$ is generated by the permutation matrix $M_{1}$ cycling the coordinates. In fact, if we call $x, y, z$ respectively the generators of $\bar{G}$, we have
	\[M_{1} \cdot x \cdot M_{1}^{-1} = y,~~ M_{1} \cdot y \cdot M_{1}^{-1} = x^{-1} \cdot y^{-1},~~ M_{1} \cdot z \cdot M_{1}^{-1} = z^{-d} \cdot x^{\alpha} \cdot y^{\beta} \]
	for some $\alpha$, $\beta$. It follows that $M_{1}$ normalizes $G$ as well, and hence the normalizer of $G$ contains $\GG_{m}^{3} \rtimes C_{3}$.
	
	Choose $k$ such that there exists a non-zero cohomology class $\tau \in \H^{1}(k, C_{3})$, i.e. a surjective homomorphism $\gal(\bar{k}/k) \to C_{3}$. For instance, we may take $k = K((t))$ so that $\H^{1}(K((t)), C_{3}) = \hom(\hat{\ZZ}, C_{3}) = C_{3}$. 
	
	Let $T$ be the twist of $\GG_{m}^{3}$ by $\tau$. Notice that the embedding $\mu_{3} \subseteq \GG_{m}^{3}$ is fixed by $C_{3}$, hence we have an induced embedding $\mu_{3} \subseteq T$.
	
	The character group $X(T) \simeq \ZZ[C_{3}]$ is a $\gal(\bar{k}/k)$-module where $\gal(\bar{k}/k)$ acts via the surjective homomorphism $\gal(\bar{k}/k) \to C_{3}$. By Lemma~\ref{lem:perm}, the character group $X(T/\mu_{3}) \subseteq X(T) \simeq \ZZ[C_{3}]$ of $T/\mu_{3}$ is not a direct summand of a permutation $\gal(\bar{k}/k)$-module. By \cite[Theorem 18]{huruguen}, $T/\mu_{3}$ is not special, hence up to enlarging $k$ we may assume that $\H^{1}(k, T/\mu_{3}) \neq 0$.
	
	Let $\fG'$ be the twist of $\fG$ by $\tau$, where $\fG$ is the standard group scheme structure on $G$. The twist of $\GL_{3}$ by $\tau$ is again $\GL_{3}$: in fact, this is equivalent to twisting by the image of $\tau$ in $\H^{1}(k, \GL_{3})$, which is trivial. Hence, we have an embedding $T \subseteq \GL_{3}$ which identifies $\fG'(\bar{k})$ with a conjugate of $G$.
	
	By Lemma~\ref{lem:not-neutral}, to show that $G$ is not neutral it is sufficient to show that $\H^{1}(k, T) \to \H^{1}(k, T/\fG')$ is not surjective. Notice that the character group $X(T)$ of $T$ is a permutation Galois module, because $T$ is obtained as a twist by a torsor which permutes the coordinates, hence $\H^{1}(k, T) = 0$ by \cite[Theorem 18]{huruguen}. 
	
	It is then enough to prove that $\H^{1}(k, T/\fG') \neq 0$. First, notice that $\H^{1}(k, T/\mu_{3})$ is $3$-torsion, because we have an exact sequence
	\[\H^{1}(k, T) = 0 \to \H^{1}(k, T/\mu_{3}) \to \H^{2}(k, \mu_{3}),\]
	and $\H^{2}(k, \mu_{3})$ is $3$-torsion. Secondly, $\fG' / \mu_{3}$ has order $a^{2}n$, which is prime with $3$, and hence $\H^{1}(k, \fG'/\mu_{3})$ is torsion of exponent prime with $3$. It follows that $\H^{1}(k, \fG'/\mu_{3}) \to \H^{1}(k, T/\mu_{3})$ is the $0$ map, and hence
	\[\H^{1}(k, T/\mu_{3}) \to \H^{1}(k, T/\fG')\]
	is injective. Since $\H^{1}(k, T/\mu_{3})$ is not trivial, the same holds for $\H^{1}(k, T/\fG')$, and we conclude.
\end{proof}

\subsubsection{$G$ not of type $(a) - (e)$} In \cite[Theorem 4.1]{bresciani-plane}, it is proved that $G$ is $\rho$-neutral if it is not of type $(a) - (e)$, and it is not conjugate to $H_{1}$. To conclude, it is enough to prove that $H_{1}$ is $\rho$-neutral as well.

Let $\tilde{H}_{1} \subseteq \GL_{3}(K)$ be the group generated by $M_{0}$, $M_{1}$ as in Definition~\ref{def:hessian}, it coincides with the inverse image of $H_{1}$ in $\SL_{3}$. By Proposition~\ref{prop:saturated-neut}, $\tilde{H}_{1}$ is neutral if and only if $H_{1}$ is $\rho$-neutral. We may then conclude the proof of Theorem~\ref{thm:pgl3} by proving the following.

\begin{proposition}\label{prop:27}
	The group $\tilde{H}_{1}$ is saturated and neutral.
\end{proposition}

\begin{proof}
	The matrix presentation of $\tilde{H}_{1}$ identifies it with the geometric points of a group scheme $\tilde{\fH}_{1} \subseteq \GL_{3}$, because $M_{1}$ is Galois invariant and the only Galois conjugate of $M_{0}$ is $M_{0}^{2}$.
	
	Let $\fN$ be the normalizer of $\tilde{\fH}_{1}$, and $\tilde{\fH}_{5} = \fN \cap \SL_{3}$. Since $\tilde{\fH}_{5}$ is the kernel of the determinant $\fN \to \GG_{m}$, we have that $\tilde{\fH}_{5}$ is a normalizing gate for $\tilde{\fH}_{1}$ thanks to Lemma~\ref{lem:gate-special} and Proposition~\ref{prop:gate-norm}. Write $\tilde{H}_{5} = \tilde{\fH}_{5}(K)$.
	
	The image of $\tilde{\fH}_{5}$ in $\PGL_{3}$ is the normalizer of $H_{1}$. By \cite[Lemma 3.4]{bresciani-plane}, the normalizer of $H_{1}$ is the group $H_{5}$ as in Definition~\ref{def:hessian}. If $\pi: \tilde{H}_{5} \to H_{5}$ is the projection, then $\pi^{-1}(\pi(\tilde{H}_{1})) = \pi^{-1}(H_{1}) = \tilde{H}_{1}$, hence $\tilde{H}_{1}$ is saturated.
	
	Recall from \S\ref{sect:h4-cohom} that we have a $2$-Sylow subgroup scheme $\fQ_{8} \subset \tilde{H}_{5}$ generated by $M_{3}$, $\zeta_{3}M_{1}M_{4}$. Let $\mu_{9} \subseteq \tilde{H}_{5}$ be the cyclic subgroup generated by the matrix $M_{0}M_{5}^{-1}$, one can check by direct computation that $M_{0}M_{5}^{-1}$ normalizes $\fQ_{8}$. More abstractly, $\tilde{H}_{5}$ has $9$ $2$-Sylow subgroups, each one isomorphic to a quaternion group and containing a unique, distinct element of order $2$ (the $9$ conjugates of $M_{2}$). Since $M_{0}M_{5}^{-1}$ normalizes $M_{2} \in \fQ_{8}$, it must normalize $\fQ_{8}$ as well.
	
	We get a subgroup scheme $\fQ_{8} \rtimes \mu_{9} \subseteq \tilde{\fH}_{5}$, and $\tilde{\fH}_{5}/\tilde{\fH}_{1}$ identifies naturally with its quotient $\fQ_{8} \rtimes \mu_{3}$.
	
	By Theorem~\ref{thm:coh-crit}, $\tilde{H}_{1}$ is neutral if and only if $\H^{1}(k, \tilde{\fH}_{5}) \to \H^{1}(k, \tilde{\fH}_{5}/\tilde{\fH}_{1})$ is surjective for every extension $k$ of $k_{0}$. To check this, it is enough to prove that $\H^{1}(k, \fQ_{8} \rtimes \mu_{9}) \to \H^{1}(k, \fQ_{8} \rtimes \mu_{3})$ is surjective.
	
	The kernel of $\fQ_{8} \rtimes \mu_{9} \to \fQ_{8} \rtimes \mu_{3}$ is central. By non-abelian cohomology, it is enough to show that the connecting map $\H^{1}(k,\fQ_{8} \rtimes \mu_{3}) \to \H^{2}(k, \mu_{3})$ is trivial. This factorizes as $\H^{1}(k,\fQ_{8} \rtimes \mu_{3}) \to \H^{1}(k,\mu_{3}) \to \H^{2}(k, \mu_{3})$, and $\H^{1}(k,\mu_{3}) \to \H^{2}(k, \mu_{3})$ is trivial because $\H^{1}(k,\mu_{9}) \to \H^{1}(k, \mu_{3})$ is surjective. This concludes the proof of Proposition~\ref{prop:27}.
\end{proof}	

\section{The classification of neutral subgroups of $\GL_{3}$}\label{sect:gl3}

Recall we have defined $\SL_{n}^{(m)} \subseteq \GL_{n}$ as the subgroup of matrices $M$ with $\det M^{m} = 1$. Let us extend this definition.

\begin{definition}\label{def:slnm}
	Let $\underline{n} = (n_{1}, \dots, n_{s}) \in \ZZ^{s}_{>0}$, $\underline{m} = (m_{1}, \dots, m_{s}) \in \ZZ^{s}$ be vectors of integers with $n_{i} > 0$, write $n = n_{1} + \dots + n_{s}$.
	
	We define
	\[\SL_{\underline{n}}^{\underline{m}} \subseteq \GL_{n_{1}} \times \dots \times \GL_{n_{s}} \subseteq \GL_{n}\]
	as the subgroup of block matrices, with blocks $M_{1}, \dots, M_{s}$ of dimensions $n_{1}, \dots, n_{s}$, such that $\det M_{1}^{m_{1}} \cdot \dots \cdot \det M_{s}^{m_{s}} = 1$.
	
	Equivalently, $\SL_{\underline{n}}^{\underline{m}}$ is the kernel of the composition
	\[\GL_{n_{1}} \times \dots \times \GL_{n_{s}} \xrightarrow{\det} \GG_{m}^{s} \to \GG_{m},\]
	where $\GG_{m}^{s} \to \GG_{m}$ is the character $(x_{1}, \dots, x_{s}) \mapsto x_{1}^{m_{1}} \cdots x_{s}^{m_{s}}$.
\end{definition}

Assume $\underline{m} \neq (0, \dots, 0)$. Notice that $\SL_{\underline{n}}^{\underline{m}} \to (\GL_{n_{1}} \times \dots \times \GL_{n_{s}} )/\GG_{m} \subseteq \PGL_{n}$ is surjective if and only if it has finite kernel (the dimensions are equal), if and only if $m_{1}n_{1} + \dots + m_{s}n_{s} \neq 0$.

We refer to Definition~\ref{def:hessian} for the definitions of $M_{i} \in \PGL_{3}(K)$, $H_{i} \subseteq \PGL_{3}(K)$. Notice that $H_{i}$ is Galois invariant for every $i$, and hence corresponds to a well defined sub-group scheme of $\PGL_{3}$ \cite[\S 3.3]{bresciani-plane}.

\begin{theorem}\label{thm:gl3}
	Assume $\cha k_{0} = 0$, and let $G \subseteq \GL_{3}(K)$ be a finite subgroup with image $\bar{G} \subseteq \PGL_{3}(K)$. The following are equivalent.
	\begin{enumerate}
		\item $G$ is not neutral.
		\item $G$ is saturated and $\bar{G}$ is not $\rho$-neutral.
		\item $G$ is conjugate to one of the following groups.
			\begin{itemize}
				\item Given a positive integer $c$, the inverse image in $\SL_{3}^{(c)}$ of
				\[\left< \diag(\zeta_{a},1,1), \diag(1,\zeta_{a},1), \diag(\zeta_{an},\zeta_{an}^{d},1) \right> \subseteq \PGL_{3}(K).\]
				for some positive integers $a,d,n$ with $d^{2} - d + 1 \cong 0 \pmod{n}$.

				\item 
				
				Given integers $c_{1},c_{2}$ with $2c_{1} + c_{2} \neq 0$, the inverse image in $\SL_{(2,1)}^{(c_{1},c_{2})}$ of
				 \[\left< \diag(\zeta_{2m}, \zeta_{2m}, 1), \diag(\zeta_{2n}, \zeta_{2n}^{-1}, 1)\right> \subseteq \PGL_{3}(K)\]
				for some positive integers $m, n$.
				
				\item Given a positive integer $c$, the inverse image in $\SL_{3}^{(c)}$ of $H_{2}$.
				
				\item If $\zeta_{3} \notin k_{0}$, given a positive integer $c$, the inverse image in $\SL_{3}^{(c)}$ of $H_{3}$.
			\end{itemize}
	\end{enumerate}
\end{theorem}

{$(3) \Rightarrow (2)$.} If $G$ is one of the groups listed in $(3)$, then $\bar{G}$ is not $\rho$-neutral by Theorem~\ref{thm:pgl3}. Let us show that $G$ is saturated as well. This is clear for the first, third and fourth class. For the second class, the image $\SL_{(2,1)}^{(c_{1},c_{2})}$ in $\PGL_{3}$ is $\GL_{2}$, which contains the normalizer of $\bar{G}$. Hence, $\GL_{2}$ is a gate for $\bar{G}$ by Proposition~\ref{prop:gate-norm}, and $G$ is saturated in this case too.

{$(2) \Rightarrow (1)$.} If $G$ is saturated and not $\rho$-neutral, then $G$ is not neutral by Proposition~\ref{prop:saturated-neut}.

{$(1) \Rightarrow (3)$.} Assume that $G$ is not neutral, we want to show that it is conjugate to one of the groups listed in $(3)$. Fix a non-neutral $j$-gerbe $\cV \to \cG$.

By Proposition~\ref{prop:rhoneut-neut} and Theorem~\ref{thm:pgl3}, we have that $\bar{G}$ is of type $(a) - (e)$ as in Definition~\ref{def:type-x}. Let us subdivide the analysis by cases.

\subsection{Type $(a)$ or $(e)$}

Assume that $\bar{G}$ is either of type $(a)$ or $(e)$, i.e. up to conjugacy
\[\bar{G} = \left< \diag(\zeta_{a},1,1), \diag(1,\zeta_{a},1), \diag(\zeta_{an},\zeta_{an}^{d},1)\right>\]
for integers $a,n,d$ with $a,n \ge 1$ and satisfying $d^{2}-d+1\cong 0\pmod{n}$. In particular, $G$ is diagonal, and we may consider it as a group scheme with the structure given in Lemma~\ref{lem:diag-hyp}.

The proof is quite long. We subdivide it in several short steps.

\subsubsection{The singularities of $\PP^{2}/\bar{G}$ are of type $\rR$}

The singularities of $\PP^{2}/\bar{G}$ are of type $\frac{1}{n}(1, d)$ with $d^{2} \cong d - 1 \pmod{n}$. If they are not of type $\rR$, by \cite[Theorem~4]{bresciani-sing} we also have $d^{2} \cong 1 \pmod{n}$, hence $d \cong 2 \pmod{n}$ and $n \mid 3$, which is absurd since by \cite[Theorem~4]{bresciani-sing} again we also have $2 \mid n$. It follows that the singularities of $\PP^{2}/\bar{G}$ are of type $\rR$. If $\bP$ is the coarse moduli space of $\PP(\cV)$, we get that $\bP(k') \neq \emptyset \Leftrightarrow \cG(k') \neq \emptyset$ for every extension $k'/k$ thanks to Lemma~\ref{lem:compression-neutral}.

\subsubsection{The index of $\cG$ is $3$ and $3 \mid |G|$}

Let us show that the index of $\cG$ is $3$ and that $3 \mid |G|$. If $an = 1$,  i.e. $G = \left< \zeta_{m}\Id \right>$ for some $m \ge 1$, then $\bP$ is a Brauer--Severi surface, hence there exists $k'/k$ with $[k':k] \mid 3$ such that $\bP(k') \neq \emptyset$. If $an \neq 1$, the three axes are $1$-dimensional eigenspaces of $G$, their union is distinguished and descends to a $0$-cycle of degree $3$ on $\bP$, hence there exists $k'/k$ with $[k':k] \mid 3$ such that $\bP(k') \neq \emptyset$ in this case as well. 

In any case, this implies $\cG(k') \neq \emptyset$ by Lemma~\ref{lem:compression-neutral}. Since $\cG(k) = \emptyset$ by assumption, then $[k':k] = 3$ and $\cG$ has index $3$ by Lemma~\ref{lem:index-neutral}. It follows that $3 \mid |G|$ by Lemma~\ref{lem:index-primes}.

\subsubsection{Completing the proof if $an = 1$}

If $an = 1$, then $G = \left< \zeta_{c}\Id \right>$ for some $c \ge 1$, and $3 \mid c = |G|$, so $G$ has the desired form.

\subsubsection{Structure of $G$ and of the normalizer if $an \neq 1$}\label{sect:an-not-1}

Assume $an \neq 1$, and let $N_{G} \rtimes \gal(K/k_{0})$ be the extended normalizer (see Remark~\ref{rmk:norm-hyp}). Since $an \neq 1$, the three axes are $1$-dimensional eigenspaces of $G$, and $N_{G}$ permutes them. If the action of $N_{G}$ on the axes is not transitive, then at least one of them is distinguished, hence it descends to a rational point on $\bP$ of $\PP(\cV)$. This implies that $\cG(k) \neq \emptyset$, which is absurd, hence the action of $N_{G}$ on the axes must be transitive.

By Lemma~\ref{lem:permutation}, the normalizer is $\GG_{m}^{3} \rtimes S$ with $S$ either $C_{3}$ or $S_{3}$. Let $M_{1} \in N_{G}$ be as in Definition~\ref{def:hessian}.

Let $\left< w = \zeta_{c}\Id \right>$ be the kernel of $G \to \bar{G}$. Given the presentation of $\bar{G}$ above, in $G$ we have elements $x = \lambda \cdot \diag(\zeta_{3a}^{2}, \zeta_{3a}^{-1}, \zeta_{3a}^{-1})$, $z = \lambda' \cdot \diag(\zeta_{3an}^{2-d}, \zeta_{3an}^{2d-1}, \zeta_{3an}^{-d-1})$, and $y = M_{1} \cdot x \cdot M_{1}^{-1} = \lambda \cdot \diag(\zeta_{3a}^{-1}, \zeta_{3a}^{2}, \zeta_{3a}^{-1})$.

Furthermore, we have
\[\lambda^{3}\Id = x \cdot (M_{1} \cdot x \cdot M_{1}^{-1}) \cdot (M_{1}^{2} \cdot x \cdot M_{1}^{-2}) \in \left< w = \zeta_{c}\Id \right>,\]
\[\lambda'^{3}\Id = z \cdot (M_{1} \cdot z \cdot M_{1}^{-1}) \cdot (M_{1}^{2} \cdot z \cdot M_{1}^{-2}) \in \left< w = \zeta_{c}\Id \right>,\]
hence we may write $\lambda = \zeta_{3c}^{\alpha}$, $\lambda'=\zeta_{3c}^{\beta}$.

To sum up, we have
\[G = \left< w = \zeta_{c}\Id,~ x = \zeta_{3c}^{\alpha} \cdot \diag\left(\zeta_{3a}^{2}, \zeta_{3a}^{-1}, \zeta_{3a}^{-1}\right), \right.\]
\[\left. y = \zeta_{3c}^{\alpha} \cdot \diag \left(\zeta_{3a}^{-1}, \zeta_{3a}^{2}, \zeta_{3a}^{-1}\right),~ z = \zeta_{3c}^{\beta} \cdot \diag\left(\zeta_{3an}^{2-d}, \zeta_{3an}^{2d-1}, \zeta_{3an}^{-d-1}\right) \right>\]
and up to multiplying $x, y, z$ by a suitably power of $w$, we may assume that $0 \le \alpha, \beta \le 2$. 

The statement is then equivalent to $3 \mid c$ and $\alpha = \beta = 0$, which in turn is equivalent to $3 \mid c$ and $G \subseteq \SL_{3}^{(c/3)}$.

\subsubsection{Completing the proof if $3 \nmid an$}
If $3 \nmid an$, then $3 \mid c$ since $3 \mid |G|$ but $3 \nmid |\bar{G}|$. Furthermore, $x^{3} \cdot w^{-\alpha}$, $y^{3} \cdot w^{-\alpha}$, $z^{3} \cdot w^{-\beta}$ are elements of $\SL_{3}^{(c/3)}$ mapping to generators of $\bar{G}$, hence $G = \pi^{-1}(\bar{G}) \cap \SL_{3}^{(c/3)}$.

\subsubsection{Proof of $3 \mid c$ if $3 \mid an$}

Assume $3 \mid an$. Let us show that $3 \mid c$.

If $3 \mid a$, then

\[x^{a} \cdot w^{-a\alpha/3} = \zeta_{3c}^{a\alpha} \cdot \zeta_{3}^{2} \cdot \zeta_{3c}^{-a\alpha} \Id = \zeta_{3}^{2} \Id \in \left< w = \zeta_{c}\Id \right>,\]
hence $3 \mid c$.

Assume $3 \nmid a$, $3 \mid n$, write $n = 3^{\gamma}n'$ with $3 \nmid n'$. If, by contradiction, $3 \nmid c$, the $3$-Sylow $G_{3} \subseteq G$ is generated by
\[x^{ca} = y^{ca} = \zeta_{3}^{a\alpha} \cdot \diag(\zeta_{3}^{2}, \zeta_{3}^{2}, \zeta_{3}^{2})^{c} = \zeta_{3}^{a\alpha + 2c}\Id\]
and
\[z^{acn'} = \zeta_{3}^{an'\beta} \diag \left(\zeta_{3^{\gamma + 1}}^{2 - d}, \zeta_{3^{\gamma + 1}}^{2d - 1}, \zeta_{3^{\gamma + 1}}^{-d - 1} \right)^{c}.\]
Since $3 \nmid c$ and the kernel of $G \to \bar{G}$ is generated by $\zeta_{c}\Id$, then $\zeta_{3}^{a\alpha + 2c} = 1$. Notice that $d^{2} - d + 1 \cong 0 \pmod{n}$ and $3 \mid n$ implies $d \cong 2 \pmod{3}$, write $d = 2 + 3d'$, hence
\[z^{acn'} = \zeta_{3}^{an'\beta} \diag \left(\zeta_{3^{\gamma}}^{-d'}, \zeta_{3^{\gamma}}^{1 + 2d'}, \zeta_{3^{\gamma}}^{-1 - d'} \right)^{c}\]
Notice that $-d'$, $1 + 2d'$, $-1 - d'$ are pairwise different modulo $3$, and since $\gamma \ge 1$ then exactly one eigenvalue of $z^{acn'}$ is a $3^{\gamma - 1}$th root of $1$. This is in contradiction with the fact that $G_{3} = \left<z^{acn'}\right>$ is normalized by $M_{1}$, hence we must have $3 \mid c$.

\subsubsection{Reduction to $3$-groups}\label{sect:3-group}

We have to prove that $G \subseteq \SL_{3}^{(c/3)}$. Let us show that it is enough to do so in the case in which $|G|$ is a power of $3$.

The band $\cA$ of $\cG$ is a twisted form of $G$ over $k$, write $\cA = \cA_{3} \times \cA_{3}'$ where $|\cA_{3}|$ is a power of $3$ and $|\cA_{3}'|$ is prime with $3$. By \cite[Lemma 2.3]{bresciani-toric}, $\cG = \cG_{3} \times \cG_{3}'$, where $\cG_{3}$ is the induced $\cA_{3}$-gerbe and $\cG_{3}'$ is the induced $\cA_{3}'$-gerbe. Since $\cG$ has index $3$, then $\cG_{3}'$ is neutral by Lemma~\ref{lem:index-neutral}, it follows that $\cG_{3}$ is not neutral and has index $3$. Furthermore, by choosing a $k$-section of $\cG_{3}'$, we may pull-back $\cV$ to $\cG_{3}$, giving $\cG_{3}$ the structure of a $G_{3}$-gerbe where $G_{3} \subseteq G$ is the $3$-Sylow subgroup. Hence, we may reduce to the case in which $|G|$ is a power of $3$ thanks to the following lemma.

\begin{lemma}\label{lem:3-sylow}
	Let $G_{3} \subseteq G$ be the $3$-Sylow, and $c'$ the largest power of $3$ dividing $c$. Then $G_{3} \subseteq \SL_{3}^{(c'/3)}$ if and only if $G_{3} \subseteq \SL_{3}^{(c/3)}$, if and only if $G \subseteq \SL_{3}^{(c/3)}$.
\end{lemma}

\begin{proof}
	The first equivalence follows from the fact that, if $g \in \SL_{3}^{(c/3)}$ is an element whose order is a power of $3$, then $g \in \SL_{3}^{(c'/3)}$ as well. Let us work on the second equivalence.
	
	The ``if'' part is obvious. Assume $G_{3} \subseteq \SL_{3}^{(c/3)}$, and let $G_{3}' \subseteq G$ be the subgroup of elements of order prime with $3$. We already know that $G \subseteq \SL_{3}^{(c)}$, so $G_{3}' \subseteq \SL_{3}^{(c)}$. If $g$ has order prime with $3$, then $\det g^{c} = 1 \Leftrightarrow \det g^{c/3} = 1$, hence $G_{3}' \subseteq \SL_{3}^{(c/3)}$ as well. Since $G = G_{3} \times G_{3}'$, we conclude that $G \subseteq \SL_{3}^{(c/3)}$.
\end{proof}

\subsubsection{Proof of $G \subseteq \SL_{3}^{(c/3)}$ for $3$-groups}

So we may assume that $|G|$ is a power of $3$. Equivalently, $c$, $a$, $n$ are powers of $3$. Let us prove that $G \subseteq \SL_{3}^{(c/3)}$.

We want to use Theorem~\ref{thm:diag-neutral+} to show that, if $(\alpha, \beta) \neq (0, 0)$, the index of $\cG$ is prime with $3$, giving a contradiction since we have already proved that the index of $\cG$ is $3$. The normalizer of $G$ is $\GG_{m}^{3} \rtimes S$ with $S$ either $C_{3}$ or $S_{3}$, hence it is sufficient to show that $(\alpha, \beta) \neq (0, 0)$ then $X(\GG_{m}^{3}/G)$ contains a permutation $C_{3}$-module of index prime with $3$.

The equation $d^{2} - d + 1 \cong {0}$ has no solutions modulo $9$, hence $n$ is either $1$ or $3$.

If $n = 1$ we may ignore $z$ and assume $\beta = 0$ since $w, x, y$ already generate $G$. The character group $X(\GG_{m}^{3}/G) \subseteq X(\GG_{m}^{3}) = \ZZ^{3}$ is the subgroup of vectors $(r, s, t) \in \ZZ^{3}$ satisfying
\[r + s + t \cong 0 \pmod{c}\]
\[a\alpha(r + s + t) + c(2r - s - t) \cong 0 \pmod{3ac}\]
\[a\alpha(r + s + t) + c(-r + 2s - t) \cong 0 \pmod{3ac}.\]

If $n = 3$, we may choose $d = -1$, hence we have the three equations above plus
\[a\beta(r + s + t) + c(r - s) \cong 0 \pmod{3ac}.\]

In any case, the vector
\[v = \left(\frac{c - a\alpha}{3}, \frac{c - a\alpha + 3a\beta}{3}, \frac{c + 2a\alpha - 3a\beta}{3} \right).\]
is in $X(\GG_{m}^{3}/G)$; it can be checked directly.

Notice that, if $n = 3$ and $\alpha \neq 0$, the other vectors in the $C_{3}$-orbit of $v$ do not satisfy the fourth equation: this implies that $G$ is not normalized by $C_{3}$, a case which we have already ruled out in \S\ref{sect:an-not-1}. Hence, if $n = 3$ we may assume $\alpha = 0$.

Assume by contradiction $(\alpha, \beta) \neq (0, 0)$. Since we are assuming $G$ to be normalized by $C_{3}$, the $C_{3}$-orbit of the vector is a basis of a permutation subgroup $N$ of $X(\GG_{m}^{3}/G)$.

The index of $N$ in $\ZZ^{3}$ is the determinant of the $3 \times 3$ matrix of the basis.

If $n = 1$, $\beta = 0$ and $\alpha \neq 0$, by explicit computation the determinant is $ca^{2}\alpha^{2}$ and $|G| = ca^{2}$, so $N$ has index $\alpha^{2}$ in $X(\GG_{m}^{3}/G)$. Since $3 \nmid \alpha$, by Theorem~\ref{thm:diag-neutral+} $\cG$ has index prime with $3$, which is absurd since $\cG$ has index $3$.

If $n = 3$, $\alpha = 0$ and $\beta \neq 0$, by explicit computation the determinant is $3ca^{2}\beta^{2}$ and $|G| = 3ca^{2}$, so $N$ has index $\beta^{2}$ in $X(\GG_{m}^{3}/G)$. Since $3 \nmid \beta$, by Theorem~\ref{thm:diag-neutral+} $\cG$ has index prime with $3$, which is absurd since $\cG$ has index $3$.

This concludes the proof in the case $G$ is of type $(a)$ or $(e)$.

\subsection{Type $(b)$}

Assume that $\bar{G}$ is of type $(b)$, i.e. up to conjugacy
\[\bar{G} = \left< \diag(\zeta_{2m}, \zeta_{2m}, 1), \diag(\zeta_{2n}, \zeta_{2n}^{-1}, 1) \right>\]
with $m,n$ positive integers. In particular, $G$ is diagonal, and we may consider it as a group scheme with the structure given in Lemma~\ref{lem:diag-hyp}.

\subsubsection{The index of $\cG$ is $2$ and $2 \mid |G|$}

The line $L = \{[x:y:0]\} \subseteq \PP^{2}$ is distinguished, because it is the only line fixed by $2m$ elements of $\bar{G}$, namely the multiples of $\diag(\zeta_{2m}, \zeta_{2m}, 1)$.

Hence, $L$ descends to a subvariety $\bL$ of the coarse moduli space $\bP$ of $\PP(\cV)$. Since $\bP$ is normal, then it is regular in codimension $1$. The fact that $\bL$ is a genus $0$ curve then implies that there exists a point $p \in \bL$, regular in $\bP$, such that $[k(p):k] \mid 2$. By Lemma~\ref{lem:compression-neutral}, $\cG(k(p)) \neq \emptyset$. Since $\cG(k) = \emptyset$ by assumption, then $[k(p):k] = 2$ and $\cG$ has index $2$ by Lemma~\ref{lem:index-neutral}. It follows that $2 \mid |G|$ by Lemma~\ref{lem:index-primes}.

\subsubsection{Structure of $G$ and of the normalizer}

Let $N_{G} \rtimes \gal(K/k_{0})$ be the extended normalizer (see Remark~\ref{rmk:norm-hyp}). If $n = 1$, then $N_{G} \simeq \GL_{2} \times \GG_{m}$. If $n \neq 1$, then $N_{G} \simeq \GG_{m}^{3} \rtimes S$ where $S \subseteq S_{3}$ permutes the coordinates.  We have already seen above that the line $\{[x:y:0]\} \subseteq \PP^{2}$ is distinguished, hence $S$ must fix the third coordinate, i.e. $S$ is either trivial or a cyclic group $C_{2}$ swapping the first two coordinates. Let us show that $S = C_{2}$.

If by contradiction $S$ is trivial, the point $[1:0:0]$ is distinguished and descends to a rational point $\bP$ with singularity of type $\frac{1}{2m'n'}(1, d-1)$, where $m',n',d$ are as in the proof of Lemma~\ref{lem:equiv-rep}. This singularity is of type $\rR$ because $d-1 \cong 0 \pmod{2}$ \cite[Theorem 4]{bresciani-sing}, hence we get that $\cG(k) \neq \emptyset$ thanks to Lemma~\ref{lem:compression-neutral}, which is absurd.

Let $\left< x = \zeta_{c}\Id \right>$ be the kernel of $G \to \bar{G}$. Given the presentation of $\bar{G}$ above, in $G$ we have elements
\[y = \lambda \cdot \diag(\zeta_{2m}, \zeta_{2m}, 1),\]
\[z = \eta \cdot \diag(\zeta_{2n}, \zeta_{2n}^{-1}, 1),\]
and $G = \left< x, y, z \right>$. 

We have that $y^{2m} = \lambda^{2m}\Id \in \left< x = \zeta_{c}\Id \right>$, hence $\lambda = \zeta_{2mc}^{\alpha}$ for some $\alpha$. Furthermore, since $C_{2}$ normalizes $G$, then $\eta \diag(\zeta_{2n}^{-1}, \zeta_{2n}, 1) \in G$, which implies that
\[\eta \diag(\zeta_{2n}, \zeta_{2n}^{-1}, 1) \cdot \eta \diag(\zeta_{2n}^{-1}, \zeta_{2n}, 1) = \eta^{2}\Id \in \left< x = \zeta_{c}\Id \right>,\]
so we may write $\eta = \zeta_{2c}^{\beta}$ for some $\beta$. To sum up,
\[G = \left< x = \zeta_{c}\Id, y = \zeta_{2mc}^{\alpha}\diag(\zeta_{2m}, \zeta_{2m}, 1), z = \zeta_{2c}^{\beta}\diag(\zeta_{2n}, \zeta_{2n}^{-1}, 1)\right>.\]
Up to multiplying $z$ by a suitable power of $x$, we may assume that $0 \le \beta \le 1$.
Up to multiplying $y$ by a suitable power of $x$, we may assume that $0 \le \alpha \le 2m - 1$.

Furthermore, notice that
\[y^{m}z^{-n} = \zeta_{2c}^{\alpha - \beta n} \Id \in \left< x = \zeta_{c}\Id \right>,\]
or equivalently $\alpha - \beta n$ is even.

\subsubsection{Reformulation of statement}

We claim that $G$ has the desired form if $\beta = 0$. Assume that $\beta = 0$, then $\alpha = \alpha - \beta n$ is even as well. Write $c_{1} = -\alpha/2$ and $c_{2} = c - 2c_{1} = c + \alpha$.

The inverse image of $\bar{G}$ in $\SL_{(2,1)}^{(c_{1},c_{2})}$ has order $2mnc = |G|$, so it is enough to check that $G \subseteq \SL_{(2,1)}^{(c_{1},c_{2})}$, or $x, y, z \in \SL_{(2,1)}^{(c_{1},c_{2})}$, which by definition amounts to checking that some products of determinants are $1$. 

For $x$, we have $(\zeta_{c}^{2})^{c_{1}}\zeta_{c}^{c_{2}} = \zeta_{c}^{c} = 1$. For $y$, we have $\zeta_{2mc}^{\alpha c} \cdot \zeta_{m}^{c_{1}} = \zeta_{2m}^{\alpha} \cdot \zeta_{m}^{c_{1}} = 1$. For $z$, we have $\zeta_{2c}^{\beta c} = 1$.

It is then enough to prove that $\beta = 0$.

\subsubsection{Reduction to $2$-groups}

Assume we know that $G$ has the desired form when $|G|$ is a power of $2$, let us prove the general case. Let $G_{2} \subseteq G$ be the subgroup of elements whose order is a power of $2$. Recall that we have fixed a $j$-gerbe $\cV \to \cG$ of index $2$. As in \S\ref{sect:3-group}, we may write $\cG = \cG_{2} \times \cG_{2}'$, where $\cG_{2}'$ is neutral and the restriction of $\cV$ to $\cG_{2}$ is a $G_{2}$-gerbe of index $2$. In particular, $G_{2}$ is not neutral.

Since we are assuming the case of $2$-groups, then $G_{2}$ is the inverse image of $\bar{G}_{2}$ in $\SL_{(2,1)}^{(c_{1},c_{2})}$ for some $c_{1},c_{2}$. The fact that the kernel of $G_{2} \to \bar{G}_{2}$ is a subgroup of $\left< \zeta_{c}\Id \right>$ implies that $2c_{1} + c_{2}$ divides $c$, write $c = c' \cdot (2c_{1} + c_{2})$.

Let $d$ be the largest odd number dividing the order of $z$. Since $z^{d} \in G_{2} \subseteq \SL_{(2,1)}^{(c_{1},c_{2})}$, then $\zeta_{2c}^{d \beta (2c_{1} + c_{2})} = \zeta_{2c'}^{d\beta} = 1$. Hence, the fact that $d$ is odd implies that $\beta$ is even. We conclude that $\beta = 0$ since $\beta$ is either $1$ or $0$.

\subsubsection{Finishing the proof for $2$-groups}

We may then assume that $G$ is a $2$-group, i.e. $m, n, c$ are powers of $2$.

Recall that we have two cases: $n = 1$, with normalizer $\GL_{2} \times \GG_{m}$, or $n \neq 1$, with normalizer $\GG_{m}^{3} \rtimes C_{2}$.

If $n = 1$ and $\beta = 1$, then $\alpha$ is odd and $G = \left< y = \zeta_{2mc}^{\alpha}\diag(\zeta_{2m}, \zeta_{2m}, 1) \right>$. In particular, the third coordinate is distinguished and defines a faithful line bundle, hence by Lemma~\ref{lem:line-neut} $\cG$ is neutral and we get a contradiction.

Assume $n \neq 1$. To conclude, we want to use Theorem~\ref{thm:diag-neutral+} to show that, if $\beta = 1$, then $\cG$ has odd index, which gives a contradiction since we have proved that $\cG$ has index $2$. 

The character group $X(\GG_{m}^{3}/G) \subseteq X(\GG_{m}^{3}) = \ZZ^{3}$ is the group of vectors $(r, s, t) \in \ZZ^{3}$ such that
\[r + s + t \cong 0 \pmod{c},\]
\[\alpha(r + s + t) + c(r + s) \cong 0 \pmod{2mc},\]
\[n(r + s + t) + c(r - s) \cong 0 \pmod{2nc},\]
where these three equations correspond to the three generators $x, y, z$ of $G$ respectively.

Recall that $0 \le \alpha \le 2m - 1$. Let us divide in two cases: either $\alpha = 0$, or $1 \le \alpha \le 2m - 1$.

If $\alpha = 0$, then $n$ is even since we know that $\alpha - n$ is even. The three vectors
\[\left(m + \frac{n}{2}, m - \frac{n}{2}, c - 2m \right), ~ \left(m - \frac{n}{2}, m + \frac{n}{2}, c - 2m \right), ~ \left(m, m, 2(c-m) \right)\]
are in $X(\GG_{m}^{3}/G)$, as can be checked directly, and form a permutation basis of a $C_{2}$-subgroup $N \subseteq X(\GG_{m}^{3}/G)$. The corresponding $3 \times 3$ matrix has determinant $2mnc = |G|$, hence $X(\GG_{m}^{3}/G) = N$ has a permutation basis. By Theorem~\ref{thm:diag-neutral}, $G$ is neutral, giving a contradiction.

Assume $1 \le \alpha \le 2m - 1$, and write $\gamma = m / \gcd(\alpha, m)$. Recall that $\alpha - n$ is even. The three vectors
\[\left( -\frac{n + \alpha}{2}, \frac{n - \alpha}{2}, c + \alpha \right), \left( \frac{n - \alpha}{2}, -\frac{n + \alpha}{2}, c + \alpha \right), \left(0, 0, 2 \gamma c \right)\]
are in $X(\GG_{m}^{3}/G)$, as can be checked directly, and form a permutation basis of a $C_{2}$-subgroup $N \subseteq X(\GG_{m}^{3}/G)$. The corresponding $3 \times 3$ matrix has determinant $2\alpha \gamma n c = |G| \cdot \frac{\alpha}{\gcd(m, \alpha)}$. Since $1 \le \alpha \le 2m - 1$ and $m$ is a power of $2$, then $\alpha/\gcd(m, \alpha)$ is odd. It follows that $N$ has odd index in $X(\GG_{m}^{3}/G)$. By Theorem~\ref{thm:diag-neutral+}, every $j$-gerbe has odd index, which gives a contradiction since $\cG$ has index $2$.

\subsection{Type $(c)$}

Assume that $\bar{G}$ is of type $(c)$, i.e. up to conjugacy $\bar{G} = H_{2}$, that is
\[\bar{G} = \left< M_{0} = \left( \begin{array}{ccc}
						1 & 0 &	0 \\
						0 & \zeta_{3} & 0 \\
						0 & 0 & \zeta_{3}^{2} \\
						\end{array} \right),~
						M_{1} = \left( \begin{array}{ccc}
						0 & 0 & 1 \\
						1 & 0 & 0 \\
						0 & 1 & 0
						\end{array} \right),~
						M_{2} = -\left( \begin{array}{ccc}
						1 & 0 & 0 \\
						0 & 0 & 1 \\
						0 & 1 & 0
						\end{array} \right) \right>.\]
						
Let $\left< \zeta_{c_{0}}\Id \right> \subseteq G$ be the kernel of $G \to \bar{G}$. We may write
\[G = \left< w = \zeta_{c_{0}}\Id, x = \lambda_{0} M_{0}, y = \lambda_{1} M_{1}, z = \lambda_{2} M_{2}\right>\]
for some $\lambda_{i} \in K$.

We have
\[x \cdot y \cdot x^{-1} \cdot y^{-1} = \zeta_{3} \Id \in G,\]
\[x^{3} \cdot z \cdot x^{-1} \cdot z^{-1} \cdot x^{-1} = \lambda_{0}\Id \in G,\]
\[y^{3} \cdot z \cdot y^{-1} \cdot z^{-1} \cdot y^{-1} = \lambda_{1}\Id \in G.\]
Hence, $c_{0} = 3c$ for some integer $c$ and, up to changing $x,y$, we may assume $\lambda_{0} = \lambda_{1} = 1$.

Notice that, if $d$ is any odd number, then $\left< w, x, y, z^{d} \right>$ has image equal to $\bar{G}$ in $\PGL_{3}$, and contains the kernel of $G \to \bar{G}$, hence $\left< w, x, y, z^{d} \right> = G$. This implies that, up to replacing $z$ with $z^{d}$ for some odd number $d$, we may assume $\lambda_{2} = \zeta_{2^{e}}$. Furthermore, we may assume that $2^{e}$ does not divide $c$ if $e \neq 0$: if $2^{e} \mid c$, then $M_{2} \in G$, and up to replacing $z$ with $M_{2}$ we may reduce to the case $e = 0$. However, $2^{e-1} \mid c$ if $e \neq 0$, because $z^{2} = \zeta_{2^{e-1}} \Id \in G$.

To sum up, we have
\[G = \left< w = \zeta_{3c}\Id, x = M_{0}, y = M_{1}, z = \zeta_{2^{e}} M_{2}\right>,\]
with $2^{e} \nmid c$ if $e \neq 0$.

\subsubsection{$e = 0$}

If $e = 0$, then $G$ is the inverse image of $H_{2}$ in $\SL_{3}^{(c)}(K)$. In fact, $G \subseteq \SL_{3}^{(c)}(K)$ can be easily checked, and the inverse image of $H_{2}$ has order $3c \cdot |H_{2}| = |G|$.

\subsubsection{$e \neq 0$}

If $e \neq 0$, we have that $2^{e} \nmid c$. In this case, the intersection $G \cap \SL_{3}(K)$ is the group $\tilde{H}_{1} = \left< M_{0}, M_{1} \right>$. This implies that $G$ is neutral by Proposition~\ref{prop:27} and the following Lemma~\ref{lem:sl-neutral}.

\begin{lemma}\label{lem:sl-neutral}
	Let $G \subseteq \GL_{n}(K)$ be a finite subgroup. If $G_{0} = G \cap \SL_{n}(K)$ is neutral, then $G$ is neutral as well.
\end{lemma}

\begin{proof}
	Let $\cG \to B\GL_{n}$ be a $j$-gerbe, and $\GL_{n} \to \GG_{m}$ the determinant. Let $\cG \to \cH \to B\GG_{m}$ be the canonical factorization of the composition $\cG \to B\GL_{n} \to B\GG_{m}$. The faithful morphism $\cH \to B\GG_{m}$ defines a faithful line bundle on $\cH$, hence $\cH$ is neutral by Lemma~\ref{lem:line-neut}.
	
	Fix $\spec k \to \cH$ any section, and let $\cG_{0} = \cG \times_{\cH} \spec k$ be the fiber product. Then $\cG_{0}$ is a $G_{0}$-gerbe, hence it is neutral by hypothesis. It follows that $\cG$ is neutral as well, since we have a morphism $\cG_{0} \to \cG$.
\end{proof}

\subsection{Type $(d)$}

Assume that $\bar{G}$ is of type $(d)$, i.e. up to conjugacy $\bar{G} = H_{3}$, that is
\[\bar{G} = \left< M_{0} = \left( \begin{array}{ccc}
						1 & 0 &	0 \\
						0 & \zeta_{3} & 0 \\
						0 & 0 & \zeta_{3}^{2} \\
						\end{array} \right),~
						M_{1} = \left( \begin{array}{ccc}
						0 & 0 & 1 \\
						1 & 0 & 0 \\
						0 & 1 & 0
						\end{array} \right),~\right.
						\]
						
						\[
						\left.M_{2} = -\left( \begin{array}{ccc}
						1 & 0 & 0 \\
						0 & 0 & 1 \\
						0 & 1 & 0
						\end{array} \right),~
						M_{3}=\frac{-\zeta_{4}}{\sqrt{3}}\left( \begin{array}{ccc}
						1 & 1 & 1 \\
						1 & \zeta_{3} & \zeta_{3}^{2} \\
						1 & \zeta_{3}^{2} & \zeta_{3}
						\end{array} \right){}~ \right>.\]

If $\zeta_{3} \in k_{0}$, then $G$ is neutral by Proposition~\ref{prop:rhoneut-neut} and Theorem~\ref{thm:pgl3}. Assume $\zeta_{3} \notin k_{0}$, so that $H_{3}$ is not neutral by Theorem~\ref{thm:pgl3}.

Let $\left< \zeta_{c_{0}}\Id \right>$ be the kernel of $G \to \bar{G}$. Notice that $M_{3}^{2} = M_{2}$, so we may remove $M_{2}$ from the generators. We may then write
\[G = \left< w = \zeta_{c_{0}}\Id,~ x = \lambda_{0} M_{0},~ y = \lambda_{1}M_{1},~ z = \lambda_{3}M_{3} \right>\]
for some $\lambda_{i} \in K$. We have equalities
\[xyx^{-1}y^{-1} = \zeta_{3}\Id \in G,\]
\[x^{3}(z^{2}x^{-1}z^{-2})x^{-1} = \lambda_{0}\Id \in G,\]
\[y^{3}(z^{2}y^{-1}z^{-2})y^{-1} = \lambda_{1}\Id \in G,\]
\[z^{4} = \lambda_{3}^{4}\Id \in G,\]
so $c_{0} = 3c$ for some integer $c$ and, up to multiplying the generators by multiples of $w$, we may assume $\lambda_{0} = \lambda_{1} = 1$. Since $M_{3}^{4} = \Id$, we may replace $z$ with $z^{d}$ for some odd $d$ and assume that $\lambda_{3} = \zeta_{2^{e}}$. If $2^{e} \mid c$, we may multiply $z$ by a suitable power of $w$ and reduce to the case $e = 0$, hence we may assume $2^{e} \nmid c$ if $e \neq 0$. However, notice that $2^{e-2} \mid c$ if $e \ge 2$, because $\lambda_{3}^{4}\Id \in G$.

\[G = \left< w = \zeta_{3c}\Id,~ x = M_{0},~ y = M_{1},~ z = \zeta_{2^{e}}M_{3} \right>.\]

\subsubsection{$e = 0$}

If $e = 0$, then $G$ is the inverse image in $\SL_{3}^{(c)}$ of $\bar{G}$, it is saturated and not neutral by Proposition~\ref{prop:saturated-neut}.

\subsubsection{$e \neq 0$, $2^{e-1} \nmid c$}

If $e \neq 0$ and $2^{e-1} \nmid c$, then $G \cap \SL_{3} = \left< M_{0}, M_{1} \right> = \tilde{H}_{1}$ which is neutral by Proposition~\ref{prop:27}. We conclude that $G$ is neutral by Lemma~\ref{lem:sl-neutral}.

\subsubsection{$e \neq 0$, $2^{e-1} \mid c$}

Assume $e \neq 0$ and $2^{e-1} \mid c$, or equivalently $e \neq 0$ and $M_{2} = z^{2}w^{-3c/ 2^{e-1}} \in G$.

In this case, $G$ is Galois invariant. The Galois conjugates of $w$ are clearly in $G$. The same is true for $x$ and $y$, because the only non-trivial Galois automorphism of $k_{0}(\zeta_{3})/k_{0}$ maps $M_{i}$ either to itself or to $M_{i}^{-1}$ for every $i = 0, 1, 2, 3$. A Galois conjugate of $z$ has the form $\zeta_{2^{e}}^{d}M_{3}^{\pm1}$, with $d$ odd. If we multiply it by $z^{\pm 1}$, we get $\zeta_{2^{e-1}}^{(d \pm 1)/2}\Id$, which is in $G$ because $2^{e-1} \mid c$. 

In particular, $G = \fG(K)$ for a group scheme $\fG$ over $k_{0}$. Its image in $\PGL_{3}$ is $\fH_{3}$ and the normalizer of $\fH_{3}$ is $\fH_{4}$, where $\fH_{i}$ is the group scheme whose geometric points are the elements of $H_{i}$. Let $\Pi$ be the inverse image of $\fH_{4}$ in $\SL_{3}^{(2c)}$, that is $\Pi = \left< \zeta_{6c}\Id,~ M_{0},~ M_{1},~ M_{3},~ M_{4}\right>$. Notice that $\fG \subset \Pi$ because $2^{e} \mid 2c$.

Notice that $M_{4}$ normalizes $H_{1}$, hence it normalizes the inverse image of $H_{1}$ in $\SL_{3}^{(c)}$, which is $\left< w, x, y\right>$. Furthermore,
\[M_{4}zM_{4}^{-1} = M_{4}(\zeta_{2^{e}}M_{3})M_{4}^{-1} = M_{1}^{-1}(\zeta_{2^{e}}M_{3}^{-1})M_{1} = \]
\[= w^{3c \cdot 2^{1-e}} \cdot M_{1}^{-1}(\zeta_{2^{e}}M_{3})^{-1}M_{1} = w^{3c \cdot 2^{1-e}} \cdot M_{1}^{-1}z^{-1}M_{1},\]
hence $M_{4}$ normalizes $G$. Since $M_{0}, M_{1}, M_{3}$ clearly normalize $G$, we get that $\Pi$ normalizes $\fG$. If $\fN$ is the normalizer of $\fG$, then $\fN / \Pi \simeq \GG_{m}$, hence $\Pi$ is a gate for $\fG$ by Proposition~\ref{prop:gate-norm} and Lemma~\ref{lem:gate-special}.

Since $|G| = 2^{2} 3^{3} c$ and $|\Pi| = 6c |H_{4}| = 2^{4} 3^{3} c$, then $\fG$ has index $4$ in $\Pi$. The quotient $\Pi/\fG$ is generated by $M_{3}$ and $M_{4}$, and is a twisted form of $C_{2} \times C_{2}$. The base change to $k_{0}(\zeta_{3})$ of $\Pi/\fG$ is constant, because $M_{4}$, $M_{3}$ are defined over $k_{0}(\zeta_{3})$. The action of $\gal(k_{0}(\zeta_{3})/k_{0})$ maps $M_{3}$ to $M_{3}^{-1}$, and $M_{4}$ to $\zeta_{3}^{2}M_{3}M_{4}$, hence the Galois action on $\Pi/\fG(k_{0}(\zeta_{3})) \simeq C_{2} \times C_{2}$ fixes $(1, 0)$ and maps $(0,1)$ to $(1, 1)$. This implies that $\Pi/\fG$ is isomorphic to the twisted Klein group $\fKl \simeq \fH_{4}/\fH_{2}$ which we studied in \S\ref{sect:type-d}. 

We want to prove that $G$ is neutral. By Theorem~\ref{thm:coh-crit}, this is equivalent to showing that $\H^{1}(k, \Pi) \to \H^{1}(k, \fKl)$ is surjective for every extension $k$ of $k_{0}$. 

Let $c' = c/2^{e-1}$, it is odd because $2^{e} \nmid c$. In $\Pi$, we have a normal subgroup scheme $\fM = \left< \zeta_{3c'}\Id,~ M_{0},~ M_{1} \right>$ of odd order $3^{3}c'$ and index $2^{4+e-1} = 2^{e+3}$. Furthermore, we have a $2$-Sylow subgroup $\mu_{2^{e}} \times \fQ_{8} = \left< \zeta_{2^{e}}\Id,~ M_{3},~ \zeta_{3}M_{1}M_{4} \right>$ of order $2^{e+3}$, hence $\Pi/\fM \simeq \mu_{2^{e}} \times \fQ_{8}$. We get a factorization
\[\mu_{2^{e}} \times \fQ_{8} \subset \Pi \to \mu_{2^{e}} \times \fQ_{8} \to \fKl.\]

To show that $\H^{1}(k, \Pi) \to \H^{1}(k, \fKl)$ is surjective, it is then sufficient to show that $\H^{1}(k, \mu_{2^{e}} \times \fQ_{8}) \to \H^{1}(k, \fKl)$ is surjective. This is Corollary~\ref{cor:klq8}, plus the fact that $\H^{1}(k, \mu_{2^{e}}) \to \H^{1}(k, C_{2})$ is surjective.

\appendix
\section{The normalizer of a faithful morphism of gerbes}\label{sect:normalizer}

Let $\cC$ be a site, fix $\cH$ a gerbe on $\cC$. We want to study faithful morphisms $\cG \to \cH$, where $\cG$ is another gerbe. We think of $\cH$ as a fibered category $\cH \to \cC$, so that objects of $\cH$ correspond to pairs $(S,s)$ where $S \in \cC$ and $s: S \to \cH$ is a section.

\subsection{Locally equivalent morphisms}

\begin{definition}\label{def:loc-eq}
	Two faithful morphisms $f: \cG \to \cH$, $f': \cG' \to \cH$ with $\cG, \cG'$ gerbes are \emph{locally equivalent} if, for every object $S \in \cC$ and every $2$-commutative diagram
	\[\begin{tikzcd}
		S \rar["s"] \dar["s'"]		&	\cG \dar["f"]	\\
		\cG'\rar["f'"]				&	\cH
	\end{tikzcd}\]
	with $2$-isomorphism $\alpha: f \circ s \Rightarrow f' \circ s'$, the group sheaves $\underaut_{\cG}(s)$ and $\alpha^{-1}\underaut_{\cG'}(s')\alpha$ are locally conjugate in $\underaut_{\cH}(f \circ s)$.
\end{definition}

The following lemma justifies the name ``locally equivalent''.

\begin{lemma}\label{lem:easy-loc-eq}
	Assume that $\cC$ has a terminal object $T$. Two faithful morphisms of gerbes $\cG \to \cH$, $\cG' \to \cH$ are locally equivalent if and only if there exists a covering $\{U_{i} \to T\}$ such that, for every $i$, there exists an isomorphism $\cG|_{U_{i}} \simeq \cG'|_{U_{i}}$ making the obvious diagram $2$-commutative.
\end{lemma}

\begin{proof}
	Assume that the morphisms are locally equivalent. By definition of gerbe, there exists a covering $\{U_{i} \to T\}$ such that $\cG(U_{i})$ and $\cG'(U_{i})$ are non-empty. Since $\cH$ is a gerbe, up to refining the covering we may assume that for every $i$ there are morphisms $u:U_{i} \to \cG$, $u':U_{i} \to \cG'$ and  $\alpha: f \circ u \Rightarrow f' \circ u'$ making the diagram of Definition~\ref{def:loc-eq} $2$-commutative. Hence, $\alpha\underaut_{\cG}(u)\alpha^{-1}$ and $\underaut_{\cG'}(u')$ are locally conjugate in $\underaut_{\cH}(f' \circ u')$. After further refinement of the covering, this gives us the desired isomorphism $\cG|_{U_{i}} = B\underaut_{\cG}(u) \simeq B\underaut_{\cG'}(u') \simeq \cG'|_{U_{i}}$ which $2$-commutes with the morphisms .
	
	On the other hand, assume that there exists a covering $\{U_{i} \to T\}$ with isomorphisms $\cG|_{U_{i}} \simeq \cG'|_{U_{i}}$ as above. Assume we have a $2$-commutative diagram as in Definition~\ref{def:loc-eq} for some object $S \in \cC$, with $\alpha: f \circ s \Rightarrow f' \circ s'$. We want to check that the group sheaves $\alpha\underaut_{\cG}(s)\alpha^{-1}$ and $\underaut_{\cG'}(s')$ are locally conjugate in $\underaut_{\cH}(f' \circ s')$, and we may do so after replacing $T$ with $U_{i}$. If $T = U_{i}$ this is obvious, because we have an isomorphism $\cG \simeq \cG'$ which $2$-commutes with the morphisms to $\cH$.
\end{proof}

\subsection{Types}

We want now to make the notion of local equivalence more flexible, in order to compare gerbes defined on different sub-sites. We are going to define the \emph{type} of a faithful morphism $\cG \to \cH$. 

\begin{definition}\label{def:sample}
	A \emph{sample} of $\cC$ is a fibered category $\cS$ such that, for every $S \in \cC$, there exists a cover $\{U_{i} \to S\}$ with $\cS(U_{i}) \neq \emptyset$ for every $i$.
\end{definition}

\begin{remark}\label{rmk:samples-over-fields}
	If $\cC$ is the fppf site over a field, a sample is just a fibered category $\cS$ such that $\cS(k') \neq \emptyset$ for some finite extension $k'/k$.
	
	In fact, if $\cS(k') \neq \emptyset$, then $\cS(S_{k'}) \neq \emptyset$ for every scheme $S$ over $k$. On the other hand, if $\cS(S) \neq \emptyset$ for some scheme $S$ locally of finite presentation over $k$, then $S(k') \neq \emptyset$ for some finite extension $k'/k$, hence $\cS(k') \neq \emptyset$ as well.
\end{remark}

Recall that the inertia stack $\rI_{\cH}$ is defined by the cartesian diagram
\[\begin{tikzcd}
	\rI_{\cH} \rar \dar		&	\cH \dar	\\
	\cH \rar				&	\cH \times \cH,
\end{tikzcd}\]
and it is a relative group sheaf over $\cH$. Equivalently, the inertia is the stack of pairs $(h, \alpha)$ where $h \in \cH(S)$ and $\alpha$ is an automorphism of $h$.

Let $\rP_{\cH} \to \cH$ be the fibered category of parts of the inertia, i.e. an object of $\rP_{\cH}(S)$ is a morphism $s: S \to \cH$ together with a subsheaf of sets of $\underaut_{\cH}(s)$. 

\begin{definition}\label{def:type}
	A \emph{pre-type} is a functor $\phi: \cS \to \rP_{\cH}$ from a sample $\cS$, with composition $f:\cS \to \rP_{\cH} \to \cH$, such that either of the following equivalent conditions holds.
	
	\begin{itemize}
		\item If $s,s' \in \cS(S)$ are objects and $\alpha: f(s) \Rightarrow f(s')$ is an isomorphism in $\cH$, the subsheaf $\alpha^{-1} \phi(s') \alpha$ is locally conjugate to $\phi(s)$ in $\underaut_{\cH}(s)$. 
		\item If $s,s' \in \cS(S)$ are objects, there exists a cover $\{U_{i} \to S\}$ and isomorphisms $\alpha_{i}:f(s_{U_{i}}) \to f(s'_{U_{i}})$ such that $\alpha_{i}\phi(s_{U_{i}})\alpha_{i}^{-1} = \phi(s'_{U_{i}})$.
	\end{itemize}
	
	Two pre-types $\phi : \cS \to \rP_{\cH}$, $\phi' : \cS' \to \rP_{\cH}$ are equivalent if 
	\[\phi \sqcup \phi': \cS \sqcup \cS' \to \rP_{\cH}\]
	is a pre-type as well.
	
	A \emph{type} is an equivalence class of pre-types.
\end{definition}

\begin{definition}
	Let $\Theta$ be a type, and $X$ a fibered category over $\cC$.
	
	A functor $\psi : X \to \rP_{\cH}$ is of type $\Theta$ if $\psi \sqcup \phi : X \sqcup \cS \to \rP_{\cH}$ is a pre-type in $\Theta$, where $\phi: \cS \to \rP_{\cH}$ is any pre-type in $\Theta$.
\end{definition}

Given a faithful morphism $f:\cG \to \cH$ with $\cG$ a gerbe, we have a natural functor $\underaut_{\cG}:\cG \to \rP_{\cH}$, $s \mapsto \underaut_{\cG}(s) \subseteq \underaut_{\cH}(f(s))$.

\begin{lemma}\label{lem:gerbe-type}
	The pair $\underaut_{\cG} : \cG \to \rP_{\cH}$ is a pre-type.
\end{lemma}

\begin{proof}	
	Clearly, $\cG$ is a sample, let us check that $\underaut_{\cG}$ satisfies the second condition of Definition~\ref{def:type}. Consider $s, s' \in \cG(S)$ objects. Up to passing to a cover of $S$, we may assume that there is an isomorphism $\alpha: s \Rightarrow s'$ in $\cG$. Then $f_{*}\alpha \underaut_{\cG}(s)f_{*}\alpha^{-1} = \underaut_{\cG}(s')$ in $\underaut_{\cH}(f(s'))$.
\end{proof}

\begin{lemma}\label{lem:equiv-type}
	Two faithful morphisms $\cG \to \cH$, $\cG' \to \cH$ are locally equivalent if and only if $\underaut_{\cG}$, $\underaut_{\cG'}$ define the same type.
\end{lemma}

\begin{proof}
	This follows directly from definitions.
\end{proof}

\begin{definition}\label{def:covering-sample}
	Assume that $\cC$ has a terminal object $T$, and let $E = \{T_{i} \to T\}_{i}$ be a covering. The covering $E$ defines an associated sample $\cS_{E}$ with a morphism $\cS_{E} \to B\Gamma$ as follows. An object of $\cS_{E}(S)$ is a pair $(i, S \to T_{i})$. 
	
	Consider two objects $S' \to T_{i}$, $S \to T_{j}$. If $i \neq j$, there are no morphisms in $\cS_{E}$ between them. If $i = j$, a morphism in $\cS_{E}$ is a morphism $S' \to S$ in $\cC$ which commutes with the projection to $T_{i}$.
	
	The morphism $\cS_{E} \to B\Gamma$ simply maps every element to the tautological section of $B\Gamma$, and every morphism to the identity.
\end{definition}

\begin{lemma}\label{lem:loc-type}
	Assume that $\cC$ has a terminal object, and let $E = \{T_{i} \to T\}_{i}$ be a covering with associated sample $\cS_{E}$.
	
	Let $\Gamma$ be a group sheaf on $\cC$. Assume that, for every $i$, we have a subsheaf $\Phi_{i} \subseteq \Gamma|_{T_{i}}$ such that, for every pair of morphisms $S \to T_{i}$, $S \to T_{j}$, the subsheaves $\Phi_{i}|_{S}$ and $\Phi_{j}|_{S}$ are locally conjugate in $\Gamma|_{S}$. The functor $\phi: \cS_{E} \to \rP_{\cH}$ sending $(i, S \to T_{i})$ to $\Phi_{i}|_{S} \subset \Gamma|_{S}$ is a pre-type.
\end{lemma}

\begin{proof}
	Let us check the first condition of Definition~\ref{def:type}. Assume we have objects $S \to T_{i}$, $S \to T_{j}$ of $\cS_{E}(S)$ and $\alpha$ an isomorphism between their images in $B\Gamma$. This simply means that $\alpha \in \Gamma(S)$, because the composition $\cS_{E} \to B\Gamma$ maps everything to the tautological section. By hypothesis, $\Phi_{i}|_{S}$ and $\Phi_{j}|_{S}$ are locally conjugate, hence $\alpha \Phi_{i}\alpha^{-1}$, $\Phi_{j}$ are locally conjugate as well.
\end{proof}

A typical situation where we may apply Lemma~\ref{lem:loc-type} is the following. Let $\cC$ be the fppf site on a field $k$, $\Gamma$ an algebraic group and $k'/k$ a finite Galois extension with a subset $M \subseteq \Gamma(k')$ such that, if $M'$ is a Galois conjugates of $M$, up to extending $k'$ there exists $\gamma \in \Gamma(k')$ with $\gamma M \gamma^{-1} = M'$. Then $M$ defines a type for $B\Gamma$.

\subsection{The normalizer of a type}

\begin{theorem}\label{thm:type-normalizer}
	Let $\cC$ be a small site, $\cH$ a gerbe on $\cC$ and $\Theta$ a type for $\cH$. There exists a gerbe $\cN$ on $\cC$ with a morphism $\nu:\cN \to \cH$, called the \emph{normalizer of $\Theta$}, with the following properties.
	
	\begin{enumerate}
		\item	$\nu$ is faithful.
		\item	For every functor $\psi : X \to \rP_{\cH}$, $f: X \to \cH$ of type $\Theta$, there exists a factorization
					\[X \xrightarrow{\mu} \cN \xrightarrow{\nu} \cH\]
				of $f$ with the following property.
				
				For every $x,y \in X(S)$, the sheaf $\underisom_{\cN}(\mu(x), \mu(y))$ is the subsheaf of isomorphisms $\alpha \in \underisom_{\cH}(f(x),f(y))$ such that $\alpha^{-1} \psi(y) \alpha = \psi(x) \subset \underaut_{\cH}(f(x))$.
				
				In particular, $\underaut_{\cN}(\mu(x))$ is the normalizer of $\underaut_{\cG}(x)$ in $\underaut_{\cH}(f(x))$.
				
		\item	The normalizer is unique in the following sense. If $\cN'$ is a gerbe with a faithful morphism $\cN' \to \cH$ and $\cS \to \rP_{\cH}$ is a pre-type in $\Theta$ with a factorization $\cS \to \cN' \to \cH$ as in $(2)$, there exists an isomorphism $\cN \to \cN'$ making the diagram
	\[\begin{tikzcd}[row sep = tiny]
								&	\cN \ar[dr] \ar[dd]		&				\\
		\cS \ar[ur]\ar[dr]		&							&	\cH		\\
								&	\cN' \ar[ur]
	\end{tikzcd}\]
	$2$-commutative.
	\end{enumerate}
\end{theorem}

\begin{proof}
	Fix $\phi : \cS \to \rP_{\cH}$ a pre-type in $\Theta$.
	
	Define a fibered category $\cN_{0}$ as follows. For every $S \in \cC$, the objects of $\cN_{0}(S)$ are the objects of $\cS(S)$. For every $x, y \in \cN_{0}(S)$, a morphism $x \Rightarrow y$ in $\cN_{0}(S)$ is a morphism $\alpha : f(x) \Rightarrow f(y)$ in $\cH(S)$ such that $\alpha^{-1} \phi(y) \alpha = \psi(x)$. Notice that $\cN_{0}$ is a prestack, because the $\hom$ functors are sheaves.
	
	Since $\cC$ is small, we may define $\cN$ as the stackification of $\cN_{0}$ \cite[\href{https://stacks.math.columbia.edu/tag/02ZN}{Tag 02ZN}]{stacks-project}. Since $\cH$ is a stack we have an induced morphism
	\[\cN \to \cH\]
	which is faithful by construction. Notice that property $(2)$ is satisfied for $X = \cS \to \cH$. Let us show that $\cN$ is a gerbe.
	
	First, notice that for every $S \in \cC$, there exists a covering $\{U_{i} \to S\}$ such that $\cN(U_{i}) \neq \emptyset$, because the same holds for $\cS$.
		
	By definition of stackification, every object of $\cN$ lifts locally to $\cS$, so to prove that $\cN$ is a gerbe it is enough to show that any two objects $x, y \in \cS(S)$ are locally isomorphic in $\cN$. Up to passing to a covering of $S$, we may assume that there is an isomorphism $\alpha: f(x) \Rightarrow f(y)$. By definition of pre-type, up to passing to a cover of $S$ there exists $\beta \in \underaut_{\cH}(f(x))$ such that $\beta^{-1}\alpha^{-1}\phi(y)\alpha\beta = \phi(x) \subset \underaut_{\cH}(f(x))$. This means that $\alpha\beta$ is a local isomorphism $x \Rightarrow y$, hence $\cN$ is a gerbe.
	
	Let $\cS \to \cN' \to \cH$ as in $(2)$, with $\cN' \to \cH$ faithful. We can define a morphism $\cN_{0} \to \cN'$ as follows. Since the underlying sets of $\cS(S)$, $\cN_{0}(S)$ coincide, we just use $\cS \to \cN'$ to define $\cN_{0} \to \cN'$ on objects. By definition, a morphism $\alpha : x \Rightarrow y$ in $\cN_{0}(S)$ is a morphism $\alpha : f(x) \Rightarrow f(y)$ in $\cH(S)$ such that $\alpha^{-1} \phi(y) \alpha = \phi(x) \subset \underaut_{\cH}(f(x))$. Since $\cN'$ satisfies $(2)$, we have that $\alpha$ lifts to a morphism in $\cN'$.
	
	Since both $\cN$ and $\cN'$ satisfy property $(2)$ with respect to the given morphisms $\cS \to \cN \to \cH$, $\cS \to \cN' \to \cH$, we get that $\cN \to \cN'$ is a morphism of gerbes inducing an isomorphism on automorphism groups, hence it is an equivalence.
	
	Hence, we have all three properties with respect to the given pre-type $\cS \to \rP_{\cH}$. Let $X$ be a fibered category with a functor $\psi: X \to \rP_{\cH}$ of type $\Theta$. Since $\psi \sqcup \phi : X \sqcup \cS \to \rP_{\cH}$ is a pre-type in $\Theta$, we may repeat all the above with respect to this new pre-type and obtain a gerbe $\cN''$ with a morphism $\cN'' \to \cH$ which respects the three properties for the pre-type $\psi \sqcup \phi$. In particular, $\cN'' \to \cH$ respects property $(3)$ for the pre-type $\phi: \cS \to \rP_{\cH}$ as well, hence we get an equivalence $\cN \simeq \cN''$. 
	
	It follows that $\cN \to \cH$ respects property $(2)$ for $X \to \cN \to \cH$ as well, and if $X$ is another sample we get property $(3)$ too. 
\end{proof}

As a consequence, we obtain the following.

\begin{theorem}\label{thm:gerbe-normalizer}
	Let $f : \cG \to \cH$ be a faithful morphism of gerbes on a small site $\cC$. There exists a gerbe $\cN$ and a factorization
	\[\cG \xrightarrow{\phi} \cN \xrightarrow{\psi} \cH\]
	of $f$ called the \emph{normalizer} of $f$, having the following properties.
	
	\begin{enumerate}
		\item	$\phi$, $\psi$ are faithful.
		\item	For every $S \in \cC$ and every object $s: S \to \cG$, the group sheaf $\underaut_{\cN}(\phi(s))$ is the normalizer of $\underaut_{\cG}(s)$ in $\underaut_{\cH}(\psi \circ \phi (s))$.
		\item	The normalizer is unique in the following sense. If $\cG \to \cN' \to \cH$ is another factorization having properties $(1)$ and $(2)$, there exists an isomorphism $\cN \to \cN'$ making the diagram
	\[\begin{tikzcd}[row sep = tiny]
								&	\cN \ar[dr] \ar[dd]		&				\\
		\cG \ar[ur]\ar[dr]		&							&	\cH		\\
								&	\cN' \ar[ur]
	\end{tikzcd}\]
	$2$-commutative.
		\item  If $\cG \to \cH$, $\cG' \to \cH$ are locally equivalent, the normalizers of $\cG \to \cH$ and $\cG' \to \cH$ coincide. More precisely, if $\cG \to \cN \to \cH$ is the normalizer of $\cG \to \cH$, there exists a faithful morphism $\cG' \to \cN$ such that the composition $\cG' \to \cN \to \cH$ is the normalizer of $\cG' \to \cH$.
	\end{enumerate}
\end{theorem}

\begin{proof}
	By Lemma~\ref{lem:gerbe-type}, $\cG$ defines a type, hence by Theorem~\ref{thm:type-normalizer} we get a factorization $\cG \to \cN \to \cH$. Let us show that it enjoys properties \ref{thm:gerbe-normalizer}.(1)-(4).
	
	\ref{thm:gerbe-normalizer}.(1) and \ref{thm:gerbe-normalizer}.(2) are direct consequences of \ref{thm:type-normalizer}.(1) and \ref{thm:type-normalizer}.(2).
	
	The uniqueness statement \ref{thm:gerbe-normalizer}.(3) is more subtle, because \ref{thm:gerbe-normalizer}.(2) is a priori weaker than \ref{thm:type-normalizer}.(2). Let us show that, if $\cG \to \cN' \to \cH$ satisfies \ref{thm:gerbe-normalizer}.(1)-(2), then it satisfies \ref{thm:type-normalizer}.(2) as well.
	
	Let $x,y \in \cG(S)$ and $\alpha:f'(x) \Rightarrow f'(y)$ be an isomorphism in $\cH$. We want to show that $\alpha^{-1}\underaut_{\cG}(y)\alpha = \underaut_{\cG}(x) \subset \underaut_{\cH}(f'(x))$ if and only if $\alpha$ lifts to $\cN'$. This is a local statement, hence up to passing to a cover of $S$ we might assume the existence of an isomorphism $\beta:x \Rightarrow y$. In particular, $\underaut_{\cG}(y) = \beta\underaut_{\cG}(x)\beta^{-1}$. Hence, $\alpha^{-1}\underaut_{\cG}(y)\alpha = \underaut_{\cG}(x)$ if and only if $\beta^{-1}\alpha$ normalizes $\underaut_{\cG}(x)$, if and only if $\beta^{-1}\alpha$ lifts to $\cN'$ (by \ref{thm:gerbe-normalizer}.(2)), if and only if $\alpha$ lifts to $\cN'$.
	
	This proves that $\cG \to \cN' \to \cH$ satisfies \ref{thm:type-normalizer}.(2). We may then apply \ref{thm:type-normalizer}.(3), and get the desired isomorphism $\cN \to \cN'$. 
	
	Finally, \ref{thm:gerbe-normalizer}.(4) is a direct consequence of Lemma~\ref{lem:equiv-type}.
\end{proof}

\begin{remark}
	There are at least two ways in which Theorem~\ref{thm:type-normalizer} generalizes Theorem~\ref{thm:gerbe-normalizer}.
	
	The first is obvious: we are not restricted to normalizers of subgroups, but we are free to work with normalizers of arbitrary subsets. For instance, the normalizer of a single element of a group is the centralizer of the element.
	
	The second is more subtle. Assume that $\cC$ is the fppf site on a field, and that $\cH$ is the classifying stack $B \Gamma$ of an algebraic group $\Gamma$. Consider a finite Galois extension $k'/k$ and a finite subgroup $G \subset \Gamma(k')$. If $G$ is Galois invariant, then it coincides with the geometric point of a group scheme $\fG$, and if $\fN \subseteq \Gamma$ is the normalizer of $\fG$ then $B\fN$ is the normalizer of $B\fG \to B\Gamma$.
	
	However, to define the normalizer at the level of gerbes, it is sufficient that the Galois conjugates of $G$ are conjugate to $G$ as subgroups of $\Gamma(k')$, up to enlarging $k'$: we may then use Lemma~\ref{lem:loc-type} to define a type and Theorem~\ref{thm:type-normalizer} to define the normalizer.
\end{remark}

\begin{example}\label{example:C2S3}
	Consider the symmetric group $S_{3}$ and a cyclic subgroup $C_{2} \subset S_{3}$ of order $2$. Fix a base field $k$. Since $C_{2}$ is its own normalizer in $S_{3}$, the normalizer of $B C_{2} \to B S_{3}$ is again $B C_{2}$. Theorem~\ref{thm:gerbe-normalizer} then implies that there are no non-trivial morphisms locally equivalent to $B C_{2} \to B S_{3}$.
	
	This can also be checked more directly. Assume that $\cG \to B S_{3}$ is locally equivalent to $BC_{2} \to BS_{3}$. The tautological section $\spec k \to BS_{3}$ is a representable, degree $6$ cover, so its pullback $E \to \cG$ is a representable, degree $6$ cover as well. In particular, $E$ is a scheme étale over $k$ of degree $3$. The gerbe $\cG$ corresponds to a cohomology class $\H^{2}(k, C_{2})$ which is split by $E$; in particular, its order divides $3$. Since $\H^{2}(k, C_{2})$ is $2$-torsion, this implies that the cohomology class is trivial, i.e. $\cG$ is neutral.
\end{example}

\begin{example}\label{example:normalizers-over-fields}
	Let $\Gamma$ be an algebraic group over a perfect field $k$, and $R \subseteq \Gamma(\bar{k})$ a sub-variety such that the $\gal(\bar{k}/k)$-conjugates of $R$ are $\Gamma(\bar{k})$-conjugates of $R$ as well.
	
	Then by Lemma~\ref{lem:loc-type} and Theorem~\ref{thm:type-normalizer} there exists a gerbe $\cN$ over $k$ with a faithful morphism $\cN \to B\Gamma$ associated with the normalizer of $R$, in the following sense. If $X$ is a fibered category over $k$ with a morphism $\phi:X \to \rP_{B\Gamma}$ such that $\phi(x)$ is conjugate to $R$ for every $x \in X$, and $X(k') \neq \emptyset$ for some finite extension $k'/k$, then $X \to B\Gamma$ factorizes through $\cN$.
	
	If $R$ is a subgroup of $\Gamma(\bar{k})$, then $\cN$ corresponds to the normalizer of $R$, which is the most interesting case for the purposes of the paper. However, the construction is completely general. For instance, if $R =\{\gamma \} \subset \Gamma(\bar{k})$, then $\cN$ corresponds to the centralizer of $\gamma$.
\end{example}

\bibliographystyle{amsalpha}
\bibliography{main}

\end{document}